\documentclass[11pt]{article}
\usepackage{amsthm}

\usepackage{amssymb,amsmath,amsfonts,graphicx,subfigure,color,multicol}
\usepackage[hmargin=1.0in]{geometry}
\usepackage{amssymb}
\usepackage{bbm}
\usepackage{dsfont}
\usepackage{epstopdf}
\usepackage{algpseudocode}
\usepackage{algorithm}
\usepackage{hyperref}

\newtheorem{thm}{Theorem}[section]
\newtheorem{lem}[thm]{Lemma}
\newtheorem{cor}[thm]{Corollary}

\newtheorem{dfn}[thm]{Definition}
\newtheorem{rem}[thm]{Remark}

\newcommand{\M}{{\mathcal{M}}}
\newcommand{\Ss}{{\mathcal{S}}}
\newcommand{\vr}{\varrho}
\newcommand{\vf}{\varphi}

\begin{document}

\newcommand{\Complex}{\mathbb{C}}
\newcommand{\Real}{\mathbb{R}}

\newcommand{\bi}{\textbf{I}}

\newcommand{\bz}{\textbf{z}}

\newcommand{\dt}{\tau}

\newcommand{\dx}{h}

\newcommand{\order}{{\mathcal O}}

\newcommand{\y}{Y}
\newcommand{\f}{F}
\newcommand{\bfun}{\textbf{F}}
\newcommand{\uu}{U}
\newcommand{\un}{U_n}
\newcommand{\unone}{U_{n+1}}

\newcommand{\bA}{\textbf{A}}
\newcommand{\bb}{\textbf{b}}
\newcommand{\bc}{\textbf{c}}

\newcommand{\bL}{\textbf{L}}

\newcommand{\by}{\textbf{Y}}
\newcommand{\bv}{\textbf{v}}
\newcommand{\br}{\textbf{r}}
\newcommand{\balpha}{\boldsymbol{\alpha}}
\newcommand{\bbeta}{\boldsymbol{\beta}}
\newcommand{\btheta}{\boldsymbol{\theta}}
\newcommand{\bdelta}{\boldsymbol{\delta}}
\newcommand{\bsv}{\textbf{v}^*}
\newcommand{\bsalpha}{\boldsymbol{\alpha}^*}
\newcommand{\bsbeta}{\boldsymbol{\beta}^*}
\newcommand{\bone}{\mathbbm{1}}

\newcommand{\ty}{\widetilde{Y}}
\newcommand{\bty}{\widetilde{\textbf{Y}}}
\newcommand{\tu}{\widetilde{U}}
\newcommand{\tun}{\widetilde{U}_n}
\newcommand{\tunp}{\widetilde{U}_{n+1}}
\newcommand{\tr}{\tilde{r}}
\newcommand{\btr}{\tilde{\textbf{r}}}
\newcommand{\bd}{\textbf{d}}
\newcommand{\berr}{\epsilon}
\newcommand{\bQ}{\textbf{Q}}
\newcommand{\hy}{\widehat{Y}} 
\newcommand{\hun}{\widehat{U}_n}
\newcommand{\hunp}{\widehat{U}_{n+1}}

\newcommand{\alphatop}{\alpha_{1:s}}
\newcommand{\alphabottom}{\alpha_{s+1}}
\newcommand{\betatop}{\beta_{1:s}}
\newcommand{\betabottom}{\beta_{s+1}}
\newcommand{\vvtop}{v_{1:s}}
\newcommand{\vvbottom}{v_{s+1}}

\newcommand{\abs}[1]{\vert #1 \vert}
\newcommand{\maxnorm}[1]{\vert \vert #1 \vert \vert_{\infty}}

\newcommand{\epsmach}{\epsilon_\textup{machine}}

\newcommand{\sspcoeff}{\mathcal{C}}

\newcommand{\MEEp}{\M^{\mathrm{EE}}_p} 
\newcommand{\MEMp}{\M^{\mathrm{EM}}_p} 

\newcommand{\Cleft}{\Complex_{-,0}}

\title{Propagation of internal errors in explicit Runge--Kutta methods 
and internal stability of SSP and extrapolation methods}
\author{David I. Ketcheson\thanks{4700 King Abdullah University of Science and
Technology, Thuwal, 23955-6900, Saudi Arabia.  This publication is based on work
supported by Award No. FIC/2010/05 { 2000000231, made by KAUST}.} \and Lajos
L\'oczi\thanks{King Abdullah University of Science and Technology.} \and Matteo
Parsani\thanks{Computational Aerosciences Branch, NASA Langley Research Center, 
Hampton, VA 23681, USA. This author was also supported in part by an 
appointment to the NASA Postdoctoral Program at Langley Research Center, 
administered by Oak Ridge Associates Universities.} }

\maketitle

\begin{abstract}
In practical computation with Runge--Kutta methods, the stage equations are
not satisfied exactly, due to roundoff errors, algebraic solver errors, and so forth.
We show by example that propagation of such errors within a single step
can have catastrophic effects for otherwise practical and well-known methods.
We perform a general analysis of internal error propagation,
emphasizing that it depends significantly on how the method is implemented.
We show that for a fixed method, essentially any set of
internal stability polynomials can be obtained by modifying the implementation details.
We provide bounds on the internal error amplification 
constants for some classes of methods with many stages, including strong 
stability preserving methods and extrapolation methods.  These results are used
to prove error bounds in the presence of roundoff or other internal errors.
\end{abstract}

\section{Error propagation in Runge--Kutta methods}
Runge--Kutta (RK) methods are used to approximate the solution of initial value
ODEs:
\begin{equation} \label{eq:ivp}
    \begin{cases}
        \uu'(t) = \f(t,\uu(t)), \\
        \uu(t_0)= \uu_0,
    \end{cases}
\end{equation}
often resulting from the semi-discretization of partial differential equations (PDEs).
An $s$-stage RK method approximates the solution of \eqref{eq:ivp} as follows:
\begin{subequations} \label{eq:butcher}
\begin{align} \label{eq:stages}
        \y_i &= \un + \dt \sum_{j=1}^s a_{ij} \f(t_n+c_j \dt, \y_j) & (1\leq i \leq s), \\
        \unone &= \un + \dt \sum_{j=1}^s b_j \f(t_n+c_j \dt, \y_j).
\end{align}
\end{subequations}
Here $\un$ is a numerical approximation to $\uu(t_n)$,
$\dt = t_{n+1}-t_n$ is the step size, and the stage values $\y_j$ are approximations to
the solution at times $t_n+c_j\dt$.

Most analysis of RK methods assumes that the stage equations \eqref{eq:stages}
are solved exactly.  In practice, perturbed solution and stage values are computed:
\begin{subequations} \label{eq:butcher-pert}
\begin{align} 
        \ty_i &= \tun + \dt \sum_{j=1}^s a_{ij} \f(t_n+c_j \dt, \ty_j) + \tr_i & (1\leq i \leq s) \\
        \tunp &= \tun + \dt \sum_{j=1}^s b_j \f(t_n+c_j \dt, \ty_j) + \tr_{s+1}.
\end{align}
\end{subequations}
The {\em internal errors} (or stage residuals) $\tr_j$ include errors due to
\begin{itemize}
    \item roundoff; 
    \item finite accuracy of an iterative algebraic solver (for implicit methods).
\end{itemize}
The perturbed equations \eqref{eq:butcher-pert} are also used to study accuracy by
taking $\ty,\tu$ to be exact solution values to the ODE or PDE system, in which case
the stage residuals include
\begin{itemize}
    \item temporal truncation errors;
    \item spatial truncation errors;
    \item perturbations due to imposition of boundary conditions.
\end{itemize}
Such analysis is useful for explaining the phenomenon of {\em order reduction}
due to stiffness~\cite{Frank1985} or imposition of boundary conditions~\cite{Carpenter1995,abarbanel1996removal}.
The theory of BSI-stability and B-convergence has been developed to understand
these phenomena, and the relevant method property is usually the {\em stage order} \cite{Frank1985}.

The study of both kinds of residuals (due to roundoff or truncation errors) is referred to
as {\em internal stability}~\cite{VanderHouwen1980,Shampine1984,Sanz-Serna1986,Sanz-Serna1989,Verwer1990}.
We focus on the issue of amplification of roundoff errors in explicit
Runge--Kutta schemes, although we will see that some of our results and techniques
are applicable to other internal stability issues.
Since roundoff errors are generally much smaller than truncation errors, 
their propagation within a single step is not usually important.
However for explicit RK (ERK) methods with a large number of stages, the
constants appearing in the propagation of internal errors can be so large that
amplification of roundoff becomes an issue \cite{Verwer1977,VanderHouwen1980,Verwer1996}. 
Amplification of roundoff errors in methods with many stages is
increasingly important because
there now exist several classes of practical RK methods that use
many stages, including Runge--Kutta--Chebyshev (RKC) methods \cite{Verwer1990}, 
extrapolation methods \cite{Hairer2010}, deferred correction methods 
\cite{Dutt2000}, some strong stability preserving (SSP) methods 
\cite{SSPbook}, and other stabilized ERK methods 
\cite{Niegemann2011,Parsani2013}.  Furthermore, these methods are naturally implemented
not in the Butcher form \eqref{eq:butcher}, but in a
{\em modified Shu--Osher form} \cite{Ferracina2004,Higueras2005,SSPbook}:
\begin{equation}
    \begin{aligned} \label{eq:shu-osher}
        \y_i &= v_i \un + \sum_{j=1}^s \big(\alpha_{ij} \y_j + \dt
    \beta_{ij}\f(t_n+c_j \dt, \y_j) \big) \quad (1 \leq i \leq s+1), \\
        \unone &= \y_{s+1}.
    \end{aligned}
\end{equation}
As we will see, propagation of roundoff errors in these schemes should be based
on the perturbed equations
\begin{equation}
\begin{aligned} \label{eq:shu-osher-pert}
        \ty_i &= v_i \tun + \sum_{j=1}^s \big(\alpha_{ij} \ty_j + \dt
    \beta_{ij}\f(t_n+c_j \dt, \ty_j) \big) + \tr_i & (1 \leq i \leq s+1), \\
        \tunp &= \ty_{s+1}
\end{aligned}
\end{equation}
rather than on \eqref{eq:butcher-pert}, because internal error propagation (in 
contrast to traditional error propagation) depends on the form used to implement the method.
Through an example in Section \ref{sec:example} we will see that, even when
methods \eqref{eq:butcher} and \eqref{eq:shu-osher} are equivalent,
the corresponding perturbed methods \eqref{eq:butcher-pert} and \eqref{eq:shu-osher-pert}
may propagate internal errors in drastically different ways.
Thus the residuals in \eqref{eq:shu-osher-pert}
and in \eqref{eq:butcher-pert} will in general be different.
In Section \ref{sec:internal_error}, we elaborate on this difference and derive,
for the first time, completely general expressions for the internal stability polynomials.

We emphasize here that the difference between \eqref{eq:shu-osher} and \eqref{eq:butcher}
is distinct from the re-ordering of step sizes that was used to improve internal
stability in~\cite{VanderHouwen1980}.  Methods
\eqref{eq:shu-osher} and \eqref{eq:butcher} can have different internal stability
properties even when they are algebraically equivalent stage-for-stage.

In Section \ref{sec:iaf}, we introduce the {\em maximum internal amplification factor},
a simple characterization of how a method propagates internal errors.
Although we follow tradition and use the term {\em internal stability}, it should be
emphasized that this topic does not relate to stability
in the usual sense, as there is no danger of unbounded blow-up of errors, only their
substantial amplification.  In this sense, the maximum internal amplification factor
is similar to a condition number in that
it is an upper bound on the factor by which errors may be amplified.
In Section \ref{sec:control} we show that for a fixed ERK method, essentially any
set of internal stability polynomials can be obtained by modifying the implementation.

In Sections \ref{sec:SSP} and \ref{sec:extrap}, we analyze internal error
propagation for SSP and extrapolation methods, respectively.
Theorem \ref{thm:no-amplification} shows that SSP methods exhibit no
internal error amplification when applied under the usual assumption of
forward Euler contractivity.  Additional results in these sections
provide bounds on the internal amplification factor for general initial
value problems.
Much of our analysis follows along the lines of what was done in \cite{Verwer1990}
for RKC methods.  First we determine closed-form expressions
for the stability polynomials and internal stability polynomials of these methods.
Then we derive bounds and estimates for the maximum internal amplification factor.
Using these bounds, we prove error bounds in the presence of roundoff error for whole
families of methods where the number of stages may be arbitrarily large.


\subsection{Preliminaries}
In this subsection we define the basic setting and notation for our work. 
We consider the initial value problem (IVP) \eqref{eq:ivp}
where $\uu:[t_0,T] \rightarrow \Real^m$ and $\f:\Real\times\Real^m \rightarrow \Real^m$.
To shorten the notation, we will sometimes omit the first argument of $\f$, writing
$\f(\uu)$ when there is no danger of confusion.

The RK method \eqref{eq:butcher} and its properties 
are fully determined by the matrix $A =[a_{ij}] \in \mathbb{R}^{s \times s}$
and column vector $b=[b_j] \in \mathbb{R}^{s}$ 
which are referred to as the Butcher coefficients \cite{Butcher2008a}. 

Let us define
\begin{equation}\label{RKdef4}
    \begin{aligned}
\alpha & = \begin{pmatrix} \alphatop & 0 \\ \alphabottom & 0 \end{pmatrix} & 
    \alphatop & = \begin{pmatrix} \alpha_{11} & \cdots & \alpha_{1s} \\ \vdots & &  \vdots \\ \alpha_{s1} & \cdots & \alpha_{ss} \end{pmatrix} &
    \alphabottom = (\alpha_{s+1,1}, \dots, \alpha_{s+1,s}), \\
\beta & = \begin{pmatrix} \betatop & 0 \\ \betabottom & 0 \end{pmatrix} & 
    \betatop & = \begin{pmatrix} \beta_{11} & \cdots & \beta_{1s} \\ \vdots & &  \vdots \\ \beta_{s1} & \cdots & \beta_{ss} \end{pmatrix} &
    \betabottom = (\beta_{s+1,1}, \dots, \beta_{s+1,s}), \\
v & = \begin{pmatrix} \vvtop \\ \vvbottom \end{pmatrix} = (I_{s+1} - \alpha)\bone & \vvtop & = (v_1,\dots,v_s)^\top,
    \end{aligned}
\end{equation}
where $\bone$ is the column vector of length $s+1$ with all entries equal to
unity, and $I_k$ is the $k\times k$ identity matrix.
We always assume that $(I_s - \alphatop)^{-1}$ exists; methods without this property are not
well defined \cite{SSPbook}.
The methods \eqref{eq:butcher} and \eqref{eq:shu-osher} are equivalent under the conditions 
\begin{equation} \label{eq:so2b}
    \begin{aligned}
    A & = (I_s - \alphatop)^{-1} \betatop, \\
    b & = \betabottom + \alphabottom(I_s-\alphatop)^{-1}\betatop.
    \end{aligned}
\end{equation}
We assume that all methods satisfy the conditions for stage consistency of order one,
\textit{i.e.},
\begin{subequations} \label{eq:stage-order-one}
\begin{align}
v_i & = 1 - \sum_j \alpha_{ij}\label{1.8afirst}, \\
c & = A\bone = (I_s - \alphatop)^{-1} \betatop \bone.
\end{align}
\end{subequations}

Finally, define
\begin{align}
\by & = (\y_1, \dots, \y_{s+1})^\top & \bfun(\by) & = (\f(\y_1),\dots,\f(\y_s),0)^\top \\
\balpha & = \alpha \otimes I_m & \bbeta & = \beta \otimes I_m & \bv = v \otimes I_m,
\end{align}
where $\otimes$ denotes the Kronecker product.
The method \eqref{eq:shu-osher} can also then be written 
\begin{equation}
    \begin{aligned} \label{eq:modified-shu-osher-compact}
        \by &= \bv \un +  \balpha \by + \dt 
        \bbeta \bfun(\by), \\
        \unone &= \y_{s+1}.
    \end{aligned}
\end{equation}
Recall that $m$ denotes the dimension of $\uu$ and $s$ denotes the number of stages;
boldface symbols are used for vectors and matrices with dimension(s) of size $m (s+1)$
whenever $m\ge 2$.
When considering scalar problems ($m=1$), we use non-bold symbols for simplicity.

When studying internal error amplification over a single step, we will sometimes
omit the tilde over $\un$ to emphasize that we do not consider propagation of
errors from previous steps.

\begin{rem}
The Butcher representation of a RK method
is the particular Shu--Osher representation obtained
by setting $\alpha_{ij}$ to zero for all $i,j$ and setting
\begin{equation*}
\beta = \beta_0 := 
    \left(
        \begin{array}{cc}
            A & 0 \\
            b^\top  & 0
         \end{array}
    \right).
\end{equation*}
\end{rem}

\subsection{An example\label{sec:example}}
Here we present an example demonstrating the effect of internal error amplification.
We consider the following initial value problem (problem D2 of the non-stiff
DETEST suite~\cite{hull1972comparing}), whose solution traces an ellipse with eccentricity 0.3:
\begin{subequations} \label{eq:D2}
\begin{align} 
x''(t) & = -x/r^3 & x(0) & = 0.7 & x'(0) & = 0,\\
y''(t) & = -y/r^3 & y(0) & = 0   & y'(0) & = \sqrt{13/7}, \\
r^2 & = x^2 + y^2.
\end{align}
\end{subequations}
We note that very similar
results would be obtained with many other initial value problems.
We first compute the solution at $t=20$ using Fehlberg's 5(4) Runge--Kutta pair,
which is not afflicted by any significant internal amplification of error.
Results, shown in Figure \ref{D2}, are typical and familiar to any student
of numerical analysis.  As the tolerance is decreased, the step size controller
uses smaller steps and achieves smaller local---and global---errors, at the
cost of an increased amount of work.  Eventually, the truncation errors become
so small that the accumulation of roundoff errors is dominant and the overall error
cannot be decreased further.

Next we perform the same computation using a twelfth-order extrapolation
method based on the first-order explicit Euler method~\cite{Ketcheson2013};
this method has 67 stages.
The (embedded) error estimator for the extrapolation method is based on the eleventh-order 
diagonal extrapolation entry.  This pair is naturally implemented in
a certain Shu--Osher form (see \textbf{Algorithm} \ref{alg:extrap} in Section \ref{sec:extrap}).  The results, also shown in Figure
\ref{D2}, are similar to those of Fehlberg's method for large tolerances, 
although the number of steps required is much smaller for this twelfth-order
method.  However, for tolerances less than $10^{-9}$, the extrapolation method
fails completely.  The step size controller rejects every step and
continually reduces the step size; the integration cannot be completed to the
desired tolerance in a finite number of steps---even though that tolerance
is six orders of magnitude larger than roundoff!

Finally, we perform the same computation using an alternative implementation
of the twelfth-order extrapolation method.  The Butcher form \eqref{eq:butcher} is used for 
this implementation; it seems probable that no extrapolation method has
ever previously been implemented in this (unnatural) way.  The results
are again shown in Figure \ref{D2}; for large tolerances they are identical
to the Shu--Osher implementation.  For tolerances below $10^{-9}$,
the Butcher implementation is able to complete the integration, albeit
using an excessively large number of steps, and with errors much larger than
those achieved by Fehlberg's method at the same tolerance.

What is the cause for the surprising behavior of the extrapolation method?
We will return to and explain this example after describing the relevant theory.

\begin{figure}
\begin{center}
\subfigure[Global error versus input tolerance\label{fig:errvstol}]{
\includegraphics[width=0.4\textwidth]{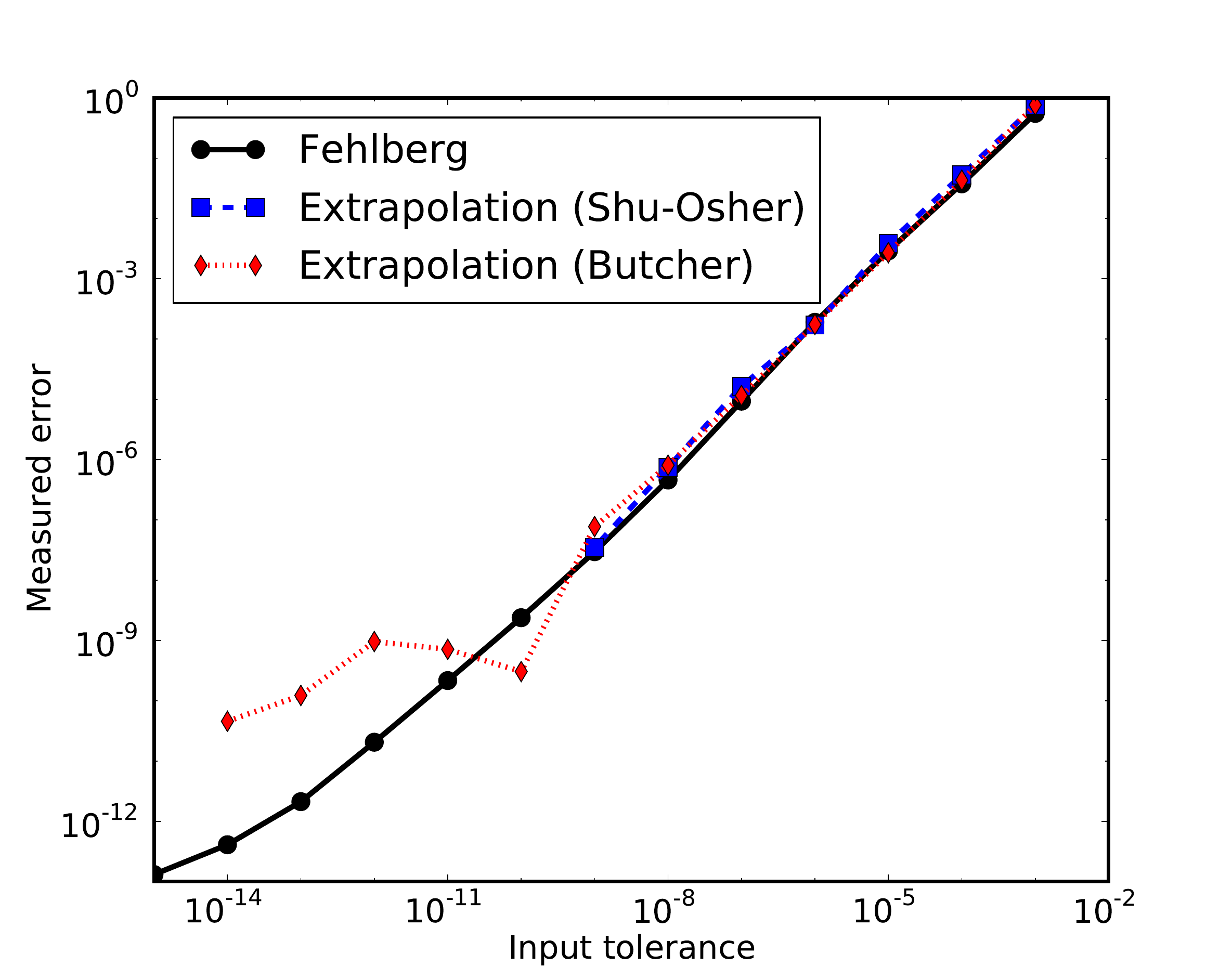}}
\subfigure[Number of steps versus global error\label{fig:stepsvserr}]{
\includegraphics[width=0.4\textwidth]{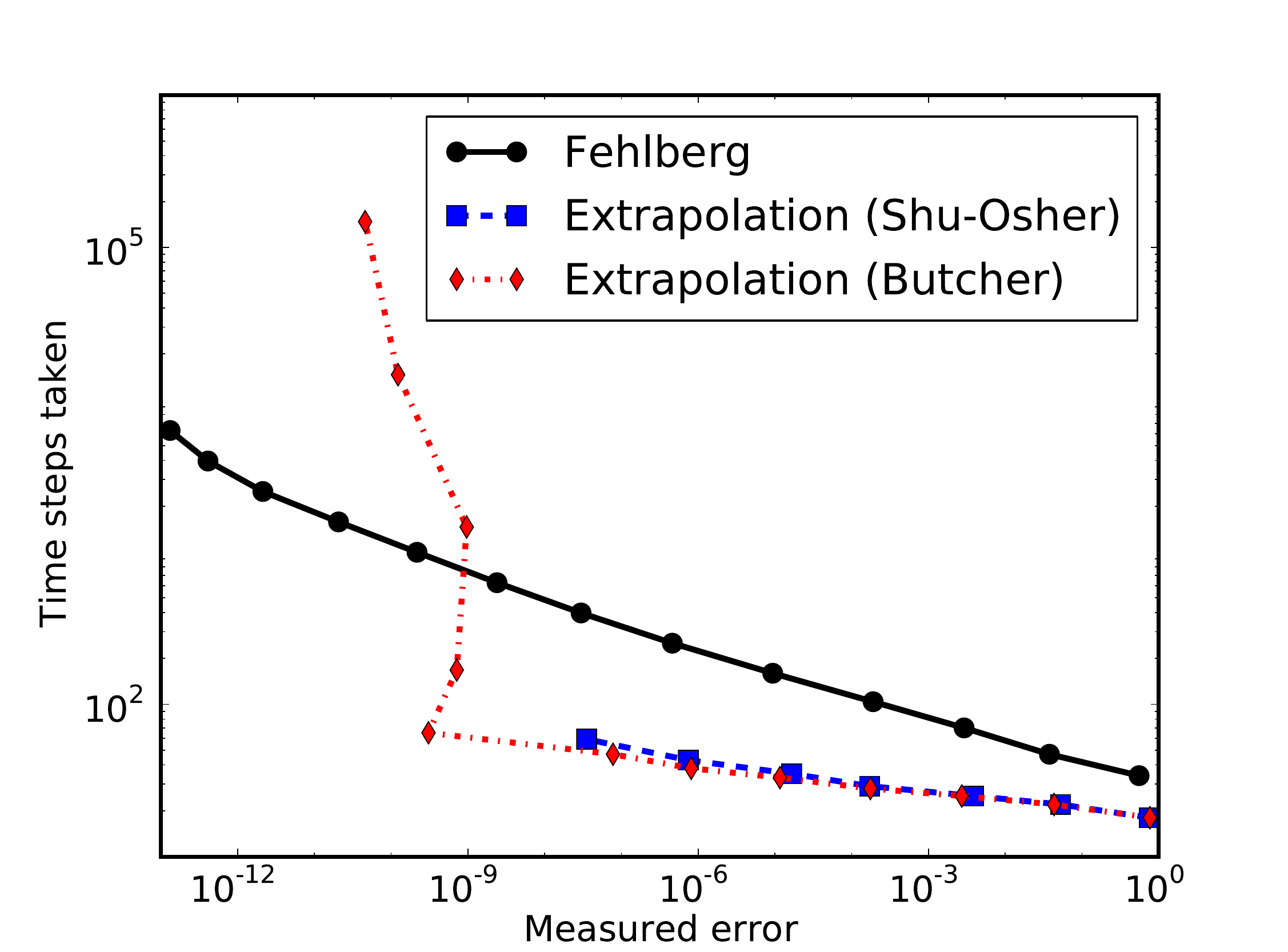}}
\caption{Internal instability of a twelfth-order extrapolation method applied
to problem \eqref{eq:D2}.
\label{D2}}
\end{center}
\end{figure}

\section{Internal error amplification\label{sec:internal_error}}
\subsection{Internal stability functions\label{sec:isf}}
Internal stability can be understood by considering
the linear, scalar ($m=1$) initial value problem  
\begin{align}\label{eq:linear-constant-IVP}
    \begin{cases}
            \uu'(t) = \lambda \uu(t), & \lambda \in \Complex\\
            \uu(t_0) = \uu_0.
        \end{cases}
\end{align}
We will consider the application of an RK method in Shu--Osher form \eqref{eq:modified-shu-osher-compact};
recall that the Butcher form \eqref{eq:butcher} is included as a special case.
Application of the RK method \eqref{eq:modified-shu-osher-compact} to 
the test problem \eqref{eq:linear-constant-IVP} yields
\begin{subequations}\label{eq:modified-shu-osher-2}
\begin{align}
    \y &= \vvtop \un + \alphatop \y + z \betatop \y
    \label{eq:stages-modified-shu-osher-2},\\
    \unone &= \vvbottom \un + \alphabottom \y + z \betabottom \y \label{eq:update-modified-shu-osher-2},
\end{align}
\end{subequations}
where $z = \dt \lambda$.
By solving equation \eqref{eq:stages-modified-shu-osher-2} for $\y$ 
and substituting into \eqref{eq:update-modified-shu-osher-2} we obtain
\begin{align}
    \unone & = P(z) \un,
\end{align}
where $P(z)$ is the \textit{stability function}, which in Shu--Osher variables
takes the form
\begin{equation}\label{eq:stab-fun-modified-shuosher}
P(z) := \vvbottom + \left(\alphabottom + z \betabottom
\right)\left(I_s - \alphatop-z\betatop \right)^{-1} \vvtop.
\end{equation}

If the stage equations are satisfied exactly,
then $P(z)$ completely determines the behavior of the numerical scheme for linear problems.
However, it is known that RK methods with many stages may exhibit
loss of accuracy even when $|P(z)|\le 1$, due to
propagation of errors within a single time step \cite{Verwer1977,VanderHouwen1980,Verwer1996}. 
In order to investigate the internal 
stability of a RK method we apply the perturbed scheme
\eqref{eq:shu-osher-pert}, which for problem \eqref{eq:linear-constant-IVP} yields
\begin{subequations}\label{eq:modufied-shuosher-2-perturbed}
\begin{align}
    \ty &= \vvtop \tun + \alphatop \ty + z\betatop  \ty + \tr_{1:s}
    \label{eq:stages-modified-shuosher-2-perturbed},\\
    \tunp &= \vvbottom \tun + \alphabottom \ty + z\betabottom \ty + \tr_{s+1}
    \label{eq:update-modified-shuosher-2-perturbed}.
\end{align}
\end{subequations}
Let $\berr_n = \tun - \un$ and $d = [d_{j}] = [\ty_{j} - \y_{j}]$.
Subtracting \eqref{eq:modified-shu-osher-2} from 
\eqref{eq:modufied-shuosher-2-perturbed} gives
\begin{align}
    d &= \vvtop \berr_n + \alphatop d + z \betatop d + \tr_{1:s}
    \label{eq:stage-errors-modified-shuosher-2},\\
    \berr_{n+1} &= \vvbottom \berr_n + \alphabottom d + z \betabottom d + \tr_{s+1}
    \label{eq:update-error-modified-shuosher-2}.
\end{align}
By solving \eqref{eq:stage-errors-modified-shuosher-2} for
$d$ and substituting the resulting expression in
\eqref{eq:update-error-modified-shuosher-2},
one arrives at the error formula
\begin{equation}\label{eq:error-modified-shuosher2}
    \begin{split}
        \berr_{n+1} & = P(z) \berr_n + Q(z;\alpha,\beta) \tr_{1:s} + \tr_{s+1}.
    \end{split}
\end{equation}
The stability function $P(z)$ has already been defined in (\ref{eq:stab-fun-modified-shuosher}), 
and the {\em internal stability functions} $Q(z;\alpha,\beta)$ are
\begin{align}\label{Qintstabfunct}
Q(z;\alpha,\beta) & \equiv (Q_1(z),Q_2(z),\ldots,Q_s(z)):= \left(\alphabottom + z \betabottom \right)\left(I_s - \alphatop - z\betatop
        \right)^{-1}.
\end{align}
Note that for convenience we have omitted the last component, $Q_{s+1}(z)$, which is always equal to 1.
We will often suppress the explicit dependence of $Q$ on $\alpha,\beta$ to keep the notation simple.
Using \eqref{eq:so2b}, we can obtain the expression
\begin{align} \label{eq:isp-form3}
Q(z) = z b^\top (I-zA)^{-1} (I -\alphatop)^{-1} + \alphabottom (I-\alphatop)^{-1}.
\end{align}

We will refer to $\epsilon$ as ``error'', though its exact interpretation depends
on what $\uu$ and $\tu$ refer to.
If $\tr$ represents roundoff error, then \eqref{eq:error-modified-shuosher2} indicates
the effect of roundoff on the overall solution.
The one-step error is given by the sum of two terms:
one governed by $P(z)$, accounting for propagation of errors committed in previous steps,
and one governed by $Q(z)$, accounting for propagation of the internal errors within the current
step.  In particular, $Q_j(z)$
governs the propagation of the perturbation $\tr_j$, appearing in stage $j$.
Clearly we must have $|P(z)|\le 1$ for stable propagation of errors, but if
the magnitude of $|Q_j(z)| \tr_j$ is larger than the magnitude of the desired tolerance 
then the second term is also important.

Note that $Q_j(z)$ herein is denoted by $Q_{sj}(z)$ in \cite{Verwer1990}.  
We are mostly interested in explicit RK methods, for which
$Q_j(z)$ is a polynomial of degree at most $s-j$
and the first component of $\tr$ is zero,
since no error is made in setting $\y_1=\un$.

\begin{rem}
For a method in Butcher form (\textit{i.e.}, with $\alpha=0$), \eqref{eq:isp-form3} reads
\begin{align} \label{eq:isp-butcher}
Q^\text{B}(z) := Q(z;0,\beta_0) = z b^\top (I-zA)^{-1}.
\end{align}
Formulas \eqref{eq:isp-form3} and \eqref{eq:isp-butcher} differ in an important way: we have
$Q^\text{B}(z)\to 0$ as $z \to 0$, so that the effects of internal errors vanish in the limit of infinitesimal
step size.  On the other hand, \eqref{eq:isp-form3} does not have this property, so
internal errors may still be amplified by a finite amount, no matter how small the 
step size is.  As we will see, this explains the different behavior of the two extrapolation
implementations in Section \ref{sec:example}.
\end{rem}

\subsubsection{Local defects}
Equation \eqref{eq:error-modified-shuosher2} can also be used to study the discretization error.
If we take
$\tun=\uu(t_n), \ty_i = \uu(t_n+c_i\dt), \tunp=\uu(t_{n+1})$, then $\epsilon_{n+1}$
is the global error and \eqref{eq:error-modified-shuosher2} describes how the
stage errors contribute to it.  We have
\[
\tr_{1:s} = \sum_{k=0}^p \frac{\dt^k}{k!} \theta_k u^{(k)}_h(t_n) + \order(\dt^{p+1}),
\]
where
\begin{equation} \label{eq:theta}
    \begin{aligned}
\theta_0 & = \bone - v - \alpha \bone, \\
\theta_k & = \frac{1}{k!}\left(I_s - \alphatop\right) \left( c^k - k A c^{k-1}\right)
    & &  & & (1\le k \le p).
    \end{aligned}
\end{equation}
Note that here $c \in \Real^{s+1}$ with $c_{s+1}=1$
and $c^k$ denotes the vector with $j$th entry equal to $c_j^k$.
Substituting the above into \eqref{eq:error-modified-shuosher2}, we obtain
\begin{align} \label{errexp_test}
\berr_{n+1} & = P(z) \berr_n + \sum_{k=2}^p \frac{\dt^k}{k!} u^{(k)}_h(t_n) Q(z) \theta_k 
        + \order(\dt^{p+1}).
\end{align}
For a method of order $p$, it can be shown that $\dt^k Q(z)\theta_k = \order(z^{p+1})$,
where $p$ is the classical order of the method, so the expected rate of convergence will
be observed in the limit $z\to 0$.  On the other hand, in problems arising from
semi-discretization of a PDE, it often does not make sense to consider the
limit $z\to0$, but only the limit $\dt \to 0$.  In that case, it can be 
shown only that $\dt^k Q(z)\theta_k = \order(\dt^{\tilde{p}+1})$, where $\tilde{p}$ is the 
{\em stage order} of the method; for all explicit methods, $\tilde{p}=1$.
This difference is responsible for the phenomenon of order reduction 
\cite{Sanz-Serna1986}.

If the stage equations are solved exactly, then the stage and solution values
computed using the Shu--Osher form \eqref{eq:shu-osher-pert} are identically equal to those
computed using the Butcher form \eqref{eq:butcher-pert}.
By comparing the Butcher and Shu--Osher forms, one finds that in general the
residuals are related by
\begin{align*}
\tr^\text{B}_{1:s} & = (I_s - \alphatop)^{-1} \tr^\text{SO}_{1:s} \\
\tr^\text{B}_{s+1} & = \alphabottom ( I_s-\alphatop )^{-1} \tr^\text{SO}_{1:s} + \tr^\text{SO}_{s+1}.
\end{align*}
Here $\tr^\text{B}, \tr^\text{SO}$ denote the residuals in \eqref{eq:butcher-pert} and
\eqref{eq:shu-osher-pert}, respectively.
Thus if the stage equations are solved exactly, the product of $Q(z)$ with the residuals is independent
of the form used for implementation.

If we wish to study the overall error in the presence of roundoff (\textit{i.e.}, the 
combined effect of discretization error and roundoff error), we may take $\un$
to be the solution given by the RK method in the presence of roundoff and
$\tun = \uu(t_n)$.  This leads to
\begin{subequations}
\begin{align} \label{errest}
\berr_{n+1} & = P(z) \berr_n + Q(z)\tr + \sum_{k=2}^p \frac{\dt^k}{k!} u^{(k)}_h(t_n) Q(z) \theta_k 
        + \order(\dt^{p+1}) \\
        & = P(z) \berr_n + Q(z)\tr + \order(\dt^{p+1}),
        \label{errexp_randd}
\end{align}
\end{subequations}
for a method of order $p$, where now $\tr$ denotes only the roundoff errors.
The effect of roundoff becomes significant when the last two terms in \eqref{errexp_randd} have
similar magnitude.

\subsubsection{Internal stability polynomials and implementation: an example\label{sec:simple-example}}
A given Runge--Kutta method can be rewritten in infinitely many different
Shu--Osher forms; this rewriting amounts to algebraic manipulation of the stage 
equations and has no effect on the method if one assumes the stage equations are 
solved exactly. However, the internal stability of a method depends on the
particular Shu--Osher form used to implement it.

For example, consider the two-stage, second order optimal SSP method
\cite{SSPbook}.  It is often written and implemented
in the following modified Shu--Osher form:
\begin{subequations}\label{ssp22}
\begin{align} 
\y_1 & = \un \\
\y_2 & = \y_1 + \dt \f(\y_1) \\
\unone = \y_3 & =\frac{1}{2} \un + \frac{1}{2}(\y_2 + \dt \f(\y_2)).
\end{align}
\end{subequations}
Applying \eqref{ssp22} to the test problem \eqref{eq:linear-constant-IVP} and introducing
perturbations in the stages we have
\begin{align*} 
\ty_1 & = \un \\
\ty_2 & = (1+z) \un + \tr_2 \\
\tunp & = \frac{1}{2} \un + \frac{1}{2}(1+z)\ty_2 + \tr_3.
\end{align*}
Substituting the equation for $\ty_2$ in that for $\tunp$ yields
$$\tunp = \left(1+z+\frac{1}{2} z^2\right) \un + \frac{1+z}{2} \tr_2 + \tr_3,$$
from which we can read off the stability polynomial
$P(z) = 1+z+z^2/2$ and the second-stage internal stability polynomial
\[Q_{2}(z) = (1+z)/2.\]

However, in the Butcher form, the equation for $\unone$ is written as
\[\unone = \un + \frac{1}{2} \dt (\f(\y_1) + \f(\y_2))\]
which leads by a similar analysis to
\[Q_{2}(z) = z/2.\]
More generally, the equation for $\unone$ can be written
\[\unone = \left(\frac{1}{2}+\beta_{31}\right)\un + \left(\frac{1}{2}-\beta_{31}\right) \y_2 + \beta_{31} \dt \f(\y_1) + \frac 1 2 \dt \f(\y_2),\]
with an arbitrary parameter $\beta_{31}\in\mathbb{R}$. 
By choosing a large value of $\beta_{31}$, the internal stability of the implementation can be made
arbitrarily poor.  For this simple method it is easy to see what a reasonable choice of
implementation is, but for methods with very many stages it is far from obvious.
In Section \ref{sec:control} we study this further.

Note that the stability polynomial $P(z)$ is independent of the choice of Shu--Osher form.

\subsection{Bounds on the amplification of internal errors\label{sec:iaf}}
We are interested in bounding the amount by which
the residuals $\tr$ may be amplified within one step, under the assumption
that the overall error propagation is stable.  
In the remainder of this section, we introduce some basic definitions
and straightforward results that are useful in obtaining such bounds.
It is typical to perform such analysis
in the context of the autonomous linear {\em system} of ODEs \cite{Verwer1990}:
\begin{align} \label{eq:ivp-linear}
    \begin{cases}
        \uu'(t)  & = L \uu(t) + g(t), \\
        \uu(t_0) & = \uu_0,
    \end{cases}
\end{align}
where $L \in \Real^{m \times m}$, $g : \Real\to\Real^m$.  Results based on such analysis
are typically useful in the context of more complicated problems, whereas
analyzing nonlinear problems directly does not usually yield further
insight~\cite{VanderHouwen1980,Sanz-Serna1989,Verwer1990,Carpenter1995}.

Application of the perturbed RK method \eqref{eq:shu-osher-pert} to 
problem \eqref{eq:ivp-linear} leads to \eqref{errexp_randd} but with $z=\dt L$.
Taking norms of both sides one may obtain
\begin{align} \label{eq1}
\| \epsilon_{n+1} \| \le \|P(z)\| \|\epsilon_n\| + \sum_{j=1}^{s+1} \|Q_j(z)\| \|\tr_j\| + \order(\dt^{p+1}).
\end{align}
It is thus natural to introduce
\begin{dfn}[Maximum internal amplification factor]
The {\em maximum internal amplification factor} of an $s$-stage RK method (\ref{eq:shu-osher}) with 
respect to a set $\Ss\subset\Complex$ is
\begin{align}\label{Mmaximalinternalamplificationfactor}
    \M \equiv  \M(\alpha,\beta,\Ss):= \max_{j=1, 2, \ldots, s} \sup_{z \in \Ss}  |Q_{j}(z)|,
\end{align}
where $Q_{j}(z)$ is defined in \eqref{eq:isp-form3}.
When the set $\Ss$ is not specified, it is taken to be
the absolute stability region of the method
$\Ss  = \left\{ z \in \Complex : |P(z)|\le 1 \right\}$, with $P(z)$ given by (\ref{eq:stab-fun-modified-shuosher}).
\end{dfn}

In order to control numerical errors, the usual strategy is to reduce the step size.
To understand the behavior of the error for very small step sizes, it is
therefore useful to consider the value
\[
\M_0 := \M(\alpha,\beta,\{0\}) = \max_{j=1, 2, \ldots, s} |Q_j(0)|.
\]

To go further, we need to make an assumption about $\sigma(L)$, the spectrum of $L$.
The next theorem provides bounds on the error in the presence of
roundoff or iterative solution errors, when $L$ is normal.  Similar results could be obtained
when $L$ is non-normal by considering pseudospectra.

\begin{thm} \label{thm:linear}
Let an RK method of order $p$ with coefficients $\alpha,\beta$ be applied to
\eqref{eq:ivp-linear}, where $L$ is a normal matrix and $\dt \sigma(L) \in \Ss$.
Let $\epsilon_n = \tun - \uu(t_n)$ where $\tun$ satisfies the perturbed RK equations
\eqref{eq:shu-osher-pert}.  Then 
\begin{align} \label{local-error}
\|\epsilon_{n+1}\| \le \|\epsilon_n\| +
 (s+1) \M(\alpha,\beta,\Ss) \max_{1\le j \le s+1} \|\tr_j\| + \order(\dt^{p+1}).
\end{align}
The factor $s+1$ can be replaced by $s$ for explicit methods, since then $\tr_1=0$.
\end{thm}
\begin{proof}
Use \eqref{eq1} and the fact that since $L$ is normal, $\|Q_j(\dt L)\| = \max_{\lambda\in \sigma(L)} |Q_j(\dt \lambda)|.$
\end{proof}

\begin{rem} \label{rem:pdes}
Combining the above with analysis along the lines of \cite{Verwer1990},
one obtains error estimates for application to PDE semi-discretizations.
The amplification of spatial truncation errors does not depend on the implementation,
as can be seen by working out the product $Q(Z) \beta$.
\end{rem}



As an example, in Table \ref{tbl:M} we list approximate maximum internal amplification factors for
some common RK methods.  All of these methods, like most RK methods, have relatively small 
factors so that their internal stability is generally not a concern.

\begin{table}
\begin{center}
\begin{tabular}{l|r|r}
    Method                                      &       $\M$ & $\M_0$ \\ \hline
    Three-stage, third order SSP \cite{shu1988} &       1.7  & 0 \\
    Three-stage, third order Heun \cite{heun1900neue} &       3.2  & 0 \\
    Classic four-stage, fourth order method  \cite{kutta1901beitrag}   &       1.7  & 0 \\
    Merson 4(3) pair   \cite{merson1957operational} &       5.6  & 0 \\
    Fehlberg 5(4) pair  \cite{fehlberg1969klassische} &       5.4  & 0 \\
    Bogacki--Shampine 5(4) pair \cite{Bogacki1996} &       7.0  & 0 \\
    Prince--Dormand order 8 \cite{Prince1981} &       138.8& 0 \\ 
    Ten-stage, fourth order SSP \cite{ketcheson2008highly} &       2.4  & 0.6 \\
    Ten-stage, first order RKC \cite{Verwer1990} &       10.0  & 10.0 \\
    Eighteen-stage, second order RKC \cite{Verwer1990} &       27.8 & 22.6 \\ \hline
\end{tabular}
\caption{Approximate maximum internal amplification factors for some RK methods.
         For RK pairs, the amplification factor of the higher-order method is listed.\label{tbl:M}}
\end{center}
\end{table}


\subsection{Understanding the example from Section \ref{sec:example}}
Using the theory of the last section, we can fully explain the results of
Section \ref{sec:example}.  First, in Table \ref{maxamp}, we give the values
of $\M$ and $\M_0$ for the three methods.  Observe that Fehlberg's method,
like most methods with few stages, has very small amplification constants.
Meanwhile, the Euler extrapolation method, with 67 stages, has a very large
$\M$.  However, when implemented in Butcher form, it necessarily has $\M_0=0$.

For the extrapolation method in Shu--Osher form, the local error will generally
be at least $\M_0 \cdot \epsmach \approx 10^{-10}$, for double precision calculations.  
Therefore, the pair will fail when the requested
tolerance is below this level, which is just what we observe.
The extrapolation method in Butcher form also begins to be afflicted by
amplification of roundoff errors at this point (since it has a similar
value of $\M$).  Reducing the step size has the effect of reducing the 
amplification constant, but only at a linear rate (\textit{i.e.},
in order to reduce the local error by a factor of ten, roughly ten times
as many steps are required).  This is observed in Figure \ref{fig:stepsvserr}.  Meanwhile,
the global error does not decrease at all, since the number of steps taken
(hence the number of roundoff errors committed) is increasing just as fast
as the local error is decreasing.  This is observed in Figure \ref{fig:errvstol}.



\begin{table}
\begin{center}
\begin{tabular}{l|cc}
    Method                                 & $\M$             & $\M_0$ \\ \hline
    Fehlberg 5(4)                          & $5.4$            & 0 \\
    Euler extrapolation 12(11) (Shu--Osher) & $3.4\cdot 10^5$ & $1.3 \cdot 10^5$ \\
    Euler extrapolation 12(11) (Butcher)   & $1.7\cdot 10^5$ & 0 \\ \hline
\end{tabular}
\caption{Maximum internal amplification factors for methods used in the example 
in Section \ref{sec:example}.\label{maxamp}}
\end{center}
\end{table}

\subsection{Effect of implementation on internal stability polynomials\label{sec:control}}
In Section \ref{sec:simple-example}, we gave an example showing how the internal stability
functions of a method could be modified by the choice of Shu--Osher implementation.
It is natural to ask just how much control over these functions is possible.

Given an $s$-stage explicit Runge--Kutta method, let $Q_1^\text{B}, Q_2^\text{B},\dots,Q_s^\text{B}$
denote the internal stability polynomials corresponding to the Butcher form implementation.
Let $d_j$ denote the degree of $Q^\text{B}_j$ and let $w_j$ denote the coefficient of $z^{d_j}$
in $Q^\text{B}_j$.
Let $Q^\text{SO}_1,Q^\text{SO}_2, \dots,Q^\text{SO}_s$ be any set of polynomials with the same degrees $d_j$ and
leading coefficients $w_j$, but otherwise arbitrary.
Then it is typically possible to find a Shu--Osher implementation of the given method such that
the internal stability polynomials are $Q^\text{SO}_1, Q^\text{SO}_2, \dots, Q^\text{SO}_s$.

How can this be done?  Comparing equations \eqref{eq:isp-form3} and \eqref{eq:isp-butcher}, 
it is clear that, except for the constant terms,
the internal stability polynomials corresponding to a given Shu--Osher implementation
are linear combinations of the internal stability polynomials of the Butcher implementation:
\begin{align} \label{btoso}
Q^\text{SO}(z) = Q^\text{B}(z) \gamma + \alphabottom \gamma.
\end{align}
Here $\gamma = (I-\alphatop)^{-1}$ can be any lower-triangular matrix with unit entries
on the main diagonal.  Given $Q^\text{B}(z)$, one can choose $\gamma$ to obtain the desired
polynomials $Q^\text{SO}(z)$ except for the constant terms, and then choose $\alphabottom$
to obtain the desired constant terms.  We have added the qualifier {\em typically}
above because it is necessary that successive subsets of the $Q^\text{B}_j$ span
the appropriate polynomial spaces.

This could be used, in principle, to improve the internal stability of a given method,
such as the twelfth-order extrapolation method in the example from Section \ref{sec:example}.
For example, as desired internal stability polynomials, we take 
$$Q^\text{SO}_j(z) = \sum_{k=1}^{d_j} \frac{z^k}{k!},$$
\textit{i.e.}, the truncated Taylor polynomials of the exponential, but with zero constant term.
This particular choice of polynomials was arrived at more by experiment than analysis, and could almost
certainly be improved upon.  Nevertheless, solving \eqref{btoso} for $\gamma$ and then $\alpha$
we obtain an implementation characterized by 
$$ \M = 8.3\cdot 10^4 \ \ \ \text{and} \ \ \M_0 = 0,$$
which are a noticeable improvement over either the Butcher or the natural Shu--Osher implementation,
though $\M$ is still rather large.

In practice, this implementation behaves similarly to the Butcher implementation, probably
because at tight tolerances the amplification is dominated by factors other than
$\M, \M_0$.
Can the practical behavior of internal errors be substantially improved by choosing a particular
implementation of a method?  The question remains open and merits further research.


\section{Internal amplification factors for explicit strong stability preserving methods\label{sec:SSP}}
In this section we prove bounds on and estimates of the internal amplification
factors for optimal explicit SSP RK methods of order 2 and 3.  These methods have
extraordinarily good internal stability properties.

We begin with a result showing that no error amplification occurs when an SSP
method is applied to a contractive problem.
Recall that an SSP RK method with SSP coefficient $\sspcoeff$ can be written in 
the {\em canonical Shu--Osher form} where $\alpha = \sspcoeff\beta$, $\alpha_{ij}\ge 0$
for all $i,j$, and the row sums of $\alpha$ are no greater than 1~\cite{SSPbook}. 
\begin{thm} \label{thm:no-amplification}
    Suppose that $F$ is contractive with respect to some semi-norm $\|\cdot\|$:
    \begin{align} \label{eq:FE-cond}
        \|U+ \tau F(U) - (\tu + \tau F(\tu))\| & \le \|U - \tu\| &
        \text{for all $U,\tu$ and for $0\le \tau \le \tau_0$}.
    \end{align}

    Let an explicit RK method with SSP coefficient $\sspcoeff>0$ be given
    in canonical Shu--Osher form $\alpha, \beta$ and applied with step size
    $\tau \le \sspcoeff \tau_0$.  Then
    \begin{align} \label{eq:no-amplification}
        \|\epsilon_{n+1}\| \le \|\epsilon_n\| + \sum_{j=1}^{s+1} \|\tr_j\|.
    \end{align}
\end{thm}
\begin{proof}
The proof proceeds along the usual lines of results for SSP methods.
From \eqref{eq:shu-osher} and \eqref{eq:shu-osher-pert}, we have
\begin{align*}
    \y_i  &= v_i \un  + \sum_{j=1}^{i-1} \alpha_{ij} \left(\y_j  + \frac{\dt}{\sspcoeff} \f(\y_j) \right) \\
    \ty_i &= v_i \tun + \sum_{j=1}^{i-1} \alpha_{ij} \left(\ty_j + \frac{\dt}{\sspcoeff} \f(\ty_j) \right) + \tr_i.
\end{align*}
Taking the difference of these two equations, we have
\begin{align*}
    d_i  &= v_i \epsilon_n  + \sum_{j=1}^{i-1} \alpha_{ij} \left(d_j  + \frac{\dt}{\sspcoeff} \left(\f(\ty_j)-\f(\y_j) \right)\right) + \tr_i,
\end{align*}
where $d$ and $\epsilon$ are defined in Section \ref{sec:isf}.
Applying $\|\cdot\|$ and using convexity, \eqref{eq:FE-cond} and
\eqref{1.8afirst}, we find
\begin{align*}
    \|d_i\|  & \le v_i \|\epsilon_n\|  + \sum_{j=1}^{i-1} \alpha_{ij} \left\|d_j  + \frac{\dt}{\sspcoeff} \left(\f(\ty_j)-\f(\y_j)\right) \right\| + \|\tr_i\| \\
    & \le v_i \|\epsilon_n\|  + \sum_{j=1}^{i-1} \alpha_{ij} \|d_j\| + \|\tr_i\| \\
    & \le \max\left(\|\epsilon_n\|,\max_{1\le j\le i-1} \|d_j\|\right) + \|\tr_i\|.
\end{align*}
Applying the last inequality successively for $i=1,2,\dots,s+1$ shows that $\|d_i\|\le \|\epsilon_n\| + \sum_{j=1}^i \|\tr_j\|$.  In particular, for $i=s+1$ this gives \eqref{eq:no-amplification}.
\end{proof}

The above result is the only one in this work that applies directly to nonlinear problems.
It is useful since it can be applied under the same assumptions that make SSP methods
useful in general.  On the other hand, SSP methods are very often applied under circumstances
in which the contractivity assumption is not justified---for instance, in the time integration
of WENO semi-discretizations.  In the remainder of this section, we develop bounds on
the amplification factor for some common SSP methods.  Such bounds can be applied even
when the contractivity condition is not fulfilled.

\begin{thm}\label{SSPdiskinternalstability}
Let an explicit SSP RK method with $s$ stages and SSP coefficient $\sspcoeff>0$ 
be given in canonical Shu--Osher form
$\alpha, \beta$ and define
\begin{align} \label{eq:disk}
D_\sspcoeff:= \{ z \in \Complex : |z+\sspcoeff|\le \sspcoeff \}.
\end{align}
Then
\[\M(\alpha,\beta,D_\sspcoeff) \le 1.\]
\end{thm}
\begin{proof}
Setting $\beta = \alpha/\sspcoeff$ in \eqref{Qintstabfunct} and using the fact that $\alphatop$
is strictly lower-triangular, we have
\begin{align*}
Q(z) & = \left(1+\frac{z}{\sspcoeff}\right)\alphabottom \left(I_s - \left(1+\frac{z}{\sspcoeff}\right)\alphatop\right)^{-1} \\
     & = \left(1+\frac{z}{\sspcoeff}\right) \alphabottom \sum_{k=0}^{s-1} \left(1+\frac{z}{\sspcoeff}\right)^{k} \alphatop^k.
\end{align*}
Thus for $z \in D_\sspcoeff$
\begin{align*}
|Q_j(z)| & \le \max_j \left(\alphabottom \sum_{k=0}^{s-1} \alphatop^k\right)_j \le \max_{i,j} \left(\sum_{k=0}^{s-1} \alphatop^k\right)_{ij}.
\end{align*}
In the last inequality, we have used the fact that $\sum_j \alpha_{s+1,j} \le 1$.
It remains only to show that the entries of the last matrix above are no larger than unity.

Observe that
\[
\sum_{k=0}^{s-1} \alphatop^k = I + \alphatop\left(I +
\alphatop\left(I+\alphatop\left(\cdots \right)\right)\right),
\]
where the number of factors $\alphatop$ is equal to $s-1$.  Next let $\mu=I+\alphatop$ and observe that $0\le \mu_{ij}\le 1$.
Since the rows of $\alphatop$ sum to at most 1, then each row of $\alphatop \mu$ is a convex combination of the
rows of $\mu$ and $0\le (\alphatop\mu)_{ij}\le 1$.  It follows that
\[
\max_{i,j} \left(\sum_{k=0}^{s-1} \alphatop^k\right)_{ij} \le 1.
\]
\end{proof}

\begin{rem} \label{rem:ssp}
By exactly the same means, one can show the following: let an explicit SSP RK method with SSP 
coefficient $\sspcoeff > 0$ be given in canonical Shu--Osher form $\alpha,
\beta$ and applied to the linear system \eqref{eq:ivp-linear}. 
Suppose that the forward Euler condition
\begin{align}
    \| \uu(t) + \tau(L\uu(t) + g(t))\| \le \|\uu(t)\|
\end{align}
holds for all $\tau \le \tau_0$.  Then $\|Q(z)\|\le 1$ for $\tau \le \sspcoeff \tau_0$.
\end{rem}

\subsection{Optimal second order SSP methods}
Here we study the family of optimal second order SSP RK methods
\cite{SSPbook}, corresponding to the most natural implementation.  For any number of stages $s\ge2$,
the optimal method is
\begin{align} \label{ssp2}
\begin{aligned}
\y_1 & = \un \\
\y_j & = \y_{j-1} + \frac{\dt}{s-1} \f(\y_{j-1}) & 2\le j\le s \\
\unone & = \frac{1}{s} \un + \frac{s-1}{s} \left(\y_s + \frac{\dt}{s-1} \f(\y_s) \right).
\end{aligned}
\end{align}

\begin{thm} \label{ssp2thm}
The maximum internal amplification factor for the method \eqref{ssp2}
satisfies
\[\M_s^{\mathrm{SSP2}} \le \frac{s+1}{s}.\]
\end{thm}
\begin{proof}
It is convenient to define
\begin{equation}\label{scalednuSSP2}
\nu(z) = 1 + \frac{z}{s-1}.
\end{equation}
Then the stability and internal stability functions are
\begin{align*}
P(z) & = \frac{1}{s} + \frac{s-1}{s} \nu(z)^s & Q_{j}(z) & = \frac{s-1}{s} \nu(z)^{s-j+1} \quad (2\le j\le s).
\end{align*}
(For brevity, we omit the details of the derivation here, but we will give a detailed proof of the analogous
SSP3 case in Lemma \ref{ssp3stabpolydetermination}.)
We have
\begin{align*}
\Ss_s & = \left\{ z \in \Complex : \left| \nu(z)^s + \frac{1}{s-1}\right|\le \frac{s}{s-1} \right\}.
\end{align*}
Let $z\in \Ss_s$ be given; we will show that $|Q_j(z)| \le \frac{s+1}{s}$ for each $j$.
We see that
$\nu(z)^s$ lies in a disk of radius $s/(s-1)$ centered at $-1/(s-1)$. Hence
\begin{align*}
\Ss_s \subset \Ss'_s := \left\{ z \in \Complex : |\nu(z)^s|\le \frac{s+1}{s-1}\right\}.
\end{align*}
For $|\nu|\le 1$ the desired result is immediate.  For $|\nu|>1$, we have
\[
\max_{j=2,\ldots, s} \sup_{z\in \Ss_s} |Q_{j}(z)|  = \max_{j=2,\ldots, s}\sup_{z\in \Ss_s} \frac{s-1}{s} |\nu(z)|^{s-j+1} \le\] \[
  \max_{j=1,\ldots, s} \sup_{z\in \Ss'_s} \frac{s-1}{s}|\nu(z)|^{s-j+1}  \le \] \[
  \sup_{z\in \Ss'_s} \frac{s-1}{s}|\nu(z)|^{s}  \le
 \frac{s-1}{s} \cdot \frac{s+1}{s-1}= \frac{s+1}{s}.
\]
\end{proof}

\begin{figure}
\begin{center}
\includegraphics[width=0.35\textwidth]{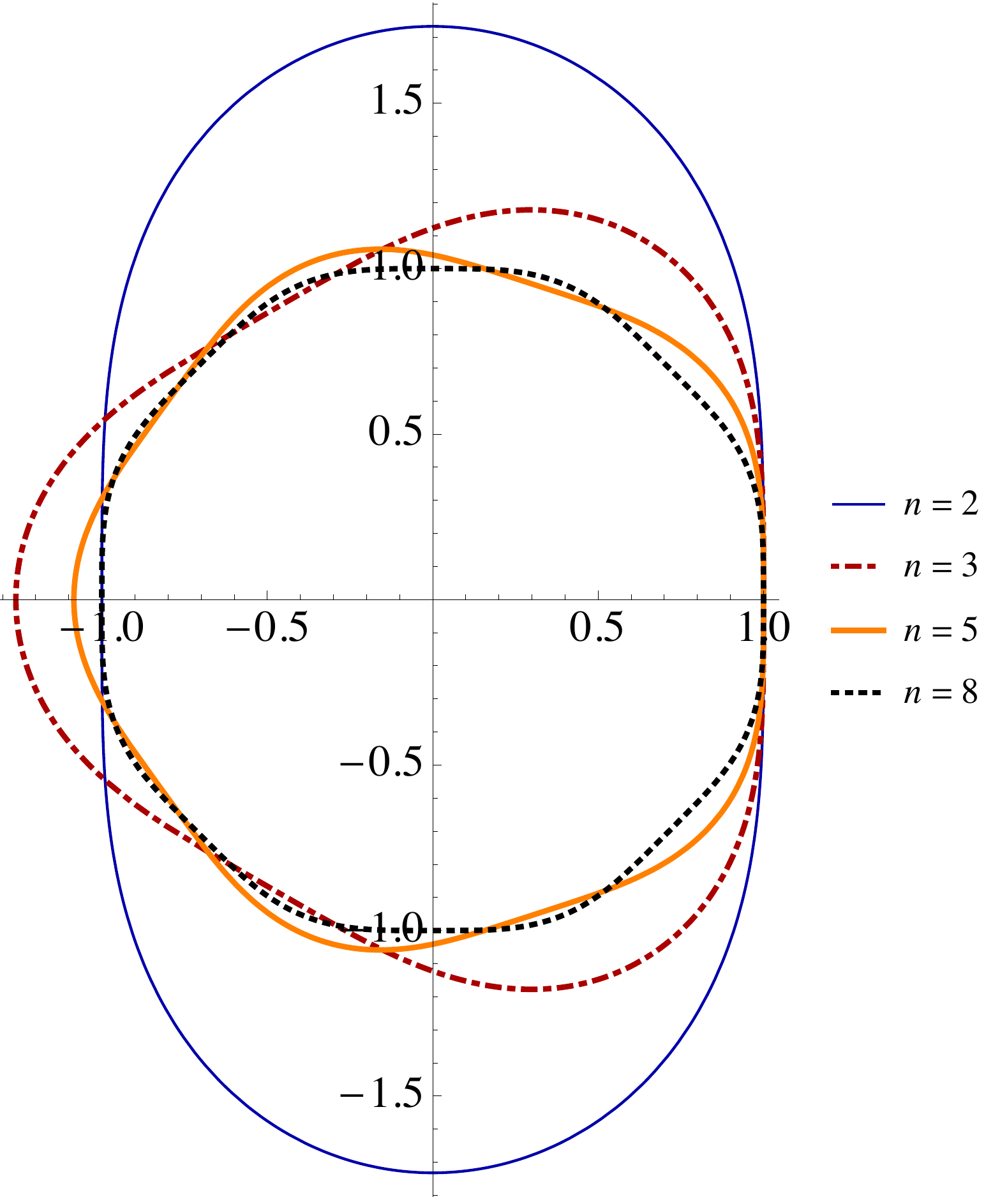}
\caption{Scaled and shifted stability regions (in the sense of (\ref{scalednuSSP2})) 
 of the optimal second order SSP RK methods 
for $s\in\{2, 3, 5, 8\}$.
The boundary curves on the plot are given by $\{\nu\in\Complex :
\left| \frac{s-1}{s}\nu^s+ \frac{1}{s}\right|=1\}$. For any $s$, the
corresponding region has $s$-fold rotational symmetry. \label{fig:ssp2fig}}
\end{center}
\end{figure}

\begin{rem}\label{remark33}
Notice that the ``For $|\nu|\le 1$ the desired result is immediate'' step
in the above proof corresponds to a 
special case of Theorem \ref{SSPdiskinternalstability}, since now 
$\sspcoeff=s-1$ (see \cite{SSPbook}), so $z\in D_\sspcoeff$ is equivalent to $|\nu(z)|\le 1$.
The set $\left\{ z \in \Complex : 
\left| \nu(z) \right|\le 1\right\}$ is represented as the unit disk in Figure \ref{fig:ssp2fig}.
\end{rem}

Combining the above result with Theorem \ref{thm:linear}, Remark \ref{rem:pdes}, or
Remark \ref{rem:ssp}, one obtains full error estimates for application to linear
systems and linear PDEs.

\subsection{Optimal third order SSP methods}\label{optimalthirdorderSSPmethods}

In this section we give various upper and lower bounds on the maximum internal 
amplification factor for the optimal third order SSP Runge--Kutta methods with 
$s\equiv s_n :=n^2$ ($n\ge 2$) stages that were proposed in \cite{ketcheson2008highly}.
The optimal third order SSP RK method with $s=n^2$ stages can be written as follows. 
\begin{align}\label{ssp3}
\begin{aligned}
	\y_1 &= \un \\
	\y_j &= \y_{j-1} + \frac{\dt}{n^2-n} \f(\y_{j-1}) & 2 \leq j\leq s, \, j \neq k_n \\
	\y_{k_n} &= \frac{n-1}{2n-1}\y_{k_n-1} + \frac{n}{2n-1}\y_{m_n} + \frac{\dt}{n(2n-1)} \f(\y_{k_n-1}) \\
	\unone &= \y_{s+1} = \y_s + \frac{\dt}{n^2-n} \f(\y_{s}),
\end{aligned}
\end{align}
where $n\ge 2$, $k_n := \frac{n(n+1)}{2}+1$ and $m_n:= \frac{(n-1)(n-2)}{2}+1$.\\

The main result of the section is the following theorem.

\begin{thm}\label{thm:SSP3}
For any $\vr\ge 1$ and $n\ge 2$ we define 
\[
\mu_n^-(\vr):=-1-\frac{n \vr^{(n-1)^2} \left( 1-\left(1-\frac{1}{n}\right) \vr^{2 n-1}\right) }{2 n-1}
\]
and denote by $\nu_n^*$ the unique root of $\mu_n^-$ in the interval $[1,+\infty)$. Then 
for any $n\ge 2$ and with $s=n^2$, the maximum internal amplification
factor for the optimal $s$-stage, third order SSP Runge--Kutta method satisfies
\begin{equation}\label{theorem35nunstarrepresentation}
\M_s^{\mathrm{SSP3}}=\left({\nu_n^*}\right)^{(n^2-n)/2}.
\end{equation}
For any $n\ge 9$ we have
\[
\frac{9}{10} \sqrt{\frac{n}{\ln (n)}}<\left(1+\frac{\ln (n)}{n^2}-\frac{\ln (\ln (n))}{n^2}\right)^{\left(n^2-n\right)/2}<\M_s^{\mathrm{SSP3}}<
\]
\[\left(1+\frac{\ln (n)}{n^2}-\frac{\ln (\ln (n))}{8 n^2}\right)^{\left(n^2-n\right)/2}<\frac{\sqrt{n}}{\sqrt[16]{\ln (n)}},
\]
and for $2\le n\le 10$ the exact values of $\M_s^{\mathrm{SSP3}}$ are presented in Table
\ref{exactMsSSP3values}.
\end{thm}

\begin{table}[h!]
\begin{center}
    \begin{tabular}{c|c||c|c}
    $n$ & $\M_{s}^{\mathrm{SSP3}}$ & $n$ & $\M_{s}^{\mathrm{SSP3}}$ \\ \hline
    2 & $1.575$  & 7 & $2.314$  \\
    3 & $1.794$ & 8 &  $2.411$ \\
    4 & $1.956$ & 9 &  $2.501$ \\
    5 & $2.091$ & 10 &   $2.585$ \\
    6 & $2.209$ &  &  \\ \hline
       \end{tabular}
\caption{Maximum internal amplification factors for the first few optimal 
third order SSP RK methods with $s=n^2$ stages. The values in the table have been rounded up.}\label{exactMsSSP3values}
\end{center}
\end{table}

\begin{rem}\label{remark3.6} As a trivial consequence, for $s=n^2$ with $n\ge 4$ we see that 
 \[\M_s^{\mathrm{SSP3}}<\sqrt[4]{s}=\sqrt{n}.\]
As an illustration of Theorem \ref{thm:SSP3}, the inequalities for $s=10^2$, 
 $s=10^4$ and $s=10^{12}$ read as follows:
\[1.875<1.927<\M_{10^2}^{\mathrm{SSP3}}\approx 2.585 <2.661<3.002,\]
\[4.193<4.587<\M_{10^4}^{\mathrm{SSP3}}\approx 5.757 <8.887<9.090,\]
\[242.135<269.038<\M_{10^{12}}^{\mathrm{SSP3}}\approx 302.551 <848.642<848.647.\]
Moreover, if we \textit{assume} that the asymptotic expansion of $\nu_n^*$ starts as
\[
\nu_n^* \sim
1+\frac{\ln(n)}{n^2}-\frac{\ln(\ln(n))}{n^2}+\frac{\ln(\ln
   (n))}{n^2 \ln(n)}-\frac{\ln(\ln(n))}{n^2 \ln^2(n)} \quad\quad (n\to +\infty),
\]
a formula obtained as a result of a heuristic argument but not proved rigorously, see
Remark \ref{remark315heuristicasymptotics} (but see 
Remark \ref{section3closingremark} as well), then we would have, for example, 
\[
\lim_{n\to +\infty}\frac{\M_{n^2}^{\mathrm{SSP3}}}{\sqrt{\frac{n}{\ln (n)}}}=1.
\]
\end{rem}



The proof of Theorem \ref{thm:SSP3} is given in the remainder of this section as follows.
The internal stability polynomials 
are described in Section \ref{generatingQSSP3subsection}. 
In Sections \ref{therealslicesubsection} and
\ref{explicitestimatesnunstarsubsection}, some reductions are carried out
and estimates on $\nu_n^*$ are proved.
The proof of Theorem \ref{thm:SSP3}
is completed in Section \ref{estimatingMssp3subsection}.

\subsubsection{The internal stability polynomials on the absolute stability region}
\label{generatingQSSP3subsection}



Just as in the proof of Theorem \ref{ssp2thm}, 
it will be convenient to apply some normalization and
introduce the scaling and shift
\begin{equation}\label{SSP3scalingdef}
\nu_n(z):= 1 + \frac{z}{n^2-n}\quad\quad (n\ge 2, \, z\in \Complex).
\end{equation}

The following lemma shows that the stability polynomial and the internal stability polynomials 
of this method can simply be expressed in terms of $\nu_n$.

\begin{lem}\label{ssp3stabpolydetermination} 
For any $n\ge 2$, the stability function of
the optimal third order SSP RK method with $s=n^2$ stages is
\[
P(z) = \frac{n-1}{2n-1} \nu_n(z)^{n^2} +\frac{n}{2n-1} \nu_n(z)^{(n-1)^2},
\]
while the internal stability functions are
\[
\begin{aligned}
	Q_{j}(z) = \begin{cases}
			\dfrac{n-1}{2n-1} \nu_n(z)^{n^2-j+1} +
\dfrac{n}{2n-1} \nu_n(z)^{(n-1)^2-j+1} \quad & 2 \leq j \leq m_n, \\\\
			\dfrac{n-1}{2n-1}  \nu_n(z)^{n^2-j+1} & m_n+1 \leq j \leq k_n-1, \\\\
			\nu_n(z)^{n^2-j+1} & k_n \leq j \leq n^2.
		\end{cases}
\end{aligned}
\] 
\end{lem}
\begin{proof}
For simplicity, we will use $\nu\equiv \nu_n(z)$ in the proof. 
The perturbed scheme for the linear test problem $\uu'(t) = \f(\uu(t)):=\lambda  \uu(t)$ with 
$z:=\lambda \dt$ now gives
\[
\begin{aligned}
	\ty_1& =\un \\
        \ty_j &= \nu \ty_{j-1} + \tr_j \quad & 2 \leq j\leq s, j \neq k_n \\
	\ty_{k_n} &= \nu \frac{n-1}{2n-1}\ty_{{k_n}-1} + \frac{n}{2n-1}\ty_{m_n} + \tr_{k_n} \\
        \tunp &=\nu\ty_{s}+\tr_{s+1}.
\end{aligned}
\]
The following steps hold for each $n\ge 2$ with the usual convention that 
$\sum_{j=j_0}^{j_1}(\ldots)=0$ when $j_0>j_1$
(occurring only in the $n=2$ case).

On one hand, as already remarked earlier, the coefficient of $\tr_1$ is always $0$ (since 
no error is made in setting $\y_1=\un$). On the other hand, 
the coefficient of $\tr_{s+1}$ is always $1$ (implying $Q_{s+1}\equiv 1$),
so ignoring this last $\tr_{s+1}$ term from the recursion simplifies the description 
further and does not affect the final values
of $P$ and $Q_j$ ($2\le j \le s$). Hence
\[
\tunp = \nu^{s-k_n+1}\ty_{k_n}+ \sum_{j=1}^{s-k_n} \nu^j \tr_{s-j+1} =
\]
\begin{equation}\label{equation39}
	  \frac{n-1}{2n-1} \nu^{s-k_n+2} \ty_{k_n-1} + \frac{n}{2n-1} \nu^{s-k_n+1} \ty_{m_n} + 
	\sum_{j=1}^{s-k_n+1} \nu^j \tr_{s-j+1}.
\end{equation}
Here, due to $m_n<k_n$, we have 
\[
\begin{aligned}
\ty_{m_n} = \nu^{m_n-1}\un + \sum_{j=1}^{m_n-1} \nu^{j-1} \tr_{m_n-j+1}
\end{aligned}
\]
and
\[
\begin{aligned}
\ty_{k_n-1} = \nu^{k_n-2}\un + \sum_{j=1}^{k_n-2} \nu^{j-1} \tr_{k_n-j}.
\end{aligned}
\]
Substituting these $\ty_{m_n}$ and $\ty_{k_n-1}$ values into (\ref{equation39}) we get
\[
	\tunp = \left(\frac{n-1}{2n-1} \nu^s +\frac{n}{2n-1} \nu^{s-k_n+m_n}\right) \un+ \]
\[
   \frac{n-1}{2n-1}\sum_{j=1}^{k_n-2} \nu^{s-k_n+j+1} \tr_{k_n-j} 
		+ \frac{n}{2n-1}\sum_{j=1}^{m_n-1} \nu^{s-k_n+j} \tr_{m_n-j+1} + 
\sum_{j=1}^{s-k_n+1} \nu^j \tr_{s-j+1}.
\]
After regrouping the sums, we obtain
\[
\tunp = \left(\frac{n-1}{2n-1} \nu^{n^2} +\frac{n}{2n-1} \nu^{(n-1)^2}\right) \un + \]
\[
 \sum_{j=2}^{m_n}\left(\frac{n-1}{2n-1} \nu^{n^2-j+1} +\frac{n}{2n-1} \nu^{(n-1)^2-j+1}\right) \tr_j 
	+ \frac{n-1}{2n-1}\sum_{j=m_n+1}^{k_n-1} \nu^{n^2-j+1} \tr_j + 
\sum_{j=k_n}^{n^2} \nu^{n^2-j+1} \tr_j.
\]
The stability polynomial and the 
internal stability polynomials appear as the coefficients of $\un$ and $\tr_j$.
\end{proof}

In order to determine the maximum internal amplification factor 
$\M_s^{\mathrm{SSP3}}\equiv \M_s^{\mathrm{SSP3}}(\Ss_s)$ 
for the method (\ref{ssp3}), we are going to estimate  
$\displaystyle \max_{j=2,\ldots, s} \sup_{z\in \Ss_s} |Q_{j}(z)|$, where 
\[
\Ss_s = \left\{ z \in \Complex : 
\left|\frac{n-1}{2n-1} \nu_n(z)^{n^2} +\frac{n}{2n-1} \nu_n(z)^{(n-1)^2}\right|\le 1\right\},
\]
$s=n^2$ and $n\ge 2$.

First, the triangle inequality shows that 
$
 \left\{ z \in \Complex : 
\left| \nu_n(z) \right|\le 1\right\}\subset  \Ss_s;
$
the set $\left\{ z \in \Complex : 
\left| \nu_n(z) \right|\le 1\right\}$ appears as the unit disk in Figure \ref{fig:ssp3fig}.
\begin{figure}
\begin{center}
\includegraphics[width=0.45\textwidth]{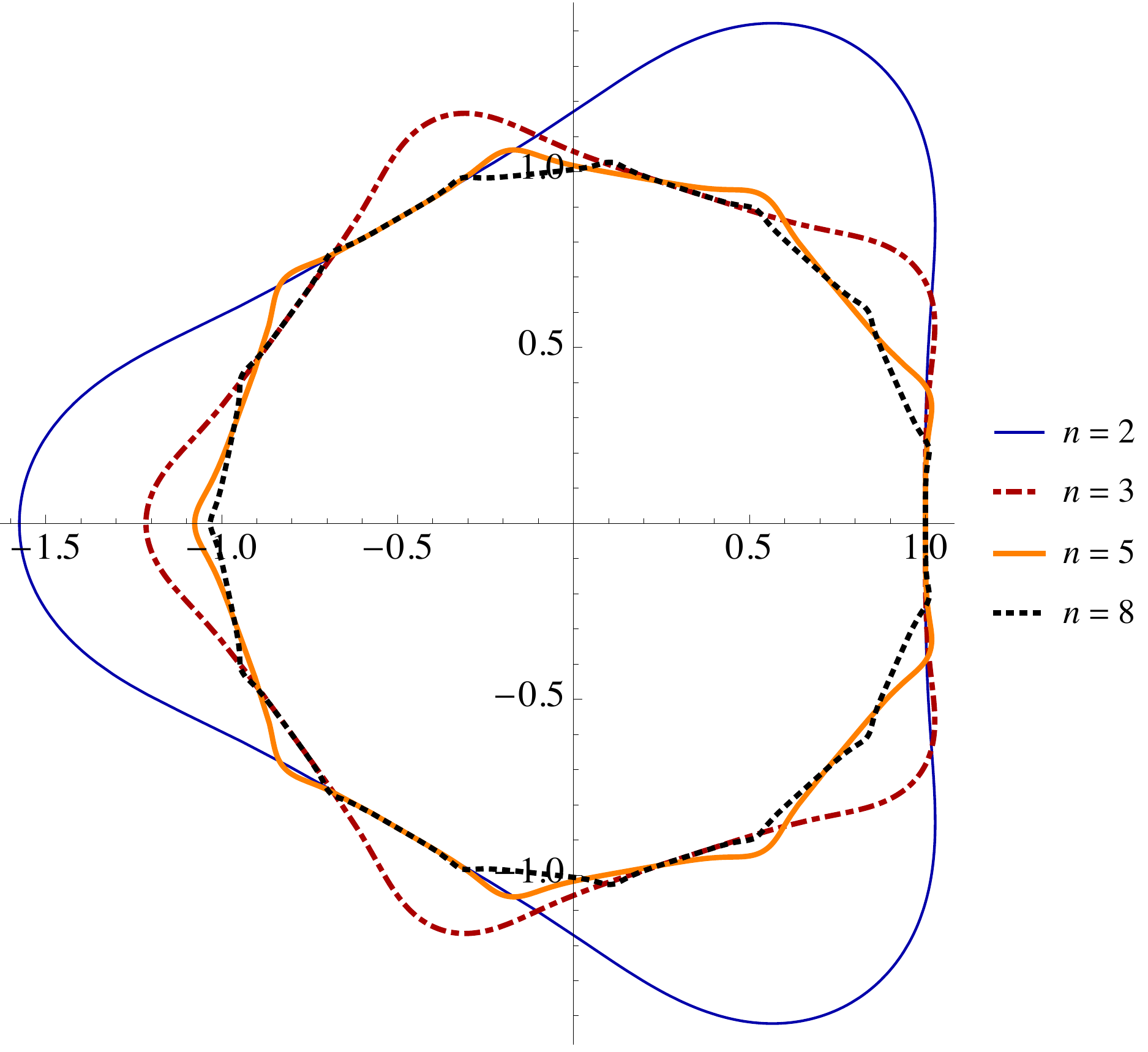}
\caption{Scaled and shifted stability regions (see (\ref{SSP3scalingdef}) and (\ref{snscdef})) 
of the optimal third order SSP RK methods with $n^2$ stages 
for $n\in\{2, 3, 5, 8\}$. The boundary curves on the plot are given by
$
\left\{ \nu \in \Complex : 
\left|\frac{n-1}{2n-1} \nu^{n^2} +\frac{n}{2n-1} \nu^{(n-1)^2}\right|= 1\right\}.
$
For any $n$, the
corresponding region has $(2n-1)$-fold rotational symmetry. 
\label{fig:ssp3fig}}
\end{center}
\end{figure}
By taking into account the explicit forms of the $Q_j$ polynomials provided by Lemma 
\ref{ssp3stabpolydetermination}, we see, again by the triangle inequality, that
$|Q_j(z)|\le 1$ for all $z\in\Complex$ with $|\nu_n(z)|\le 1$.
Hence the following corollary is established. 
\begin{cor} 
We have
\[
\M_s^{\mathrm{SSP3}}(\alpha,\beta,D_\sspcoeff) \le 1,
\] 
where $\sspcoeff = n^2-n$ and $D_\sspcoeff$ is the disk defined in \eqref{eq:disk}.
\end{cor}

\begin{rem} The above corollary again corresponds to a 
special case of Theorem \ref{SSPdiskinternalstability}
 (cf. Remark \ref{remark33}).
\end{rem}

As a consequence, it is enough to bound $|Q_j|$ only 
on the ``petals'' in Figure \ref{fig:ssp3fig}. For $2 \leq j \leq m_n$, $z\in\Ss_s$ and $|\nu_n(z)|\ge 1$ we have
\[
|Q_{j}(z)|=\left|\dfrac{n-1}{2n-1} \nu_n(z)^{n^2-j+1} +
\dfrac{n}{2n-1} \nu_n(z)^{(n-1)^2-j+1}\right|=\]
\[
\left|\nu_n(z) \right|^{-j+1} \left|\dfrac{n-1}{2n-1} \nu_n(z)^{n^2} +
\dfrac{n}{2n-1} \nu_n(z)^{(n-1)^2}\right|\le 1\cdot 1,
\]
by using the definition of $\Ss_s$ and $|\nu_n(z)|\ge 1$. We remark that this $2 \leq j \leq m_n$ case can occur only for $n\ge 3$, since $m_2=1$.

The above observations mean that
\[
 \M_s^{\mathrm{SSP3}}(\Ss_s)\equiv \max_{j=2,\ldots, s} \sup_{z\in \Ss_s} |Q_{j}(z)|=
\max\left\{1,\max_{j=2,\ldots, s} \sup_{z\in \Ss_s, |\nu_n(z)|\ge 1} |Q_{j}(z)|\right\}=
\]
\[
\max\left\{1,\max_{j=m_n+1,\ldots, s} \sup_{z\in \Ss_s, |\nu_n(z)|\ge 1} |Q_{j}(z)|\right\}=\]
\[
\max\left\{1,\max_{j=m_n+1,\ldots, k_n-1} \sup_{z\in \Ss_s, |\nu_n(z)|\ge 1} \left| Q_{j}(z)\right|,
\max_{j=k_n,\ldots, n^2} \sup_{z\in \Ss_s, |\nu_n(z)|\ge 1} \left|Q_{j}(z)\right|\right\}=
\]
\[
\max\left\{1,\max_{j=m_n+1,\ldots, k_n-1} \sup_{z\in \Ss_s, |\nu_n(z)|\ge 1} \left| \dfrac{n-1}{2n-1}  \nu_n(z)^{n^2-j+1} \right|,
\max_{j=k_n,\ldots, n^2} \sup_{z\in \Ss_s, |\nu_n(z)|\ge 1} \left|\nu_n(z)^{n^2-j+1} \right|\right\}=
\]
\[
\max\left\{1,\sup_{z\in \Ss_s, |\nu_n(z)|\ge 1} \left| \dfrac{n-1}{2n-1}  \nu_n(z)^{n^2-m_n} \right|,
\sup_{z\in \Ss_s, |\nu_n(z)|\ge 1} \left|\nu_n(z)^{n^2-k_n+1} \right|\right\}=
\]
\[
\max\left\{1,\sup_{z\in \Ss_s, |\nu_n(z)|\ge 1} \left| \dfrac{n-1}{2n-1}  \nu_n(z)^{(n^2+3n-4)/2} \right|,
\sup_{z\in \Ss_s, |\nu_n(z)|\ge 1} \left|\nu_n(z)^{(n^2-n)/2} \right|\right\}.
\]
Thus, by defining 
\begin{equation}\label{snscdef}
\Ss_n^\mathrm{sc}:=\left\{ \nu\in\Complex : \left|\frac{n-1}{2n-1} \nu^{n^2} +
\frac{n}{2n-1} \nu^{(n-1)^2}\right|\le 1\right\}
\end{equation}
to be the scaled and shifted absolute stability region, and
\[
\nu_n^*:=\sup_{\nu\in \Ss_n^\mathrm{sc}} |\nu|,
\]
we have $\nu_n^*\ge 1$ (as shown by $\nu=1$), and
\begin{equation}\label{MsSSP3upperintermsofnunstar}
\M_s^{\mathrm{SSP3}}(\Ss_s)= \max\left\{\dfrac{n-1}{2n-1}\left(\nu_n^*\right)^{(n^2+3n-4)/2},
\, \left(\nu_n^*\right)^{(n^2-n)/2} \right\}
\quad \quad (s=n^2,\, n\ge 2).
\end{equation}

\subsubsection{The real slice of the absolute stability region}\label{therealslicesubsection}

It is easily seen that $\nu\in \Ss_n^\mathrm{sc}$ implies 
\[
\nu \exp\left(\frac{2\pi i k}{2n-1}\right)\in \Ss_n^\mathrm{sc}\quad\quad (k=0, 1, \ldots),
\]
so the set $\Ss_n^\mathrm{sc}$ is rotationally symmetric. Therefore
\[
\nu_n^*\equiv \sup_{\nu\in \Ss_n^\mathrm{sc}} |\nu|=\sup \left\{|\nu| :  \nu\in \Ss_n^\mathrm{sc}, 
0\le \arg(\nu) \le \frac{2\pi}{2n-1}\right\}.
\]
By introducing polar coordinates $\nu=\varrho e^{i\varphi}$, and the real-valued function
\[
\mu_n(\varrho,\varphi):=
-1+\frac{\vr^{(n-1)^2} }{2 n-1}\sqrt{(n-1)^2 \vr^{4
   n-2}+2n (n-1)\vr^{2 n-1} \cos ((2 n-1 )\vf)+n^2}
\]
defined for all $\vr\ge 0$ and $0\le \vf<2\pi$, we rewrite $\Ss_n^\mathrm{sc}$ as
\[
\left\{ \vr e^{i\vf} : (\vr,\vf)\in [0,+\infty)\times [0,2\pi) ,  \mu_n(\vr,\vf)\le 0 \right\}.
\]
Hence
\begin{equation}\label{nuthereexistsphi}
\nu_n^*=
\sup \left\{\vr :  \vr \ge 0, \exists \vf \in \left[0,\frac{2\pi}{2n-1}\right]  
\textrm{ such that } \mu_n(\vr,\vf)\le 0 \right\}.
\end{equation}
But due to the fact that the range of 
$\left[0,\frac{2\pi}{2n-1}\right]\ni\vf\mapsto \cos((2n-1)\vf)$ is the same as the
range of 
$\left[0,\frac{\pi}{2n-1}\right]\ni\vf\mapsto \cos((2n-1)\vf)$, 
we can write $\exists \vf \in \left[0,\frac{\pi}{2n-1}\right]$ instead of
$\exists \vf \in \left[0,\frac{2\pi}{2n-1}\right]$ in (\ref{nuthereexistsphi}). 
Moreover, we have seen in the previous subsection
that $\nu_n^*\ge 1$, so $\vr\ge 0$ in (\ref{nuthereexistsphi}) can also be replaced by
$\vr \ge 1$. Let us make some more reductions. By defining
\[
\mu_n^*(\vr):=\mu_n\left(\vr,\frac{\pi}{2n-1}\right)=
-1+\frac{n \vr^{(n-1)^2} \left| 1-\left(1-\frac{1}{n}\right) \vr^{2 n-1}\right| }{2 n-1},
\]
see Figure \ref{fig:munstarrho}, we observe that for any $\vr\ge 1$ and $\vf\in \left[0,\frac{\pi}{2n-1}\right]$ 
we have $\mu_n^*(\vr)\le \mu_n(\vr,\vf)$, 
since the function $\left[0,\frac{\pi}{2n-1}\right]\ni\vf\mapsto \cos((2n-1)\vf)$ is decreasing. 
Therefore
\begin{equation}\label{munstartsimpler}
\nu_n^*=
\sup \left\{\vr :  \vr \ge 1, \mu_n^*(\vr)\le 0 \right\}.
\end{equation}

\begin{rem} By taking into account the rotational symmetry, 
the above equality expresses the geometrical fact that the farthest point of $\Ss_n^\mathrm{sc}$
from the origin occurs, for example, along the negative real slice of $\Ss_n^\mathrm{sc}$, 
that is,
\[
\nu_n^*\equiv \sup_{\nu\in \Ss_n^\mathrm{sc}} |\nu|=
 \sup_{\nu\in \Ss_n^\mathrm{sc}\cap (-\infty,-1]} |\nu|.
\]
\end{rem}

For any $n\ge 2$ and $\vr\ge 1$, let us introduce 
\begin{equation}\label{munminusdef}
\mu_n^-(\vr):=-1-\frac{n \vr^{(n-1)^2} \left( 1-\left(1-\frac{1}{n}\right) \vr^{2 n-1}\right) }{2 n-1},
\end{equation}
\[
\mu_n^+(\vr):=-1+\frac{n \vr^{(n-1)^2} \left( 1-\left(1-\frac{1}{n}\right) \vr^{2 n-1}\right)}{2 n-1}
\]
and
\[
\vr_n:=\left(1+\frac{1}{n-1}\right)^{\frac{1}{2 n-1}}>1.
\]
Then $\mu_n^*(\vr)=\mu_n^+(\vr)$ for $1\le\vr\le \vr_n$, and 
$\mu_n^*(\vr)=\mu_n^-(\vr)$ for $\vr_n\le\vr$.
\begin{figure}
\begin{center}
\includegraphics[width=0.5\textwidth]{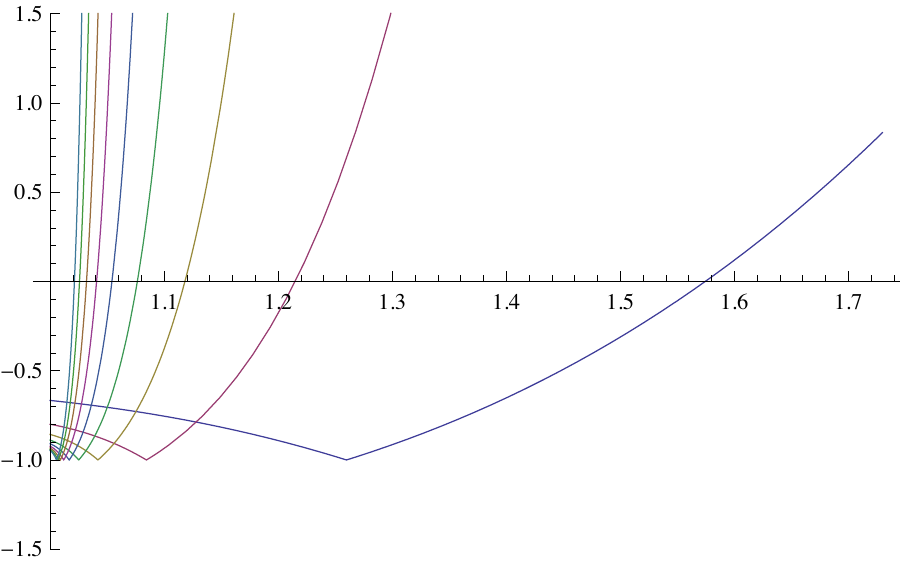}
\caption{The functions $\mu_n^*$ for 
$2\le n \le 10$. \label{fig:munstarrho}}
\end{center}
\end{figure}
Now observe that for $1\le\vr$
\[
\left(\mu_n^+\right)'(\vr)=
\frac{n^2 (n-1) \vr^{n(n-2)} \left(1-\frac{1}{n}-\vr^{2 n-1}\right)}{2 n-1}<0
\]
and $\mu_n^+(1)<0$, so $\mu_n^*<0$ on $[1,\vr_n]$. On the other hand, for
$1\le \vr$ we have
$\left(\mu_n^-\right)'(\vr)=-\left(\mu_n^+\right)'(\vr)$,
so $\mu_n^*$ is strictly increasing
on $[\vr_n,+\infty)$. Notice that $\mu_n^*(\vr_n)=\mu_n^-(\vr_n)=\mu_n^+(\vr_n)=-1$ and
$\displaystyle \lim_{+\infty} \mu_n^*=+\infty$.

By considering (\ref{munstartsimpler}) as well, the following lemma is thus established.
\begin{lem}\label{uniquerootlemma}
For each $n\ge 2$, the polynomial $\mu_n^-$ defined in (\ref{munminusdef}) has a unique zero in the interval
$[1,+\infty)$. Moreover, for any $\vr\ge 1$ we have
\[
\mu_n^-(\vr)<0 \iff \vr<\nu_n^*,
\]
\[
\mu_n^-(\vr)>0 \iff \vr>\nu_n^*,
\]
\[
\mu_n^-(\nu_n^*)=0,
\]
and \[\nu_n^*>\vr_n.\]
\end{lem}

\begin{rem}
For large $n$ values, we have $\vr_n= 1+\frac{1}{2n^2}+{\mathcal{O}}(\frac{1}{n^3})$.
\end{rem}

\subsubsection{Explicit estimates for $\nu_n^*$}\label{explicitestimatesnunstarsubsection}

\begin{lem}\label{nunstarfinersandwiching}
For any $n\ge 2$ we have
\[
1+\frac{\ln (n)}{n^2}-\frac{\ln (\ln (n))}{n^2}<\nu_n^*,
\]
while for any $n\ge 9$ we have
\begin{equation}\label{nunstarupperestimate}
\nu_n^*<1+\frac{\ln (n)}{n^2}-\frac{\ln (\ln (n))}{8n^2}=:\lambda_n.
\end{equation}
\end{lem}
\begin{proof} The proof of the lower estimate for $\nu_n^*$ is 
analogous to that of  (\ref{nunstarupperestimate}) and omitted here. The proof of  (\ref{nunstarupperestimate})
is broken into some simpler steps.\\

\noindent \textbf{Step 1.}  $n\ge 3$ $\implies$ $n^{\frac{1}{n}}>1+\frac{\ln (n)}{n}$.\\
\noindent\textbf{Proof.} $\left(n^{\frac{1}{n}}-1-\frac{\ln (n)}{n}\right)^\prime=
-\frac{1}{n^2}\left(n^{\frac{1}{n}}-1\right) (\ln (n)-1)<0$, and 
$\lim_{n\to +\infty}\left(n^{\frac{1}{n}}-1-\frac{\ln (n)}{n}\right)=0$.\\

\noindent\textbf{Step 2.} $n\ge 12\implies \mathit{LHS}:=
\left(1-\frac{1}{n}\right) \left(1+\frac{2}{n}\right) \left(1+\frac{\ln (n)}{n}\right)-1>
\frac{1+\ln(n)}{n}=:\mathit{RHS}$.\\
\noindent\textbf{Proof.} For $12\le n\le 18$, we check the statement directly. For $n \ge 19$,
\[
\left(   \mathit{LHS}- \mathit{RHS}\right)^\prime=
\frac{1}{n^4}(5 n-2 (n-3) \ln (n)-2)<0,
\]
and $\lim_{n\to +\infty} \left( \mathit{LHS}- \mathit{RHS}\right)=0$.\\

\noindent \textbf{Step 3.} $n\ge 17\implies \frac{1+\ln(n)}{\sqrt{\ln (n)+\frac{4}{5}}}>2$.\\
\noindent\textbf{Proof.} The inequality (quadratic in $\ln(n)$) is solved directly.\\

\noindent \textbf{Step 4.} $n\ge 2 \implies 
\frac{\ln (n)}{n^2}-\frac{\ln (\ln (n))}{8 n^2}>\frac{8 \ln (n)}{9 n^2}$.\\
\noindent\textbf{Proof.} Elementary.\\

\noindent \textbf{Step 5.} $n\ge 10 \implies \lambda_n^{2n-1}>
\left(1+\frac{2}{n}\right) n^{\frac{1}{n}}$.\\
\noindent\textbf{Proof.} For $10\le n \le 16$, the inequality is checked separately. So suppose
in the following that $n\ge 17$. We take logarithms of both sides, 
then the right-hand side is increased by using 
\begin{equation}\label{logupperestimate}
x>0 \implies \ln(1+x)<x,
\end{equation}
and the left-hand side is decreased by using Step 4. So it is enough to prove
\[
(2n-1)\ln\left(1+\frac{8 \ln (n)}{9 n^2}\right)>\frac{2}{n} +\frac{1}{n}\ln(n).
\]
We decrease the left-hand side further by applying
\begin{equation}\label{loglowerestimate}
x>0 \implies \ln(1+x)>x-\frac{x^2}{2}.
\end{equation}
After expanding the new left-hand side and omitting one of its positive terms, 
$\frac{32 \ln^2(n)}{81 n^4}$, 
it is enough to show that
$
-\frac{64 \ln^2(n)}{81 n^3}-\frac{8 \ln(n)}{9
   n^2}+\frac{16 \ln (n)}{9 n}>\frac{2}{n}+\frac{\ln(n)}{n}
$, that is,
\[
 \mathit{LHS}_1:=\frac{7 \ln(n)}{9}>2+\frac{8 \ln(n)}{9 n}+\frac{64 \ln^2(n)}{81 n^2}=: \mathit{RHS}_1.
\]
But $\mathit{LHS}_1- \mathit{RHS}_1>0$ at $n=17$; and $\mathit{LHS}_1$ is a monotone increasing, 
while $\mathit{RHS}_1$ is a monotone decreasing function.\\

\noindent \textbf{Step 6.} $n\ge 19   \implies \frac{221}{760} \ln (\ln(n))>\frac{2 \ln(n)}{n},\quad$
$\mathit{LHS}_2:=\frac{13 \ln (\ln (n))}{4000}>\frac{\ln^2(n)}{2 n^4},\quad$
$4\mathit{LHS}_2>\frac{\ln^2(n)}{2 n^2},\quad$
$\mathit{LHS}_2>\frac{\ln(\ln(n))}{8 n^2},\quad$
$\mathit{LHS}_2>\frac{\ln (n) \ln (\ln (n))}{4 n^3},\quad$
$\mathit{LHS}_2>\frac{\ln^2(\ln(n))}{128 n^4},\quad$
$\mathit{LHS}_2>\frac{\ln^2(\ln(n))}{128 n^2}$.\\
\noindent\textbf{Proof.} All inequalities are true for $n=19$. 
On the other hand, the functions appearing on the right-hand sides of the 7 inequalities are all
monotone decreasing, while the functions on the left-hand sides are all increasing. \\

\noindent \textbf{Step 7.} $n\ge 19  \implies$
\[
\frac{9}{20} \ln (\ln (n))>\frac{2
   \ln(n)}{n}+\frac{\ln^2(n)}{2 n^4}+\frac{\ln^2(n)}{2
   n^2}+\frac{\ln(\ln(n))}{8 n^2}+\]
\[
\frac{ \ln(n)\ln(\ln(n))}{4 n^3}+\frac{\ln^2(\ln
   (n))}{128 n^4}+\frac{\ln^2(\ln(n))}{128 n^2}+\frac{1}{8} \ln(\ln(n)).
\]
\noindent\textbf{Proof.}  Add the 7 inequalities presented in Step 6, and use the fact that
\[
\frac{9}{20}-\frac{1}{8}>\frac{221}{760}+\frac{13}{4000}+\frac{13}{1000}+\frac{13}{4000}+\frac{13}{4000}+
\frac{13}{4000}+\frac{13}{4000}.
\]\smallskip

\noindent \textbf{Step 8.} $n\ge 3  \implies 
\mathit{LHS}_3:=\frac{1}{2} \ln \left(\ln (n)+\frac{4}{5}\right)>\frac{9}{20} \ln (\ln(n))=:\mathit{RHS}_3.$\\
\noindent\textbf{Proof.} $\mathit{LHS}_3-\mathit{RHS}_3$ has a global minimum at $n=e^{36/5}\approx
1339.43$ with value $\frac{3}{20} \ln \left(\frac{2000}{729}\right)>0.15$.\\

\noindent \textbf{Step 9.}  $n\ge 10  \implies 
\lambda_n^{(n-1)^2}>\frac{n}{\sqrt{\ln \left(n\right)+\frac{4}{5}}}$.\\
\noindent\textbf{Proof.} We verify the inequality directly for $10\le n \le 18$, so we can suppose
 $n\ge 19$. After taking logarithms of both sides and applying (\ref{loglowerestimate}) 
to decrease the outer logarithm on the left-hand side, we arrive at a stronger inequality
\[
(n-1)^2 \left(\frac{\ln (n)}{n^2}-\frac{\ln (\ln (n))}{8
   n^2}-\frac{1}{2} \left(\frac{\ln (n)}{n^2}-\frac{\ln (\ln (n))}{8
   n^2}\right)^2\right)>\]
\[
\ln (n)-\frac{1}{2} \ln \left(\ln(n)+\frac{4}{5}\right).
\]
We expand the left-hand side here, cancel the common $\ln(n)$ term on both sides, then 
regroup the resulting inequality to get
\[
\frac{1}{2} \ln \left(\ln(n)+\frac{4}{5}\right)+
\frac{\ln (n)\ln (\ln (n)) }{8 n^4}+\frac{\ln^2(n)}{n^3}+\]
\[\frac{\ln
   ^2(\ln(n))}{64 n^3}+\frac{\ln(n)\ln(\ln(n)) }{8 n^2}+\frac{\ln
   (n)}{n^2}+\frac{\ln (\ln (n))}{4 n}>
\]
\[
\frac{2
   \ln(n)}{n}+\frac{\ln^2(n)}{2 n^4}+\frac{\ln^2(n)}{2
   n^2}+\frac{\ln(\ln(n))}{8 n^2}+
\]
\[
\frac{ \ln(n)\ln(\ln(n))}{4 n^3}+\frac{\ln^2(\ln
   (n))}{128 n^4}+\frac{\ln^2(\ln(n))}{128 n^2}+\frac{1}{8} \ln(\ln(n)).
\]
Now we decrease the left-hand side further by omitting all 6 terms with 
the exception of $\frac{1}{2} \ln\left(\ln(n)+\frac{4}{5}\right)$.
The remaining inequality is true due to Step 8 and Step 7.\\

\noindent \textbf{Step 10.}  $n\ge 10  \implies$
\[
\lambda_n^{(n-1)^2} \left( \left(1-\frac{1}{n}\right) \lambda_n^{2 n-1}-1\right) >2.
\]
\noindent\textbf{Proof.} For $16\ge n \ge 10$, the inequality holds
because $2.92>2.80>2.68>2.55>2.41>2.26>2.09>2$. So suppose that $n\ge 17$.
Then by applying Step 5, Step 1 and Step 2, we get\\
\[
\left(1-\frac{1}{n}\right) \lambda_n^{2 n-1}-1>
\left(1-\frac{1}{n}\right)\left(1+\frac{2}{n}\right) n^{\frac{1}{n}}-1>
\]
\[
\left(1-\frac{1}{n}\right)\left(1+\frac{2}{n}\right) \left(  1+\frac{\ln (n)}{n}\right)-1>
\frac{1+\ln(n)}{n}>0,
\]
so we can apply this with Step 9 and Step 3, and obtain
\[
\lambda_n^{(n-1)^2} \left( \left(1-\frac{1}{n}\right) \lambda_n^{2 n-1}-1\right)>
\frac{n}{\sqrt{\ln \left(n\right)+\frac{4}{5}}} \cdot \frac{1+\ln(n)}{n}>2.
\]\smallskip

\noindent \textbf{Step 11.}  $n\ge 9  \implies \mu_n^-(\lambda_n)>0$, with $\mu_n^-$ defined in
(\ref{munminusdef}).\\
\noindent\textbf{Proof.} 
\[
\mu_n^-(\lambda_n)=
-1-\frac{n \lambda_n^{(n-1)^2} \left( 1-\left(1-\frac{1}{n}\right) \lambda_n^{2 n-1}\right) }{2 n-1}>0
\]
is equivalent to
\[
\lambda_n^{(n-1)^2} \left(\left(1-\frac{1}{n}\right) \lambda_n^{2 n-1}-1\right)>2-\frac{1}{n} .
\]
The statement is true for $n=9$, because $1.91>2-\frac{1}{9}$. For $n\ge 10$, we apply Step 10.\\

\noindent \textbf{Step 12.}  Step 4 implies $1< \lambda_n$, so 
by virtue of Lemma \ref{uniquerootlemma} and Step 11, the proof is complete.
\end{proof}

\begin{rem}
One can ask whether the starting index $9$ in
$n\ge 9$ in (\ref{nunstarupperestimate})  can be decreased if, for example,  
the  coefficient $\frac{1}{8}$ on the right-hand side of (\ref{nunstarupperestimate}) 
is replaced by a smaller positive number. However,  
$n\ge 8$ is necessary even for the weaker $\nu_n^*<1+\frac{\ln (n)}{n^2}$ inequality.
\end{rem}

\begin{rem}\label{remark315heuristicasymptotics}
Let us give a heuristic asymptotic approximation to $\nu_n^*$ that 
sheds some light on the origin of Lemma \ref{nunstarfinersandwiching}.
We will use $\ln (1+x)\approx x$ and $e^{x}\approx 1+x$ for small $|x|$. 
Let us take the defining equation for $-\nu_n^*$ 
\[
\left|\frac{(n-1) (-\nu_n^*)^{n^2}}{2 n-1}+\frac{n (-\nu_n^*)^{(n-1)^2}}{2 n-1}\right|-1=0
\]
and rearrange as
\begin{equation}\label{heuristicnun*}
\nu_n^*=\left(\frac{2 n-1}{|n (-\nu_n^*)^{1-2 n}+n-1|}\right)^{1/n^2}.
\end{equation}
Since $-\nu_n^*\approx -1$, we have
\[
\nu_n^*\approx \left(\frac{2 n-1}{|n (-1)^{1-2 n}+n-1|}\right)^{1/n^2}=(2n-1)^{1/n^2}.
\]
But then 
\[\nu_n^*\approx \left(2 n-1\right)^{1/n^2}=
\exp\left({\frac{\ln \left(1-\frac{1}{2 n}\right)+\ln (2 n)}{n^2}}\right)\approx 
\exp\left({\frac{-\frac{1}{2 n}+\ln (2 n)}{n^2}}\right)=\]
\[
\exp\left(-\frac{1}{2 n^3}+\frac{\ln 2 }{n^2}+\frac{\ln ( n)}{n^2}\right)\approx 
\exp\left(\frac{\ln ( n)}{n^2}\right)\approx 1+\frac{\ln ( n)}{n^2}.\]
Substituting this approximation into (\ref{heuristicnun*}) yields 
\[
\nu_n^*\approx 
\left(\frac{2 n-1}{\left|n \left(-1-\frac{\ln(n)}{n^2}\right)^{1-2 n}+n-1\right|}\right)^{1/n^2}= 
\left(\frac{2 n-1}{n\left(1-\left(1+\frac{\ln ( n)}{n^2}\right)^{1-2 n}-\frac{1}{n}\right)}\right)^{1/n^2}=\]
\[
\exp\left(\frac{\ln\left(2-\frac{1}{n}\right)}{n^2}-\frac{1}{n^2}\ln\left(1-\left(1+\frac{\ln ( n)}{n^2}\right)^{1-2 n}-\frac{1}{n}\right)\right)\approx
\]
\[
\exp\left(-\frac{1}{n^2}\ln\left(1-\frac{1}{n}-\left(1+\frac{\ln ( n)}{n^2}\right)^{1-2 n}\right)\right)=
\]
\[
\exp\left(-\frac{1}{n^2}\ln\left(1-\frac{1}{n}-\exp\left((1-2n)\ln\left(1+\frac{\ln ( n)}{n^2}\right)\right)\right)\right)\approx
\]
\[
\exp\left(-\frac{1}{n^2}\ln\left(1-\frac{1}{n}-\exp\left((1-2n)\frac{\ln ( n)}{n^2}\right)\right)\right)=
\]
\[
\exp\left(-\frac{1}{n^2}\ln\left(1-\frac{1}{n}-\exp\left(\frac{\ln ( n)}{n^2}-\frac{2\ln ( n)}{n}\right)\right)\right)\approx
\]
\[
\exp\left(-\frac{1}{n^2}\ln\left(1-\frac{1}{n}-1-\frac{\ln ( n)}{n^2}+\frac{2\ln ( n)}{n}\right)\right)\approx
\exp\left(-\frac{1}{n^2}\ln\left(\frac{2\ln ( n)}{n}\right)\right)=
\]
\[
\exp\left(-\frac{\ln 2}{n^2}-\frac{1}{n^2}\ln\left(\frac{\ln ( n)}{n}\right)\right)\approx
\exp\left(-\frac{1}{n^2}\ln\left(\frac{\ln ( n)}{n}\right)\right)=
\]
\[
\exp\left(\frac{\ln(n)-\ln \ln(n)}{n^2}\right)\approx 1+\frac{\ln(n)-\ln \ln(n)}{n^2}=1+\frac{\ln(n)}{n^2}-\frac{\ln \ln(n)}{n^2}.
\]
By iterating this process further, we can get finer and finer asymptotic estimates. For example,
\[
\nu_n^* \approx 
1+\frac{\ln(n)}{n^2}-\frac{\ln(\ln(n))}{n^2}+\frac{\ln(\ln
   (n))}{n^2 \ln(n)}-\frac{\ln(\ln(n))}{n^2 \ln^2(n)}.
\]
\end{rem}

\begin{rem}\label{spikypolyremark}
It can be shown that the sequence $\nu_n^*$ is strictly decreasing for $n\ge 2$. During the proof
of this statement we discovered the following interesting family of polynomials. Let us choose
and fix an arbitrary integer $k\ge 1$, and consider the function
\[
\Real \ni x\mapsto (4 k^2-k)x^{8 k}+(8 k^2-1)x^{4 k-1}+4 k^2+k.
\]
It can be shown that the above $3$-term polynomial first strictly decreases, has a unique and 
positive global minimum, then strictly increases. 
 However, as $k$ is increased, the ``spike'' near $x=-1$ becomes narrower: 
Figure \ref{spikypoly} shows a typical polynomial of this class
on 3 scales.
Hence members of this family can be used, for example, as arbitrarily hard
test examples---with the same, simple structure---for numerical optimizers or solvers.
\end{rem}

\begin{figure}
\begin{center}
\subfigure[Restriction to the interval $(-1.3,1.3)$]{
\includegraphics[width=0.4\textwidth]{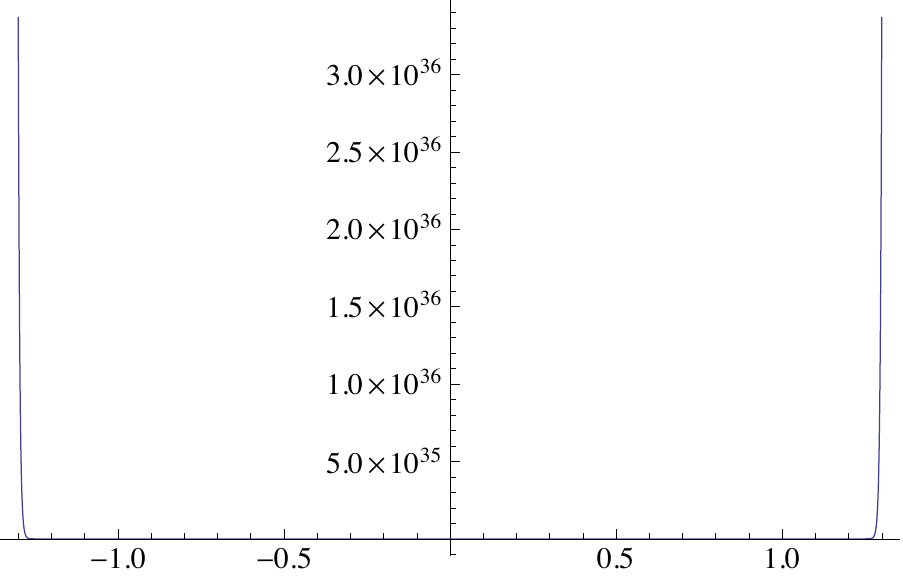}}
\subfigure[In the window $(-1.3,1.3)\times (0,6000)$]{
\includegraphics[width=0.4\textwidth]{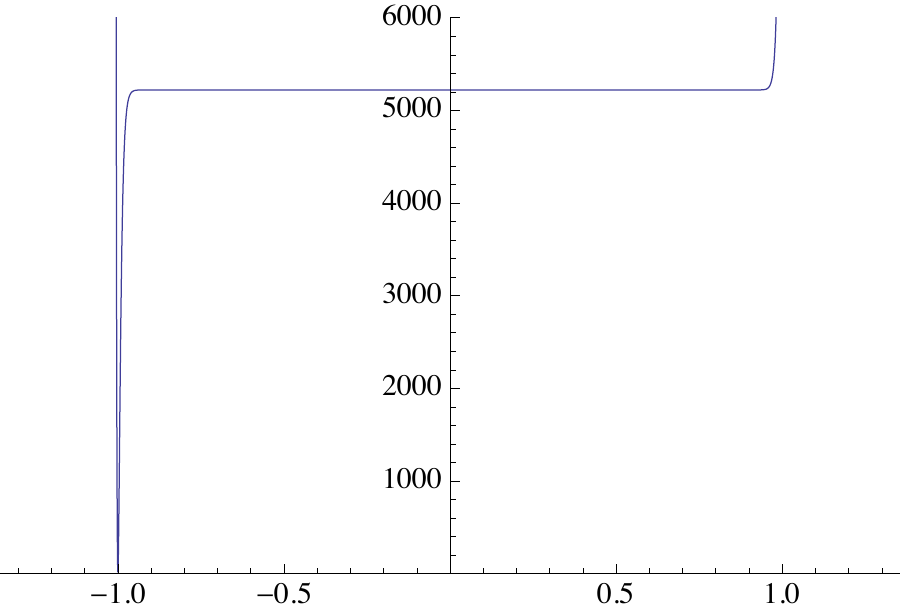}}
\subfigure[And around the global minimum]{
\includegraphics[width=0.4\textwidth]{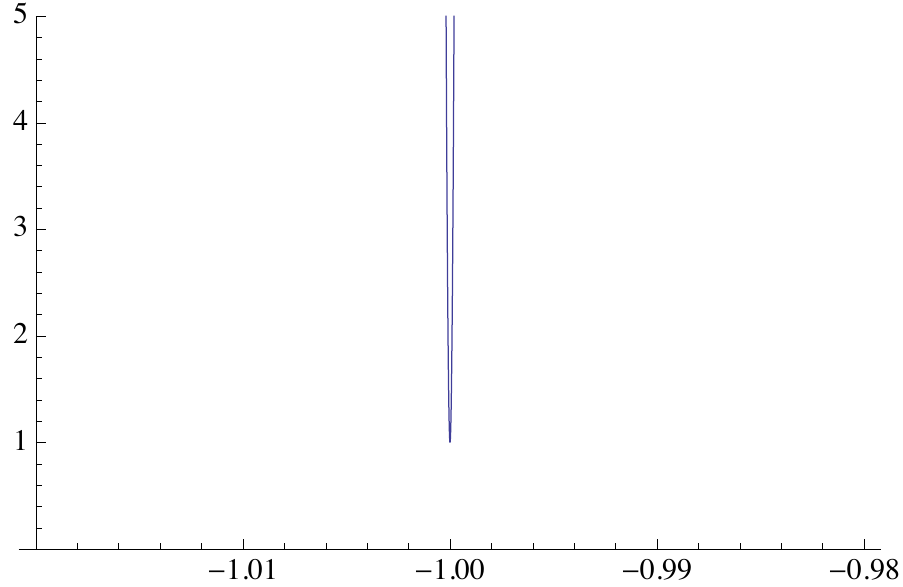}}
\caption{The polynomial $5148 x^{288}+10367 x^{143}+5220$ 
from Remark \ref{spikypolyremark} with $k=36$.
\label{spikypoly}}
\end{center}
\end{figure}

\subsubsection{Estimating $\M_s^{\mathrm{SSP3}}(\Ss_s)$}\label{estimatingMssp3subsection}

The entries in Table \ref{exactMsSSP3values} can be obtained by 
representing $\nu_n^*$ as an algebraic number given by Lemma \ref{uniquerootlemma},  
approximating it to any desired precision 
(from below or above) with \textit{Mathematica}, then substituting that value into
(\ref{MsSSP3upperintermsofnunstar}). From this we also verify that for $2\le n\le 10$
we have the simpler representation
(\ref{theorem35nunstarrepresentation}) instead of (\ref{MsSSP3upperintermsofnunstar}).
To conclude Section \ref{optimalthirdorderSSPmethods}, it is therefore enough to prove
(\ref{theorem35nunstarrepresentation}) and the chain of inequalities in Theorem \ref{thm:SSP3} for $n\ge 9$. 

\begin{lem}\label{MsSSP3withoutmax}
For any $n\ge 9$ and $s=n^2$, we have
\[
\M_s^{\mathrm{SSP3}}(\Ss_s)= \left(\nu_n^*\right)^{(n^2-n)/2}.
\]
\end{lem}
\begin{proof} Since the inequality
\[
(\nu_n^*)^{\left(n^2-n\right)/2}>\frac{n-1}{2n-1}(\nu_n^*)^{\left(n^2+3 n-4\right)/2}
\]
is equivalent to 
\[
\nu_n^*<\left(\frac{2 n-1}{n-1}\right)^{1/(2n-2)},
\]
by taking into account (\ref{nunstarupperestimate}), it is enough to show that for $n\ge 9$ we
have
\[
1+\frac{\ln (n)}{n^2}-\frac{\ln (\ln (n))}{8n^2}\le \left(\frac{2 n-1}{n-1}\right)^{1/(2n-2)}.
\]
By using $1+x\le e^x$ ($x\in\Real$) and $\frac{2 n-1}{n-1}>2$, it is sufficient to verify that
$
\exp\left(\frac{\ln (n)}{n^2}-\frac{\ln (\ln (n))}{8n^2}\right)\le 
2^{1/(2n-2)}
$, or, in other words, that
\[
-\frac{2 \ln (n)}{n^2}+\frac{\ln (\ln (n))}{4 n^2}+\frac{2 \ln (n)}{n}-\frac{\ln (\ln (n))}{4
   n}\le \ln (2).
\]
But $\frac{\ln (\ln (n))}{4 n^2}-\frac{\ln (\ln (n))}{4n}<0$, so a stronger inequality is
$
-\frac{2 \ln (n)}{n^2}+\frac{2 \ln (n)}{n}<\ln (2),
$
which is equivalent to $0<n^2 \ln (2)-2 (n-1) \ln(n)$. The proof is finished by noting that the function on this right-hand side
is monotone increasing and its value at $n=9$ is positive.
\end{proof}

Now by combining Lemma \ref{MsSSP3withoutmax} and Lemma \ref{nunstarfinersandwiching},
we have proved 
\[
\left(1+\frac{\ln (n)}{n^2}-\frac{\ln (\ln (n))}{n^2}\right)^{\left(n^2-n\right)/2}<\M_{n^2}^{\mathrm{SSP3}}<\left(1+\frac{\ln (n)}{n^2}-\frac{\ln (\ln (n))}{8 n^2}\right)^{\left(n^2-n\right)/2}
\]
in  Theorem \ref{thm:SSP3}. The remaining two auxiliary inequalities in 
Theorem \ref{thm:SSP3} are shown below.

\begin{lem}\label{auxiliaryLemma3.18} For any $n\ge 9$ we have
\begin{equation}\label{lemma3181}
\left(1+\frac{\ln (n)}{n^2}-\frac{\ln (\ln (n))}{8 n^2}\right)^{\left(n^2-n\right)/2}<\frac{\sqrt{n}}{\sqrt[16]{\ln (n)}}
\end{equation}
and
\begin{equation}\label{lemma3182}
\frac{9}{10} \sqrt{\frac{n}{\ln (n)}}<\left(1+\frac{\ln (n)}{n^2}-\frac{\ln (\ln (n))}{n^2}\right)^{\left(n^2-n\right)/2}.
\end{equation}
\end{lem}
\begin{proof}
We estimate the left-hand side of (\ref{lemma3181}) as follows:
\[
\left( 1+\frac{\ln (n)}{n^2}-\frac{\ln (\ln (n))}{8n^2} \right)^{(n^2-n)/2}<\left( 1+\frac{\ln (n)}{n^2}-\frac{\ln (\ln (n))}{8n^2} \right)^{n^2/2}\le\]
\[ 
\exp\left(\frac{n^2}{2}\left(\frac{\ln (n)}{n^2}-\frac{\ln (\ln (n))}{8n^2} \right)\right) =\frac{\sqrt{n}}{\sqrt[16]{\ln (n)}},
\]
by taking into account that $1+\frac{\ln (n)}{n^2}-\frac{\ln (\ln (n))}{8n^2}>1$, and $1+x\le e^x$ again. As for the right-hand side of (\ref{lemma3182}), we write it as
$\exp\left(\frac{n^2-n}{2} \ln\left( 1+\frac{\ln (n)}{n^2}-\frac{\ln (\ln (n))}{n^2}\right)\right)$, then use $1+\frac{\ln (n)}{n^2}-\frac{\ln (\ln (n))}{n^2}>1$ with  (\ref{loglowerestimate}) to get
\[
\left(1+\frac{\ln (n)}{n^2}-\frac{\ln (\ln (n))}{n^2}\right)^{\left(n^2-n\right)/2}>\]
\begin{equation}\label{eansqrtnlnn}
\exp\left(\frac{n^2-n}{2} \left(\frac{\ln (n)}{n^2}-\frac{\ln (\ln (n))}{n^2}-\frac{1}{2} \left(\frac{\ln (n)}{n^2}-\frac{\ln (\ln(n))}{n^2}\right)^2\right)\right)=
e^{A_n}\cdot\sqrt{\frac{n}{\ln (n)}},
\end{equation}
with
\begin{equation}\label{Andefn}
A_n:=-\frac{1}{4 n^3}  [\ln (n)-\ln (\ln (n))] \left(2 n^2+(n-1) [\ln (n)- \ln(\ln (n))]\right).
\end{equation}
Since $e^{A_9}>e^{A_7}>\frac{9}{10}$, the proof will be completed as soon as we have shown
that $-A_n$ is strictly decreasing for $n\ge 7$. The derivative of $-A_n$ with respect to $n$ 
is given by
$
\frac{1}{{4 n^4 \ln(n)}}(a_n n^2+b_n n +c_n),
$
where \[a_n:= -2 \ln (n) [\ln(n)-\ln (\ln (n))-1]-2\]
\[b_n:=-2 \ln (n) \left[\ln ^2(n)-\ln
   (n)+1\right]-2 \ln(\ln (n)) \left[-2 \ln ^2(n)+\ln (\ln (n)) \ln (n)+\ln
   (n)-1\right]\]
and
\[
c_n:=\ln (n) \left[3 \ln ^2(n)-2 \ln (n)+2\right]+\ln (\ln (n)) \left[-6 \ln
   ^2(n)+3 \ln (\ln (n)) \ln (n)+2 \ln (n)-2\right].
\]
From these expressions one can prove that $a_n n^2+b_n n +c_n<0$ for $n\ge 7$; the 
elementary details here are omitted.
\end{proof}

\begin{rem}\label{section3closingremark}
Trivially, $A_n\to 0$ ($n\to +\infty$) with $A_n$ defined in (\ref{Andefn}), so by (\ref{eansqrtnlnn}) they imply
\[
\frac{\M_{n^2}^{\mathrm{SSP3}}}{\sqrt{\frac{n}{\ln (n)}}}>
\frac{\left(1+\frac{\ln (n)}{n^2}-\frac{\ln (\ln (n))}{n^2}\right)^{\left(n^2-n\right)/2}}{\sqrt{\frac{n}{\ln (n)}}}>e^{A_n}\to 1 \quad (n\to +\infty),
\]
in other words, 
\[
\liminf_{n\to +\infty} \frac{\M_{n^2}^{\mathrm{SSP3}}}{\sqrt{\frac{n}{\ln (n)}}}\ge 1.
\]
\end{rem}

\section{Internal amplification factors of extrapolation methods\label{sec:extrap}}
In this section we give values, estimates, and bounds for the maximum internal
amplification factor  
(\ref{Mmaximalinternalamplificationfactor}) for two classes of extrapolation methods:  
explicit Euler (EE) extrapolation and explicit midpoint (EM) extrapolation, both of which 
can be interpreted as explicit Runge--Kutta methods.  
We write $\MEEp$ ($\MEMp$) to denote the maximum internal amplification factor of the
order $p$ explicit Euler (explicit midpoint) extrapolation method.
We first summarize the structure and the main results of Section \ref{sec:extrap}. 

The extrapolation
algorithms are given as \textbf{Algorithm \ref{alg:extrap}} and
\textbf{Algorithm \ref{alg:extrapm}} in Section \ref{subsection41theextrapolationalgorithm}. Then
the internal stability polynomials are
described in Section \ref{section41InternalstabilityPolynomials}. 
Section \ref{extrapolshapeofSp} gives information on the absolute stability
region $\Ss_p$. The internal stability polynomials are estimated on 
certain subsets of $\Ss_p$ in Section \ref{subsection43boundsonMpEE} and 
Section \ref{subsection43boundsonMpEM}; first over all of $\Ss_p$.
We know that $\Ss_p$ may include regions far into the right
half-plane.  In order to achieve tighter practical bounds, we also focus on the
value of the internal stability polynomials over $\Ss_p \cap \Cleft$, where $\Cleft$ denotes 
the set of complex numbers with negative or zero
real part. To understand the internal stability of the methods with very small step sizes, we investigate a third quantity,  
$\M^{\mathrm{EE}}_p(\{0\})$ and $\M^{\mathrm{EM}}_p(\{0\})$
as well. 

Using the explicit expressions for the internal stability polynomials, one can employ \textit{Mathematica} to
compute exact values of $\M_p({\mathcal{S}})$ for small to moderate values of $p$,
for both extrapolation methods, and for the sets ${\mathcal{S}}=\Ss_p$, ${\mathcal{S}}=\Ss_p\cap \Cleft$ or ${\mathcal{S}}=\{0\}$. Decimal approximations of these exact values are given in 
several tables. The values for Euler extrapolation
corroborate the behavior observed with this method in earlier sections.
The values for midpoint extrapolation show that it has much better internal
stability. Then we give upper (and sometimes lower) bounds on the amplification factors for arbitrary $p$ values. For $p$ sufficiently large, some improved estimates are also provided. We indicate certain conjectured asymptotically optimal growth rates as well.
Although neither of the proved upper bounds is tight, these theorems again suggest
that midpoint extrapolation is much more internally stable than EE extrapolation.

\subsection{The extrapolation algorithm}\label{subsection41theextrapolationalgorithm}

The extrapolation algorithm (\cite[Section II.9]{Hairer2010}) consists of two parts. 
In the first part, a \textit{base method} and a \textit{step-number sequence} are chosen.  
In general, the step-number sequence is a strictly monotone increasing sequence
of positive integers 
$n_j$ ($j=1,2, \ldots$).  The base method is applied to compute multiple approximations to the ODE 
solution at time $t_{n+1}$ based on the solution value at $t_n$: the first approximation, denoted 
by $T_{1,1}$, is obtained by dividing the interval $t_{n+1}-t_n$ into $n_1$ step(s), the second 
approximation, $T_{2,1}$ is obtained by dividing it into $n_2$ steps, and so on. In the second 
part of the extrapolation algorithm, these $T_{m,1}$ values (that is, the low-order approximations) 
are combined by using the \textit{Aitken--Neville interpolation formula} to get a higher-order
approximation to the ODE solution at time $t_{n+1}$.

Here, the base method is the explicit 
Euler method (\textbf{Algorithm \ref{alg:extrap}}) or the explicit midpoint rule
(\textbf{Algorithm \ref{alg:extrapm}}), and the step-number sequence is chosen to be the 
harmonic sequence $n_j:=j$, being the ``most economic'' (\cite[formula (9.8)]{Hairer2010}). 
In the EM case, we also assume that the desired order of
accuracy $p$ is even.  We do not consider the effects of smoothing \cite{Hairer2010}.

\begin{algorithm}\caption{explicit Euler extrapolation}
\label{alg:extrap}
\begin{algorithmic}

\For{$m = 1 \to p$}  \Comment{Compute first order approximations}
    \State $\y_{m,0} = \un$
    \For{$j=1 \to m$}
        \State $\y_{m,j} = \y_{m,j-1} + \frac{\dt}{m}\f(\y_{m,j-1})$
    \EndFor
    \State $T_{m,1} = \y_{m,m}$
\EndFor

\For{$k=2 \to p$}  \Comment{Extrapolate to get higher order}
    \For{$j=k \to p$}
        \State $T_{j,k} = T_{j,k-1} + \frac{T_{j,k-1}-T_{j-1,k-1}}{\frac{j}{j-k+1}-1}$
        \Comment{Aitken--Neville formula for extrapolation to order $k$}
    \EndFor
\EndFor
\State $\unone = T_{p,p}$ \Comment{New solution value}
\end{algorithmic}
\end{algorithm}

\begin{algorithm}\caption{explicit midpoint extrapolation}
\label{alg:extrapm}
\begin{algorithmic}
\State $r=\frac{p}{2}$
\For{$m = 1 \to r$}  \Comment{Compute second-order approximations}
    \State $\y_{m,0} = \un$
    \State $\y_{m,1} = \y_{m,0} + \frac{\dt}{2m}\f(\y_{m,0})$ \Comment{Initial Euler step}
    \For{$j=2 \to 2m$}
        \State $\y_{m,j} = \y_{m,j-2} + \frac{\dt}{m}\f(\y_{m,j-1})$  \Comment{Midpoint steps}
    \EndFor
    \State $T_{m,1} = \y_{m,2m}$
\EndFor

\For{$k=2 \to r$}  \Comment{Extrapolate to get higher order}
    \For{$j=k \to r$}
        \State $T_{j,k} = T_{j,k-1} + \frac{T_{j,k-1}-T_{j-1,k-1}}{\frac{j^2}{(j-k+1)^2}-1}$
        \Comment{Aitken--Neville formula for extrapolation to order $2k$}
    \EndFor
\EndFor
\State $\unone = T_{r,r}$ \Comment{New solution value}
\end{algorithmic}
\end{algorithm}


\subsection{The internal stability polynomials}\label{section41InternalstabilityPolynomials}
By analyzing the perturbed scheme along the lines of 
\textbf{Algorithm \ref{alg:extrap}} and \textbf{Algorithm \ref{alg:extrapm}}, 
we will get the internal stability polynomials $Q_{p,m,\ell}$---this time it is natural to use two indices, 
say, $m$ and $\ell$, to label them, with the dependence on $p$ also indicated.
As in earlier sections, the analysis is carried out by choosing $\f(\uu):=\lambda \uu$ and $z:=\lambda\dt$.

Then, in the EE extrapolation case,  for $1\le m \le p$ and $1\le j\le m$ we have $\ty_{m,0}:=\un$, 
$\tr_{m,0}:=0$ and 
\[\ty_{m,j}:=\left(1+\frac{z}{m}\right)\ty_{m,j-1}+\tr_{m,j}.\]
This implies that
\begin{equation}\label{Tm1EEextrapolation}
T_{m,1}^{\mathrm{EE}}\equiv\ty_{m,m}=\left(1+\frac{z}{m}\right)^m \un+
\sum_{\ell=1}^m \left(1+\frac{z}{m}\right)^{m-\ell} \tr_{m,\ell}.
\end{equation}

As for the EM extrapolation, $p$ is assumed to be even, $r:=\frac{p}{2}$ and for $1\le m \le r$ and $2\le j\le 2m$ 
we have $\ty_{m,0}:=\un$, $\tr_{m,0}:=0$, $\ty_{m,1}:=\left(1+\frac{z}{2m}\right)\ty_{m,0}+\tr_{m,1}$ and 
\[\ty_{m,j}:=\ty_{m,j-2}+\frac{z}{m}\ty_{m,j-1}+\tr_{m,j}.\]
For any $1\le m \le r$, let us introduce an auxiliary sequence $q_{m,j}(z)$ by $q_{m,0}(z):=0$, $q_{m,1}(z):=1$, and 
\begin{equation}\label{littleqEM}
q_{m,j}(z):=\frac{z}{m}q_{m,j-1}(z)+q_{m,j-2}(z) \quad\quad (2\le j\le 2m).
\end{equation}  
Then, it can be proved that  
\begin{equation}\label{Tm1EMextrapolation}
T_{m,1}^{\mathrm{EM}}\equiv\ty_{m,2m}=
 (-1)^m \left({\Large{\texttt T}}_{2 m}\left(\frac{i z}{2 m}\right)+
i {\Large{\texttt U}}_{2 m-1}\left(\frac{i z}{2 m}\right)\right)\un+
\sum_{\ell=1}^{2m} q_{m,2m-\ell+1}(z) \, \tr_{m,\ell},
\end{equation} 
where ${\Large{\texttt T}}_{\ell}\left(\cdot\right)$ and ${\Large{\texttt U}}_{\ell}\left(\cdot \right)$
denote the $\ell^\mathrm{th}$ Chebyshev polynomials of the first kind and second kind, 
respectively, and $i$ is the imaginary unit. We are only interested in the coefficients of the 
$\tr_{m,\ell}$ quantities, hence no auxiliary computations to derive the first part of 
the expression, $(-1)^m(\ldots)\un$, are presented here. We remark that the $q_{m,j}$ polynomials can 
be expressed in a similar way, for example, as
$q_{m,j}(z)=(-i)^{j-1} {\Large{{\texttt  U}}}_{j-1}\left(\frac{i z}{2 m}\right)$,
 but we will not need this form later.\\

As a next step in both algorithms, the Aitken--Neville interpolation formula is
used. Therefore, an explicit 
form of $T_{j,k}$ together with the two special cases we 
are interested in are given. Suppose that a general step-number sequence $n_j$ ($j=1,2, \ldots$) and arbitrary starting values 
$T_{m,1}$ ($m\ge 1$) have been chosen. Then, the $T_{j,k}$ values are recursively defined by the 
Aitken--Neville formula as  
\begin{equation}\label{aitken-neville}
T_{j,k}:=T_{j,k-1}+\frac{T_{j,k-1}-T_{j-1,k-1}}{\frac{n_j}{n_{j-k+1}}-1},\quad \mathrm{for\ } j\ge 2\mathrm{\ and\ }  2\le k\le j.
\end{equation}
\begin{lem}
Suppose the sequence $T_{j,k}$ is defined by (\ref{aitken-neville}). Then for any $j\ge1$ and 
$1\le k \le j$ we have that
\[
T_{j,k}= \sum_{m=j-k+1}^{j}\left(\ \prod_{ {\ell=j-k+1,  \ell\ne m}}^{j} \frac{n_m}{n_m-n_\ell}\ \right) T_{m,1}.
\]
\end{lem}

\begin{cor}\label{Tm1EEcorollary}
In the EE extrapolation algorithm with $n_j=j$, $T_{p,p}$ ($p\ge 1$) defined by (\ref{aitken-neville}) 
can be written as
\[
T_{p,p}=\sum _{m=1}^p \frac{(-1)^{m+p} m^{p-1} }{(p-m)! (m-1)!}T_{m,1}.
\]
\end{cor}
\begin{cor}\label{Tm1EMcorollary}
In the EM extrapolation algorithm with $n_j=j^2$, $T_{r,r}$ ($r\ge 1$) defined by (\ref{aitken-neville}) 
takes the form
\[
T_{r,r}=\sum _{m=1}^r \frac{2(-1)^{m+r} m^{2r} }{(r-m)! (r+m)!}T_{m,1}.
\]
\end{cor}

Now we are ready to combine (\ref{Tm1EEextrapolation}) with 
Corollary \ref{Tm1EEcorollary}, and (\ref{Tm1EMextrapolation}) with 
Corollary \ref{Tm1EMcorollary} to get the coefficients of the $\tr_{m,\ell}$ quantities, 
these coefficients being
the internal stability polynomials $Q_{p,m,\ell}$. Thus the following lemmas are proved.

\begin{lem}\label{QEEpmlexplicitformlemma}  For any $p\ge 2$ and $z\in\mathbb{C}$, the internal stability polynomials of the
EE extrapolation method are given by
\[
Q_{p,m,\ell}^{\mathrm{EE}}(z)=\frac{(-1)^{m+p} m^{p-1} }{(p-m)! (m-1)!}
\left(1+\frac{z}{m}\right)^{\ell-1}\quad\quad (1\le m \le p, 1\le \ell \le m).
\] 
\end{lem}

\begin{lem}\label{QEMpmlexplicitform} For any even $p\ge 2$ and $z\in\mathbb{C}$, the internal stability polynomials of the
EM extrapolation method are given by
\[
Q_{p,m,\ell}^{\mathrm{EM}}(z)=\frac{2(-1)^{m+r} m^{2r} }{(r-m)! (r+m)!}
\,q_{m,2m-\ell+1}(z)\quad\quad (1\le m \le r, 1\le \ell \le 2m)
\] 
with $r=\frac{p}{2}$ and $q_{m,j}$ defined in (\ref{littleqEM}).
\end{lem}

Finally, we know that we obtain the stability polynomial $P(z)$ of the method, if---instead of the 
coefficients of the $\tr_{m,\ell}$ terms---we collect the 
coefficient of $\un$ in the relation $\unone=T_{p,p}$ (EE extrapolation) or
 $\unone=T_{r,r}$ (EM extrapolation). It is easily seen from the constructions in both
cases that the degree of $P$ is at most $p$, and $P$ approximates the exponential function to 
order $p$ near the origin. Therefore, the stability function of both methods is the degree $p$ 
Taylor polynomial of the exponential function centered at the origin.

As a by-product, (\ref{Tm1EEextrapolation}) with 
Corollary \ref{Tm1EEcorollary}, and (\ref{Tm1EMextrapolation}) with 
Corollary \ref{Tm1EMcorollary} also prove the following two identities:
\[
\sum_{m=1}^p \frac{(-1)^{m+p} m^{p-1}}{(p-m)!(m-1)!}\left(1+\frac{z}{m}\right)^m=
\sum_{m=0}^p \frac{z^m}{m!}\quad\quad (p\in\mathbb{N}^+),
\]
\[
\sum_{m=1}^{r} \frac{2 (-1)^{r} m^{2r}}{(r-m)!(r+m)!}\left({\Large{\texttt T}}_{2 m}\left(\frac{i z}{2 m}\right)+
i {\Large{\texttt U}}_{2 m-1}\left(\frac{i z}{2 m}\right)\right)=
\sum_{m=0}^{2r} \frac{z^m}{m!}\quad\quad (r\in\mathbb{N}^+).
\]

\subsection{Bounds on the absolute stability region $\displaystyle 
\Bigg|\sum_{m=0}^p \frac{z^m}{m!}\Bigg|\le 1$}\label{extrapolshapeofSp}

Let ${\mathcal{T}}_p(z):=\sum_{m=0}^p \frac{z^m}{m!}$ denote the 
degree $p$ Taylor polynomial of the exponential function around $0$. We have seen in 
the previous subsection that ${\mathcal{T}}_p$ is 
the stability polynomial of the $p$th order EE and EM extrapolation methods, hence the absolute 
stability region in both cases is the set
\[
\Ss_p=\{ z\in\mathbb{C} : |{\mathcal{T}}_p(z)|\le 1 \}.
\]
Figure \ref{firstfewTaylorStabRegions} shows the boundaries of the first few stability regions. 

\begin{figure}[h!]
  \centering
  \includegraphics[width=0.5\textwidth]{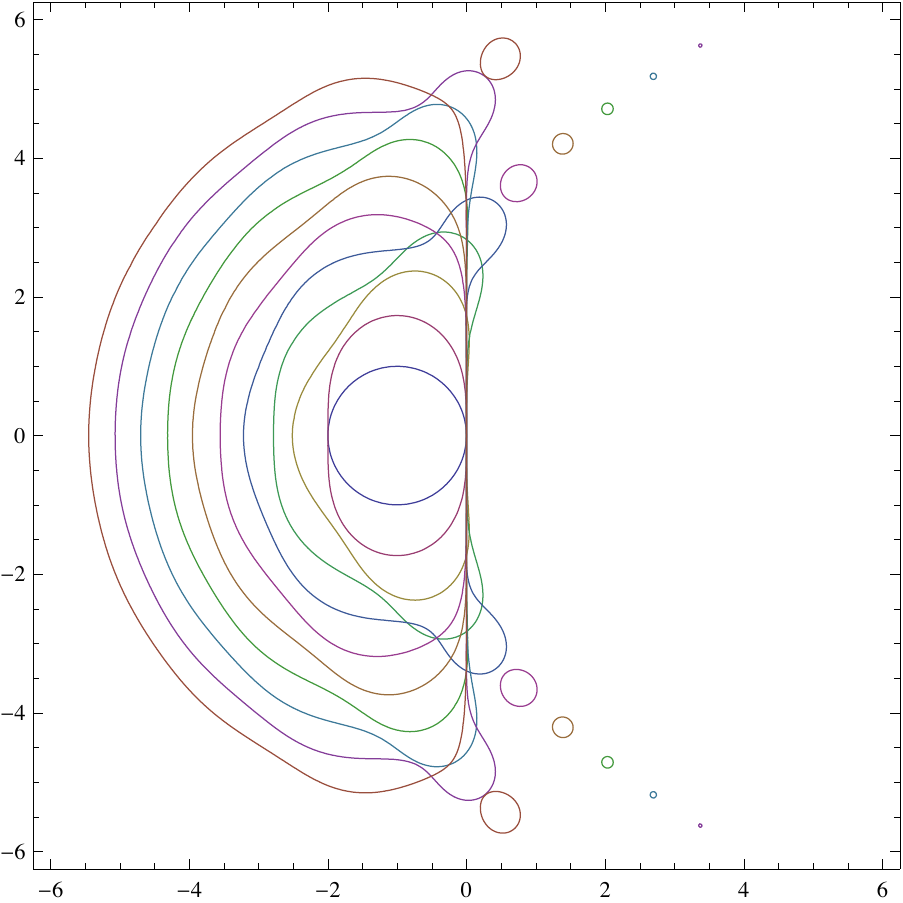}
  \caption{For fixed $1\le p\le 11$, the boundary curve(s) of the stability regions 
$\Ss_p\subset\Complex$ 
are depicted in the same color.\label{firstfewTaylorStabRegions}}
\end{figure}

Now we cite from \cite{exp_stability_regions} some bounds on the sets
$\Ss_p$ and  $\Ss_p\cap \Complex_{-,0}$ that will be used in the 
next subsections.

\begin{lem}\label{extrapollemma4.3shape}
For any $p\ge 2$ 
\[z\in \Ss_p \implies |z|\le 1.6\,p.\]
\end{lem}

\begin{lem}\label{exp_stab_region1+eps}
For any $\varepsilon\in (0,1)$ and $\mathbb{N}^+\ni p\ge \left(\frac{1.0085\, e}{\varepsilon}\right)^2$
\[
z\in \Ss_p \implies |z|\le (1+\varepsilon)p.
\]
\end{lem}

\begin{lem}\label{exp_stab_region95/100}
For any $p\ge 3$ 
\[z\in \Ss_p \cap \Complex_{-,0}\implies |z|\le 0.95 \, p.\]
\end{lem}

\begin{lem}\label{exp_stab_region1/e}
For any $\varepsilon>0$ there exists $p(\varepsilon)\in\mathbb{N}^+$ such that
$p\ge p(\varepsilon)$ and $z\in\Ss_p\cap \Complex_{-,0}$ imply 
$|z|\le \left(\frac{1}{e}+\varepsilon\right) p$.
\end{lem}

\begin{rem}\label{asymptoptimalityremark}
We remark that \cite[Theorem 5.4]{jeltschnevanlinna} shows that 
Lemmas \ref{exp_stab_region1+eps} and \ref{exp_stab_region1/e} above are asymptotically
optimal in the following sense. 
On one hand, for each $\varepsilon\in (0,1)$ and $p_0\in\mathbb{N}^+$ there is a $p\ge p_0$
and $z\in \Ss_p$  such that $\varepsilon p< |z|$; on the other hand,  
for each $\varepsilon\in (0,1)$ and $p_0\in\mathbb{N}^+$ there is a $p\ge p_0$
and $z\in \Ss_p\cap \Complex_{-,0}$ such that $\varepsilon \frac{p}{e}< |z|$.
\end{rem}

In addition, \cite{exp_stability_regions} also determines the quantities $\max_{z\in\Ss_p} |z|$  
and $\max_{z\in\Ss_p\cap \Complex_{-,0}} |z|$ as exact algebraic numbers for $1\le p\le 20$.
In Tables \ref{farthestpointTaylorpoly} and \ref{farthestpointTaylorpolylefthalfplane} 
we reproduce those figures. For the sake of brevity, instead of listing any exact (and complicated) 
algebraic numbers 
(with degrees up to 760), their values are rounded up, 
so Tables \ref{farthestpointTaylorpoly} and \ref{farthestpointTaylorpolylefthalfplane} 
provide strict upper bounds. Notice that the corresponding tables in \cite{exp_stability_regions}
refer to the scaled stability regions $\left\{z\in \Complex : 
\displaystyle \left|\sum_{m=0}^p \frac{(p z)^m}{m!}\right|\le 1\right\}$, 
but no scaling is applied in Tables \ref{farthestpointTaylorpoly} and \ref{farthestpointTaylorpolylefthalfplane}
in the present work. 

We add finally that some interesting patterns are reported in \cite{exp_stability_regions}
concerning the sets $\Ss_p\cap \Complex_0$, where $\Complex_0$ denotes
the set of complex numbers with zero real part. One result says, for example, 
that for $1\le p\le 100$, $\Ss_p\cap \Complex_0=\{0\}$ if and only if $p\in\{1,2,6\}$.

\begin{table}[h]
\begin{center}
    \begin{tabular}{c|c||c|c}
    $p$ & $\max_{z\in\Ss_p} |z|$ & $p$ & $\max_{z\in\Ss_p} |z|$ \\ \hline
    1 & 2 & 11 & 7.302 \\
    2 & $\sqrt{2 \left(1+\sqrt{2}\right)}\approx 2.198$ & 12 & 8.035 \\
    3 & 2.539 & 13 & 8.780 \\
    4 & 2.961 & 14 & 9.535 \\
    5 & 3.447 & 15 & 10.298 \\
    6 & 3.990 & 16 & 11.069 \\
    7 & 4.582 & 17 & 11.846 \\
    8 & 5.218 & 18 & 12.628 \\
    9 & 5.888 & 19 & 13.417 \\
    10 & 6.585 & 20 & 14.210 \\ \hline
    \end{tabular}
\caption{For $p\ge 3$, the exact maximum values are rounded up.\label{farthestpointTaylorpoly}}
\end{center}
\end{table}

\begin{table}[h]
\begin{center}
    \begin{tabular}{c|c||c|c}
    $p$ & $\max_{z\in\Ss_p\cap \Complex_{-,0}} |z|$ & $p$ & $\max_{z\in\Ss_p\cap \Complex_{-,0}} |z|$ \\ \hline
    1 & 2 & 11 & 5.451 \\
    2 & $\sqrt{2 \left(1+\sqrt{2}\right)}\approx 2.198$ & 12 & 5.825 \\
    3 & 2.539 & 13 & 6.231 \\
    4 & 2.961 & 14 & 6.657 \\
    5 & 3.396 & 15 & 7.108 \\
    6 & 3.581 & 16 & 7.325 \\
    7 & 3.961 & 17 & 7.700 \\
    8 & 4.367 & 18 &  8.092\\
    9 & 4.800 & 19 &  8.513\\
    10 & 5.262 & 20 &  8.955\\ \hline
    \end{tabular}
\caption{For $p\ge 3$, the exact maximum values are rounded up.\label{farthestpointTaylorpolylefthalfplane}}
\end{center}
\end{table}

\subsection{Bounds on $\M^{\mathrm{EE}}_p({\mathcal{S}})$}\label{subsection43boundsonMpEE}

\subsubsection{The case ${\mathcal{S}}=\Ss_p$}
\label{subsection421}

First we estimate
\[
\M^{\mathrm{EE}}_p(\Ss_p)\equiv \max_{1\le m \le p }\max_{1\le \ell \le m}\max_{z\in \Ss_p}\left\vert
Q_{p,m,\ell}^{\mathrm{EE}}(z)\right\vert
\]
from above for any $p\ge 2$. We are going to use the  
Lambert $W$-function (a.k.a. {\texttt{ProductLog}} 
in \textit{Mathematica}): for $x\ge-\frac{1}{e}$, there is a unique $W(x)\ge -1$ such that 
\begin{equation}\label{Wdefinition}
x=W(x) e^{W(x)}.
\end{equation}

By taking into account the fact that
for any fixed $c>0$ the function 
$[1,+\infty)\ni m\mapsto \left( 1+\frac{c}{m}\right)^{m-1}$ is monotone
increasing, Lemma \ref{extrapollemma4.3shape} yields that
\begin{equation}\label{upperboundsonMpEESpinitialestimate}
\left\vert 1+\frac{z}{m}\right\vert^{\ell-1}\le 
\left( 1+\frac{16p}{10m}\right)^{\ell-1}\le \left( 1+\frac{8p}{5m}\right)^{m-1}\le 
\left( 1+\frac{8p}{5p}\right)^{p-1}=2.6^{p-1}
\end{equation}
(notice that the above bound also holds in the
exceptional $z=-m$, $\ell=1$ case defined via a limit), so by Lemma \ref{QEEpmlexplicitformlemma} 
\begin{equation}\label{QEEpmlexplicitform}
\left\vert Q_{p,m,\ell}^{\mathrm{EE}}(z)\right\vert = \frac{m^{p-1} }{(p-m)! (m-1)!}
\left\vert 1+\frac{z}{m}\right\vert^{\ell-1}\le \frac{m^{p} }{(p-m)! m!}\cdot
\frac{2.6^{p}}{2.6}. 
\end{equation}
Now we eliminate both factorials by means of the following estimate (whose proof is again a 
standard monotonicity argument, hence omitted here).

\begin{lem}\label{factorialestimatelemma}
For any $n\in\mathbb{N}^+$ we have
$
\left(\frac{n}{e}\right)^n \sqrt{2\pi n} < n! \le e \left(\frac{n}{e}\right)^n \sqrt{n}.
$\\
\end{lem}

\noindent The case $m=p$ will be dealt with later. Otherwise, if $1\le m\le p-1$, the lemma says that 
\[
\frac{m^{p} }{(p-m)! m!}\cdot
\frac{2.6^{p}}{2.6} \le \frac{e^{p}}{2\pi }\cdot\frac{1}{\sqrt{m(p-m)}}
\left( \frac{p}{m}-1\right)^{m-p}\cdot \frac{2.6^{p}}{2.6}=\ldots
\]
Now by introducing a new variable $x:=\frac{m}{p}\in \left[ \frac{1}{p},1-\frac{1}{p}\right]$ we obtain 
\begin{equation}\label{beforeW}
\ldots=\frac{(2.6e)^{p}}{5.2\pi }\cdot\frac{1}{p \sqrt{(1-x) x}} \left(\left(\frac{1}{x}-1\right)^{x-1}\right)^{p}
\le \frac{(2.6e)^{p}}{5.2\pi }\cdot \frac{1}{\sqrt{p-1}} \left(\left(\frac{1}{x}-1\right)^{x-1}\right)^{p}.
\end{equation}
Elementary computation shows that the function $(0,1)\ni x \mapsto \left(\frac{1}{x}-1\right)^{x-1}$
is unimodal with a strict maximum at the root of the transcendental equation
\[
\frac{1}{x}+\ln \left(\frac{1}{x}-1\right)=0, 
\]
which is at $x=\frac{1}{1+W\left(\frac{1}{e}\right)}\approx 0.782188$. By using this fact and 
examining the boundary behavior 
as well, we get
that the range of 
$(0,1)\ni x \mapsto \left(\frac{1}{x}-1\right)^{x-1}$ is the interval 
$
\left(0,e^{W\left(\frac{1}{e}\right)}\right)$. This implies that the right-hand side of (\ref{beforeW})
is estimated from above further by
\[
\frac{(2.6e)^{p}}{5.2\pi }\cdot \frac{1}{\sqrt{p-1}} \left(e^{W\left(\frac{1}{e}\right)}\right)^{p}=
\frac{\left(\frac{2.6}{W\left(\frac{1}{e}\right)}\right)^p}{5.2 \pi  \sqrt{p-1}}.
\]

Finally,  again by Lemma 
\ref{factorialestimatelemma}, we give an upper estimate in the $m=p$ case, separated earlier, as
\[
\frac{m^{p} }{(p-m)! m!}\cdot
\frac{2.6^{p}}{2.6} \Bigg|_{m=p}\le \frac{2.6^{p-1} e^p}{\sqrt{2 \pi } \sqrt{p}}.
\]

By comparing the two upper estimates (for $1\le m\le p-1$ \textit{versus} the one for $m=p$) and
selecting the larger, we establish the following.

\begin{thm} \label{MEEupperestimate} We have
\[
\M^{\mathrm{EE}}_2(\Ss_2) \le \frac{13 e^2}{10 \sqrt{\pi }}<5.42,
\] 
while for any $p\ge 3$ 
\[
\M^{\mathrm{EE}}_p(\Ss_p)\le \frac{\left(\frac{2.6}{W\left(\frac{1}{e}\right)}\right)^p}{5.2 \pi  \sqrt{p-1}}<
\frac{9.34^p}{5.2 \pi  \sqrt{p-1}}.
\]
\end{thm}

In (\ref{upperboundsonMpEESpinitialestimate}), we have used Lemma \ref{extrapollemma4.3shape}
(valid for all $p\ge 2$) to estimate $|z|$ from above. By using (the asymptotically optimal) 
Lemma \ref{exp_stab_region1+eps} instead, but otherwise following the same steps as in 
the proof of Theorem \ref{MEEupperestimate}, 
one can reduce the constant $9.34$ to get an asymptotically better bound, described by the next theorem.

\begin{thm}\label{7.19theorem}
For any $\varepsilon>0$
there is a $p(\varepsilon)\in\mathbb{N}^+$ such that for all $p\ge p(\varepsilon)$
\[
\M^{\mathrm{EE}}_p(\Ss_p)\le 
\frac{\left(\frac{2+\varepsilon}{W\left(\frac{1}{e}\right)}\right)^p}{(4+2\varepsilon)\pi  \sqrt{p-1}}
< \frac{\left(7.19+3.6\varepsilon\right)^p}{(4+2\varepsilon)\pi \sqrt{p-1}}.
\]
\end{thm}

It is also possible to determine the exact values of $\M^{\mathrm{EE}}_p(\Ss_p)$ for small $p$
values by using the same direct approach as in \cite{exp_stability_regions} (when deriving 
the values in
Tables \ref{farthestpointTaylorpoly} and \ref{farthestpointTaylorpolylefthalfplane} of
the present work). Namely,  the left-hand side of 
 (\ref{QEEpmlexplicitform}) is rewritten by using $z=x+i y$ as
\[\frac{m^p} {(p-m)!m!}\left(\left(1+\frac{x}{m}\right)^2+\frac{y^2}{m^2}\right)^{\frac{\ell-1}{2}},
\]
then, for any fixed $p\ge 2$, $1\le m\le p$ and $1\le \ell \le m$, 
\textit{Mathematica}'s \texttt{Maximize} command is applied with the above objective function 
and $x+i y\in\Ss_p$.
The obtained exact algebraic numbers have been rounded up and presented in 
Table \ref{maxQEEpmlexactvalues}.
The computing time is roughly doubled in each step (from $p$ to $p+1$):
determination of  $\M^{\mathrm{EE}}_{14}(\Ss_{14})$
took approximately 21 minutes (on a commercial computer). 
Based on the data in Table \ref{maxQEEpmlexactvalues} and assuming
that $\M^{\mathrm{EE}}_p(\Ss_p)$ grows like $c_1 \cdot\frac{c_2^p}{p}$ with some $c_1>0$ and $c_2>0$
 (as suggested by the results in the present and the following subsections), 
the best fit returned by \textit{Mathematica}'s \texttt{FindFit} is
\begin{equation}\label{6.16baseapprox}
\M^{\mathrm{EE}}_p(\Ss_p)\approx 0.0391 \cdot\frac{6.16^p}{p}.
\end{equation}

\begin{table}[h!]
\begin{center}
    \begin{tabular}{c|c||c|c}
    $p$ & $\M^{\mathrm{EE}}_p(\Ss_p)$ & $p$ & $\M^{\mathrm{EE}}_p(\Ss_p)$ \\ \hline
    2 & $\sqrt{2 \left(1+\sqrt{2}\right)}\approx 2.198$  & 9 &  61597.788\\
    3 & $6.192$ & 10 &  $340968.029$\\
    4 & $25.614$ & 11 &  $1.871\cdot 10^6$\\
    5 & $115.313$ & 12 &  $1.020\cdot 10^7$\\
    6 & $524.610$  & 13 &  $5.520\cdot 10^7$\\
    7 & $2427.838$ & 14 &  $3.168\cdot 10^8$ \\
    8 & $11431.562$ &  &  \\ \hline
       \end{tabular}
\caption{For $p\ge 3$, the exact maximum values are rounded up.\label{maxQEEpmlexactvalues}}
\end{center}
\end{table}

Table \ref{maxQEEpmlgoodupperbounds} extends the upper bounds on 
$\M^{\mathrm{EE}}_p(\Ss_p)$ to the range $15\le p \le 20$ in the following sense. 
By using the triangle inequality (cf. (\ref{QEEpmlexplicitform}))
\[
\left\vert Q_{p,m,\ell}^{\mathrm{EE}}(z)\right\vert = \frac{m^{p} }{(p-m)! m!}
\left\vert 1+\frac{z}{m}\right\vert^{\ell-1}\le 
\frac{m^{p} }{(p-m)! m!}\left( 1+\frac{|z|}{m}\right)^{\ell-1},
\]
the value of $|z|$ on the right has been replaced by its corresponding maximal value given in 
Table \ref{farthestpointTaylorpoly}, then the maximal value of the new right-hand side
has been (effortlessly) determined for all $(m,\ell)$ pairs with $1\le m \le p$ and $1\le \ell \le m$.
Table \ref{maxQEEpmlgoodupperbounds}  intentionally overlaps with 
Table \ref{maxQEEpmlexactvalues} (for $11\le p\le 14$) so that the effect of the triangle inequality on the estimates
 can be studied. Tables \ref{maxQEEpmlexactvalues} and \ref{maxQEEpmlgoodupperbounds} of course
contain much better upper estimates than Theorem \ref{MEEupperestimate}.\\

\begin{table}[h!]
\begin{center}
    \begin{tabular}{c|c||c|c}
    $p$ & $\M^{\mathrm{EE}}_p(\Ss_p)\le\ldots $ & $p$ & $\M^{\mathrm{EE}}_p(\Ss_p)\le\ldots$ \\ \hline
    11 &  $5.011\cdot 10^6$ & 16 &  $2.890\cdot 10^{10}$ \\
    12 & $2.782\cdot 10^7$ & 17 &  $1.687\cdot 10^{11}$ \\
    13 & $1.529\cdot 10^8$ & 18 &  $9.742\cdot 10^{11}$ \\
    14 & $8.333\cdot 10^8$ & 19 &  $5.581\cdot 10^{12}$ \\
    15 & $4.889\cdot 10^9$  & 20 &  $3.327\cdot 10^{13}$ \\ \hline
       \end{tabular}
\caption{Upper bounds on $\M^{\mathrm{EE}}_p(\Ss_p)$\label{maxQEEpmlgoodupperbounds}}
\end{center}
\end{table}

Finally, one can ask what the smallest constant $c_3>0$ is such that
$\M^{\mathrm{EE}}_p(\Ss_p)\le c_3^p$ holds for all $p$ large enough.
On one hand, we see that the ratios $\frac{\max_{z\in\Ss_p} |z|}{p}$ in Table \ref{farthestpointTaylorpoly}
are increasing for $7\le p\le 20$, already suggesting that $6.16<c_3$ (see (\ref{6.16baseapprox})).
However, Theorem \ref{7.19theorem} says that $c_3$ is not larger than $\approx 7.19$.
The following heuristic argument tries to find this optimal $c_3$. 

By Lemma \ref{exp_stab_region1+eps} and 
Remark \ref{asymptoptimalityremark}, $\max_{z\in\Ss_p} |z|\approx p$. 
But \cite[Theorem 5.4]{jeltschnevanlinna} also says that $1\in \Ss_\infty$
(recall that $\Ss_\infty$ in \cite{jeltschnevanlinna} denotes the ``limit'' of the
\textit{scaled} stability regions), 
hence there is a subsequence $p_k\to +\infty$
and $z_k\in\Ss_{p_k}$ such that $z_k/p_k\to 1$ ($k\to +\infty$).
So for these $z_k$ and $p_k$ values
\[
\M^{\mathrm{EE}}_{p_k}(\Ss_{p_k})\ge 
\M^{\mathrm{EE}}_{p_k}(\{z_k\})\approx \max_{1\le m \le p_k }\max_{1\le \ell \le m}
\frac{m^{p_k-1} }{(p_k-m)! (m-1)!}
\left\vert 1+\frac{p_k}{m}\right\vert^{\ell-1}=
\]
\[
\max_{1\le m \le p_k } \frac{m^{p_k-1} }{(p_k-m)! (m-1)!}
\left( 1+\frac{p_k}{m}\right)^{m-1}.
\]
On the other hand, by using the triangle inequality with Lemma \ref{exp_stab_region1+eps},  
we conclude (for general $p$ large enough) that $\M^{\mathrm{EE}}_{p}(\Ss_p)$
can be estimated from above by a quantity which is approximately
\begin{equation}\label{asymptoptupperestimateforMpEESp}
\max_{1\le m \le p } \frac{m^{p-1} }{(p-m)! (m-1)!}
\left( 1+\frac{p}{m}\right)^{m-1}.
\end{equation}
Therefore, (\ref{asymptoptupperestimateforMpEESp}) yields an 
asymptotically optimal upper estimate for $\M^{\mathrm{EE}}_p(\Ss_p)$.
Notice that in the proof of Theorem \ref{MEEupperestimate}, the two terms of the
product, $\frac{m^{p-1} }{(p-m)! (m-1)!}$ and
$\left\vert 1+\frac{z}{m}\right\vert^{\ell-1}$, are estimated from above separately; here 
we treat them together.
Now we again introduce $x:=\frac{m}{p}$ to rewrite the above maximum as
$\max_{0\le x \le 1 }  f_{1,p}(x)$, with
\[
f_{1,p}(x):=\frac{(p x)^p}{(p-p x)! (p x)!}\left(1+\frac{1}{x}\right)^{p x-1}.
\]
We see that each of these $f_{1,p}$ functions is unimodal, 
having a unique maximum.
It is reasonable to assume that, for large $p$, the abscissa of the maximum of $f_{1,p}$ 
remains approximately the same if we eliminate the $\Gamma$ functions by means of
Lemma \ref{factorialestimatelemma}; we also know that the first inequality in that lemma
is asymptotically optimal. Hence (\ref{asymptoptupperestimateforMpEESp}) is approximately
equal to $\max_{0\le x \le 1-\frac{1}{p}} \left( \frac{1}{2\pi p} f_{2,p}(x)\right)$, with
\[
f_{2,p}(x):=\frac{\sqrt{x} \left(e\, x^{1-2 x} (1-x)^{x-1} (1+x)^x\right)^p}{(1+x) \sqrt{1-x}}.
\]
Just as in the proof of Theorem \ref{MEEupperestimate}, here we also restrict the domain of $f_{2,p}$ to
$\left[0,1-\frac{1}{p}\right]$, because the application of Lemma \ref{factorialestimatelemma} introduces
a singularity at $x=1$ in $f_{2,p}$ not present in $f_{1,p}$; this restriction is justified by checking that
the value of (\ref{asymptoptupperestimateforMpEESp}) is unchanged if $m$ is restricted to 
$1\le m\le p-1$ provided that $p$ is large enough.

Now we differentiate $f_{2,p}$ to get
\[
\frac{1}{2} e^p\cdot x^{p(1-2x)-\frac{1}{2}}\cdot (1-x)^{p (x-1)-\frac{3}{2}}\cdot  (1+x)^{p x-2} f_{3,p}(x),
\]
where 
\[
f_{3,p}(x):=2 p (1-x) \left(1+x (x+1) \ln\left(\frac{1}{x^2}-1\right)\right)+1+x(2 x-1).
\]
We would like to estimate the unique zero of $f_{3,p}$ in $\left[0,1-\frac{1}{p}\right]$ for large $p$
(see Figure \ref{fig:f3pgraph}). But since $1+x(2 x-1)\ge \frac{7}{8}$, we see that if $p\to +\infty$, then
the unique zero of $f_{3,p}$ in $\left[0,1-\frac{1}{p}\right]$ should converge to the unique zero of
$1+x (x+1) \ln\left(\frac{1}{x^2}-1\right)$, being equal to $x^*\approx 
0.8143$ (this constant cannot
be simply expressed in terms of the usual functions, such as the $W$ function). Therefore, the
asymptotically optimal upper estimate of  $\M^{\mathrm{EE}}_p(\Ss_p)$ is 
\[\frac{1}{2\pi p} f_{2,p}(x^*)\approx \frac{6.868^p}{2\pi p}.\]

\begin{figure}
\begin{center}
\includegraphics[width=0.5\textwidth]{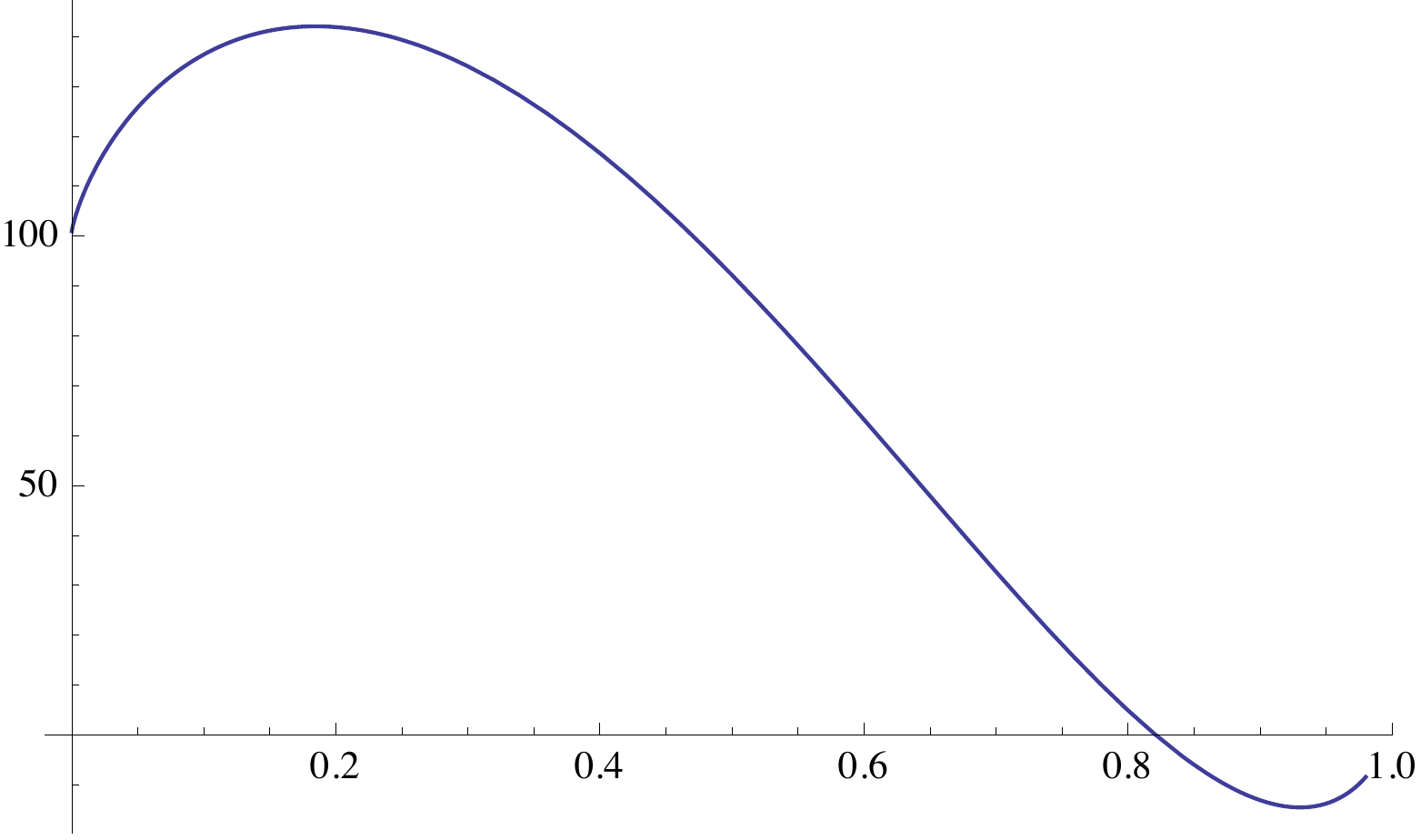}
\caption{The graph of $f_{3,50}$. \label{fig:f3pgraph}}
\end{center}
\end{figure}

We believe that the above heuristic explanation can be made more rigorous by using
techniques described in Remark \ref{remark41}, for example.

\subsubsection{The case ${\mathcal{S}}=\Ss_p\cap \Cleft$}
\label{subsection422}

The maximum internal amplification
factors with $z$ restricted to $\Ss_p\cap \Complex_{-,0}$ play an important role in
practical computations.

Repeating the proof of Theorem \ref{MEEupperestimate} by using Lemma
\ref{exp_stab_region95/100} instead of Lemma \ref{extrapollemma4.3shape}, 
we obtain the following theorem (the $p=2$ case is exactly covered by Table \ref{maxQEEpmllefthalfplaneexactvalues}).

\begin{thm}\label{MEElefthalfplaneupperestimate} For any $p\ge 3$ 
\[
\M^{\mathrm{EE}}_p(\Ss_p\cap \Complex_{-,0})\le 
\frac{\left(\frac{1.95}{W\left(\frac{1}{e}\right)}\right)^p}{3.9 \pi  \sqrt{p-1}}<
\frac{7.01^p}{3.9 \pi  \sqrt{p-1}}.
\]
\end{thm}

The counterpart of Theorem \ref{7.19theorem} (relying on Lemma \ref{exp_stab_region1/e}) 
is the following.

\begin{thm}\label{4.92theorem}
For any $\varepsilon>0$
there is a $p(\varepsilon)\in\mathbb{N}^+$ such that for all $p\ge p(\varepsilon)$
\[
\M^{\mathrm{EE}}_p(\Ss_p\cap \Complex_{-,0})\le 
\frac{\left(\frac{1+\frac{1}{e}+\varepsilon}{W\left(\frac{1}{e}\right)}\right)^p}{(2+\frac{2}{e}+
2\varepsilon)\pi  \sqrt{p-1}}
< \frac{\left(4.92+3.6\varepsilon\right)^p}{(2+\frac{2}{e}+
2\varepsilon)\pi \sqrt{p-1}}.
\]
\end{thm}

The first few  exact
values of 
$\M^{\mathrm{EE}}_p(\Ss_p\cap \Complex_{-,0})$ are displayed in 
Table \ref{maxQEEpmllefthalfplaneexactvalues} for $2\le p\le 14$. 

\begin{table}[h!]
\begin{center}
    \begin{tabular}{c|c||c|c}
    $p$ & $\M^{\mathrm{EE}}_p(\Ss_p\cap \Complex_{-,0})$ & $p$ & $\M^{\mathrm{EE}}_p(\Ss_p\cap \Complex_{-,0})$ \\ \hline
    2 & $\sqrt{2 \left(1+\sqrt{2}\right)}\approx 2.198$  & 9 & $11631.367$  \\
    3 & $6.192$ & 10 &  $46860.486$ \\
    4 & $51/2=25.5$ & 11 &  $98425.587$ \\
    5 & $\left(47+\sqrt{65}\right)^{3/2}/\sqrt{18}\approx 96.305$ & 12 &   $336910.368$ \\
    6 & $190.163$  & 13 &  $1.444\cdot 10^6$ \\
    7 & $631.328$ & 14 &   $6.561\cdot 10^6$ \\
    8 & $2549.961$ &  &  \\ \hline
       \end{tabular}
\caption{Maximum internal amplification factors for the first few explicit Euler extrapolation methods 
of order $p$ with respect to the absolute stability region of the method in the left half of the
complex plane. For $p=3$ or $p\ge 6$, the exact maximum values are rounded up 
(the exact values corresponding to $p\in\{4, 5\}$ are simple algebraic numbers).
\label{maxQEEpmllefthalfplaneexactvalues}}
\end{center}
\end{table}

The technique to determine these constants 
is completely analogous to the one
we used for creating Table \ref{maxQEEpmlexactvalues}. Computing the $p=14$ case now took 
$22.7$ minutes. We remark that the growth rate of the numbers in Table 
\ref{maxQEEpmllefthalfplaneexactvalues} is a little less regular
(due to the ``wavy'' behavior of $\max_{z\in\Ss_p\cap \Complex_{0}} |z|$ shown in
\cite{exp_stability_regions}), nevertheless, they grow like
\begin{equation}\label{4.86inequality}
0.0226\cdot \frac{4.86^p}{p}.
\end{equation}

By using the triangle inequality (cf. the corresponding 
description of Table \ref{maxQEEpmlgoodupperbounds}) 
and values from Table \ref{farthestpointTaylorpolylefthalfplane}, 
Table \ref{maxQEEpmllefthalfplanegoodupperbounds} extends the exact values for 
$2\le p\le 14$ to upper bounds on
$\M^{\mathrm{EE}}_p(\Ss_p\cap \Complex_{-,0})$ for $15\le p\le 20$. Thanks
to Table \ref{farthestpointTaylorpolylefthalfplane}, creating each entry of
Table \ref{maxQEEpmllefthalfplanegoodupperbounds} took only a fraction of a second.\\

\begin{table}[h!]
\begin{center}
    \begin{tabular}{c|c||c|c}
    $p$ & $\M^{\mathrm{EE}}_p(\Ss_p\cap \Complex_{-,0})\le\ldots $ & $p$ & $\M^{\mathrm{EE}}_p(\Ss_p\cap \Complex_{-,0})\le\ldots$ \\ \hline
    11 &  $ 1.911\cdot 10^6$ & 16 &  $ 3.800\cdot 10^{9}$ \\
    12 & $ 8.577\cdot 10^6$ & 17 &  $ 1.738\cdot 10^{10}$ \\
    13 & $ 3.847\cdot 10^7$ & 18 &  $ 7.910\cdot 10^{10}$ \\
    14 & $ 1.801\cdot 10^8$ & 19 &  $3.820\cdot 10^{11}$ \\
    15 & $ 8.946\cdot 10^8$  & 20 &  $1.893\cdot 10^{12}$ \\ \hline
       \end{tabular}
\caption{Upper bounds on $\M^{\mathrm{EE}}_p(\Ss_p\cap \Complex_{-,0})$
\label{maxQEEpmllefthalfplanegoodupperbounds}}
\end{center}
\end{table}

Finally, we can repeat the \textit{heuristic} argument at the end of Section \ref{subsection421}
to determine the smallest $c_4>0$ such that
$\M^{\mathrm{EE}}_p(\Ss_p\cap \Complex_{-,0})\le c_4^p$ for large $p$.
Since $\frac{\max_{z\in\Ss_p\cap \Complex_{-,0}} |z|}{p}$ is decreasing for $2\le p\le 20$ 
(see Table \ref{farthestpointTaylorpolylefthalfplane}), it is not surprising that this time $c_4<4.86<4.92$ 
(cf. (\ref{4.86inequality}) and Theorem \ref{4.92theorem}). Now \cite[Theorem 5.4]{jeltschnevanlinna}
says that $\Ss_{\infty}\cap \Complex_{-,0}=\{z\in\Complex : |z|\le \frac{1}{e}\}\cap \Complex_{-,0}$,
so---also taking into account the scaling present in the definition of $\Ss_\infty$---the 
largest value of $\frac{m^{p-1} }{(p-m)! (m-1)!}
\left\vert 1+\frac{z}{m}\right\vert^{\ell-1}$
(cf. the left-hand side of (\ref{QEEpmlexplicitform})) occurs when $\ell=m$ and 
$z\approx \pm\frac{i p}{e}$ (with $p$ and $m$ fixed). Hence we are interested in the quantity
\begin{equation}\label{maxwithsqareroot}
\max_{1\le m \le p } \frac{m^{p} }{(p-m)! m!}
\left( \sqrt{1+\left(\frac{p}{e m}\right)^2}\right)^{m-1}
\end{equation}
for large $p$. After analogous computations as before, we arrive at the following. Define
\[
f_{4,p}(x):=\frac{e \left(e\,x^{1-x} \, (1-x)^{x-1} \,\left(1+\frac{1}{e^2 x^2}\right)^{x/2}
   \right)^p}{2 \pi  p\sqrt{x}\,\sqrt{1-x}\,\sqrt{\frac{1}{x^2}+e^2}  },
\]
\[
f_{5,p}(x):=1-x+e^2 x^2+x\left(1+e^2 x^2\right) \left(\ln \sqrt{1+\frac{1}{e^2 x^2}}+
\ln \left(\frac{1}{x}-1\right)\right),
\]
and let $x^{**}\approx 0.7711$ denote the unique zero of $f_{5,p}$ in $[0,1]$. 
Then (\ref{maxwithsqareroot}) is approximately 
\[
f_{4,p}(x^{**})\approx 0.342\cdot \frac{3.885^p}{p}
\]
for large $p$.

\subsubsection{The case ${\mathcal{S}}=\{0\}$}

In this subsection 
we investigate a third quantity,  
\[\M^{\mathrm{EE}}_p(\{0\})=\max_{1\le m\le p}\frac{m^{p} }{(p-m)! m!}\] 
relevant when using methods with very small step sizes. The corresponding calculations in the proof 
of Theorem \ref{MEEupperestimate} already yield the following upper bound.

\begin{thm}  For any $p\ge 3$, we have
\[
\M^{\mathrm{EE}}_p(\{0\})\le 
 \frac{\left(\frac{1}{W\left(\frac{1}{e}\right)}\right)^p}{2 \pi  \sqrt{p-1}}
 < \frac{3.592^p}{2 \pi  \sqrt{p-1}}.
\]
\end{thm}

In addition, we can easily formulate explicit lower bounds on $\M^{\mathrm{EE}}_p(\{0\})$ also,
yielding lower estimates 
of $\M^{\mathrm{EE}}_p(\Ss_p)$ or $\M^{\mathrm{EE}}_p(\Ss_p\cap \Complex_{-,0})$
due to
\[
\M^{\mathrm{EE}}_p(\Ss_p)\ge \M^{\mathrm{EE}}_p(\Ss_p\cap \Complex_{-,0})\ge 
\M^{\mathrm{EE}}_p(\{0\}).
\]
First by assuming that $p\ge 4$ is divisible by 4, we choose 
$m=3p/4$ and apply Lemma \ref{factorialestimatelemma} to get
\[
\max_{1\le m\le p}\frac{m^{p} }{(p-m)! m!}\ge 
4\cdot 3^{p-1} e^{p-3} p^{\frac{3 (p-2)}{4}} (3 p-4)^{\frac{1}{2}-\frac{3 p}{4}}.
\]
Then we can simplify this lower bound and prove that for any $p\ge 4$ 
the above right-hand side is estimated from below by
$\frac{\sqrt{3}}{2 e^2}\cdot \frac{\left( \sqrt[4]{3} e \right)^p}{p}$.
On the other hand, if $p\ge 4$ is of the form $p=4k_1-1$ (with a suitable integer $k_1$), then we choose 
$m=3(p+1)/4$; if $p=4k_2-2$, then $m=3(p+2)/4$; finally, if $p=4k_3+1$, then $m=3(p-1)/4$
 is chosen: in all these cases, we verify in a similar manner that the given lower bound holds.
The choice of $m\approx \frac{3p}{4}$ is interpreted later in this 
subsection. Therefore the following theorem is established.

\begin{thm}\label{MEElowerestimate}  For any $p\ge 4$,
\[
\M^{\mathrm{EE}}_p(\{0\})\ge 
\frac{\sqrt{3}}{2 e^2}\cdot \frac{\left( \sqrt[4]{3}\, e \right)^p}{p}> 0.117\cdot \frac{ 3.577 ^p}{p}.
\]
\end{thm}

As for the first few exact values of $\M^{\mathrm{EE}}_p(\{0\})$, see 
Table \ref{maxQEEpmlattheorigin}.

\begin{table}[h!]
\begin{center}
    \begin{tabular}{c|c||c|c}
    $p$ & $\M^{\mathrm{EE}}_p(\{0\})$ & $p$ & $\M^{\mathrm{EE}}_p(\{0\})$ \\ \hline
    2 & $2$  & 12 & $137787$  \\
    3 & $9/2=4.5$ & 13 &  $459289$ \\
    4 & $27/2=13.5$ & 14 &  $1.586\cdot 10^6$ \\
    5 & $128/3\approx 42.7$ & 15 &   $5.361\cdot 10^6$ \\
    6 & $3125/24\approx 130.3$  & 16 &  $1.781\cdot 10^7$ \\
    7 & $1944/5=388.8 $ & 17 &   $5.830\cdot 10^7$ \\
    8 & $5832/5=1166.4$ & 18 &   $2.041\cdot 10^8$ \\
    9 & $4003.4$ & 19 &   $7.064\cdot 10^8$ \\
    10 & $13315.3$ & 20 &   $2.408\cdot 10^9$ \\
    11 & $43238.9$ &  &  \\ \hline
       \end{tabular}
\caption{For $p\ge 9$, the exact (rational) values are rounded up. 
\label{maxQEEpmlattheorigin}}
\end{center}
\end{table}

Finally, we give some information on 
the $p\to +\infty$ asymptotic growth rate of $\M^{\mathrm{EE}}_p(\{0\})$,
illustrating a ``discrete'' and a ``continuous'' approach
(actually, the following results were obtained earlier than the ones in
 previous subsections). 
To save space, we give only the basic ideas of the proofs.

For any $p\ge 2$ and $m\in [2,p]\cap\mathbb{N}$, let us define 
\begin{equation}\label{fpmdefinition}
f_p(m):=\frac{m^{p} }{(p-m)! m!}=\frac{m^p}{p!}\binom{p}{m}.
\end{equation}
Then by  analyzing the function 
$
m\mapsto g_p(m):=m^{-p} (m+1)^{p-1} (p-m)=\frac{f_p(m+1)}{f_p(m)}
$
defined for $p\ge 3$ and letting $m$ to be a ``continuous variable'', 
we can prove that there is a unique $\mathbb{N}\not \ni m^*(p)\in (2,p)$ such that
$g_p(m^*(p))=1$. This means that the function $[1,p]\cap\mathbb{N} \ni m\mapsto f_p(m)$ 
has a unique maximum at $\lceil m^*(p)\rceil$. The following lemma 
describes the asymptotic location of this unique maximum
 within the interval $[1,p]$.

\begin{lem}
\[
\lim_{p\to +\infty} \frac{m^*(p)}{p}=\frac{1}{1+W\left(\frac{1}{e}\right)}\approx 0.782188.
\]
\end{lem}

\noindent The interesting part in the proof has been to establish the existence of the above limit: we could not use
a monotonicity argument; instead, we first prove that $\frac{m^*(p)}{p}\in 
\left(\frac{1}{2},1\right)$ for all $p\ge 3$,
then by suitably rearranging $g_p(m^*(p))=1$, that is, the defining equation of $m^*(p)$, 
a bootstrap argument yields the existence of the limit together with its value.

Now, by using the $\Gamma$ function instead of the factorials as usual, the domain of definition of
$f_p$ is extended to the real interval $[1,p]$. Then by evaluating 
$f_p\left(\frac{1}{1+W\left(\frac{1}{e}\right)}\cdot p\right)$ and replacing the $\Gamma$ functions
with the first terms of their asymptotic expansion, we get that
$
\max_{1\le m\le p} f_p(m)
$
is approximately 
\begin{equation}\label{3.59112estimate}
\frac{1+W\left(\frac{1}{e}\right)}{2 \pi  \sqrt{W\left(\frac{1}{e}\right)}}
\frac{\left(\frac{1}{W\left(\frac{1}{e}\right)}\right)^p}{p} 
\approx 0.385588\cdot \frac{3.59112^p}{p} \quad\quad  (p\to +\infty), 
\end{equation}
see Figure \ref{fig:Mee0exactvsasymptotics}.

\begin{figure}
\begin{center}
\includegraphics[width=0.5\textwidth]{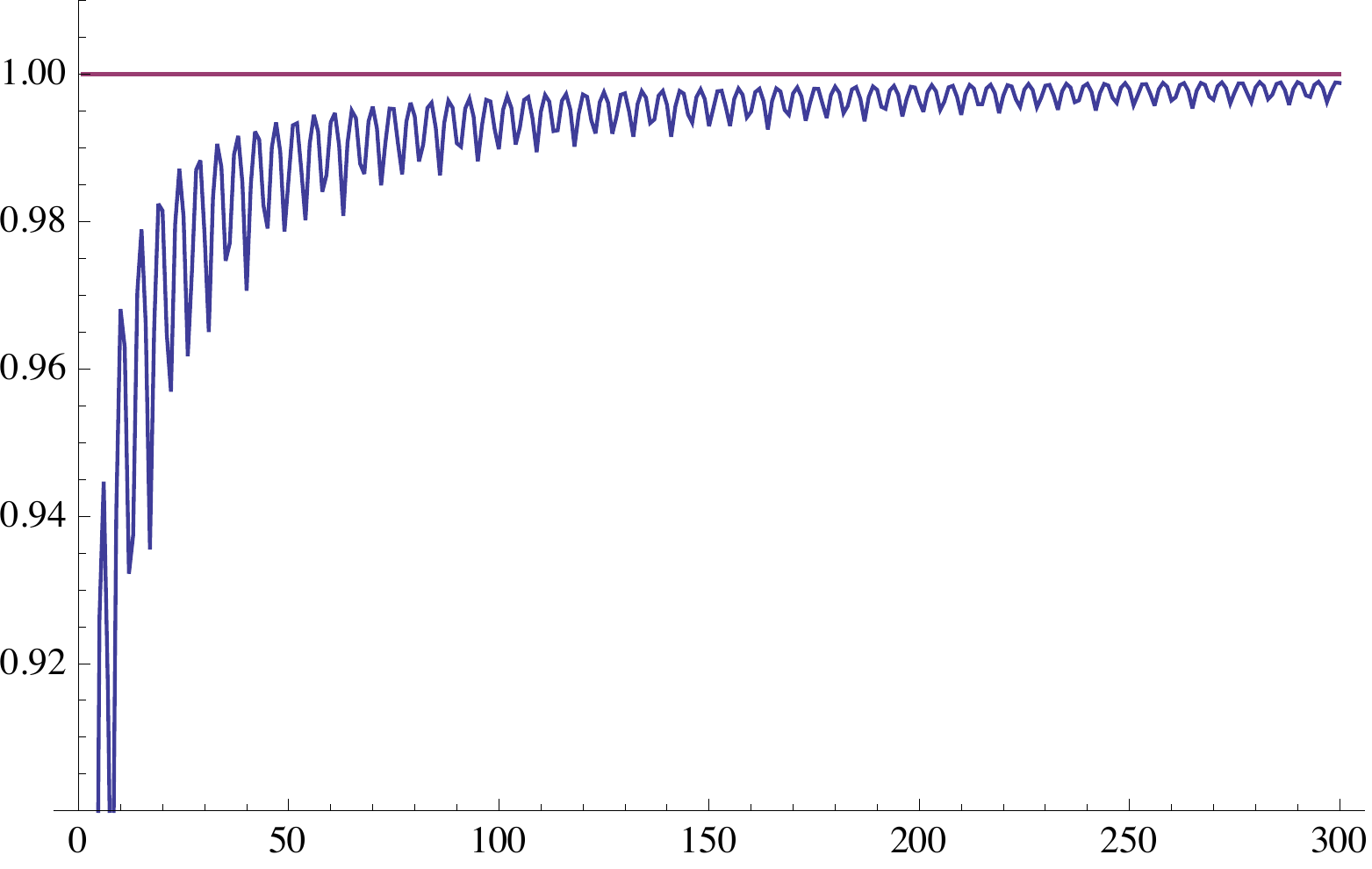}
\caption{The ratio between the quantities $\M^{\mathrm{EE}}_p(\{0\})$ and (\ref{3.59112estimate}) 
for $2\le p\le 300$ (with linear interpolation). \label{fig:Mee0exactvsasymptotics}}
\end{center}
\end{figure}

\begin{rem}\label{remark41} 
We were also able to formulate and prove the ``truly continuous'' (and slightly general,
but technically more demanding) 
counterpart of the above lemma, which we briefly describe now. Instead of $f_p$
in (\ref{fpmdefinition}), let us consider its extension,
$\frac{m^{p} }{\Gamma(1+p-m)\Gamma(1+m)}$, where $m\in [1,p]$. After  
a scaling $x:=\frac{m}{p}$, the domain of definition of the function 
\[
A_p(x):=\frac{(p x)^p}{\Gamma (1+p-p x)\Gamma (1+p x)}
\]
can be taken to be the whole (and now fixed) interval $x\in [0,1]$. Within this remark we will say 
that a function $f$ is {\em{strictly unimodal}} in $[0,1]$, if there is a unique $x^*\in[0,1]$ 
such that $f$ is strictly increasing on $[0,x^*]$, has a strict global maximum at $x^*$, 
then strictly decreasing on $[x^*,1]$. First we claim that for any fixed $p\ge 2$, 
$A_p$ is strictly unimodal in $[0,1]$.

Indeed, by using the 
following integral representation of the {\em{digamma function}}
(traditionally denoted by $\psi$, and by {\texttt{\em PolyGamma}} in Mathematica)
\[
\frac{\Gamma^\prime(z)}{\Gamma(z)}=\int_0^1 \frac{1-t^{z-1}}{1-t}dt -\gamma
\]
valid in the open right half-plane ($\gamma\approx 0.577216$ is the Euler-Mascheroni constant),
we have that
\[
\partial_x A_p(x)=\frac{p^{p+1} x^{p-1}}{\Gamma(1+p x)\ \Gamma(1+p (1-x)) }
\left[1+x \left(\frac{\Gamma^\prime(1+p(1- x))}{\Gamma(1+p (1-x))}-
\frac{\Gamma^\prime(1+p x)}{\Gamma(1+p x)}\right)\right]=
\]
\[
\frac{p^{p+1} x^{p-1}}{\Gamma(1+p x)\ \Gamma(1+p (1-x)) }
\, {\mathcal{I}}_p(x)
\]
with ${\mathcal{I}}_p(x):=\int_0^1 1+x\frac{t^{px}-t^{p(1-x)}}{1-t} \,dt$.  
Since the factor in front of ${\mathcal{I}}_p(x)$ is positive for $0<x<1$, it is enough to 
determine the sign of ${\mathcal{I}}_p(x)$. Examining the boundedness of its integrand
and appealing to the Dominated Convergence Theorem, we can prove that the function
\[
[0,1]\ni x\mapsto {\mathcal{I}}_p(x)
\]
is continuous. Then we show that ${\mathcal{I}}_p(x)\ge 1$ for $x\in\left[0,\frac{1}{2}\right]$,
${\mathcal{I}}_p(\cdot)$ is strictly decreasing on $\left[\frac{1}{2},1\right]$, and
${\mathcal{I}}_p(1)=-\sum_{k=2}^p\frac{1}{k}<0$. These imply that 
${\mathcal{I}}_p(\cdot)$ has a unique zero in $[0,1]$, which is denoted by $m_A(p)$. The
strict unimodality of $A_p$ is proved.

Our second claim (whose proof is just an application of Stirling's formula) 
is that for any fixed $0<x<1$, we have
\[
\lim_{p\to +\infty} \sqrt[p]{A_p(x)}= e \left(\frac{1-x}{x}\right)^{x-1}.
\]

Then we notice that both $\sqrt[p]{A_p(\cdot)}$ and the function $x\mapsto e \left(\frac{1-x}{x}\right)^{x-1}$
are strictly unimodal. Moreover, the abscissa of the unique maximum of $\sqrt[p]{A_p(\cdot)}$ is $m_A(p)$,
while that of $x\mapsto e \left(\frac{1-x}{x}\right)^{x-1}$ is 
$\frac{1}{1+W\left(\frac{1}{e}\right)}$.

We can now formulate our main result saying that 
\[
\lim_{p\to +\infty} m_A(p)=\frac{1}{1+W\left(\frac{1}{e}\right)}.
\]
The proof of this convergence result is obtained by combining the above claims and the 
following lemma about strictly unimodal functions (the lemma below has a short proof by 
contradiction based on the 
``\,$\forall\varepsilon$ $\exists N$ $\ldots$'' definition of the limit).
\begin{lem} Suppose that a sequence of strictly unimodal functions $\varphi_n:[0,1]\to\mathbb{R}$ 
converges pointwise on $[0,1]$ to a strictly unimodal function $\varphi:[0,1]\to\mathbb{R}$.
 Let $m_{\varphi}(n)\in [0,1]$ denote the abscissa of the unique maximum of $\varphi_n$, and 
$m_{\varphi}\in [0,1]$ denote the abscissa of the unique maximum of $\varphi$. Then 
$m_{\varphi}(n) \to m_{\varphi}$ as $n\to +\infty$.
\end{lem}
\end{rem}

\subsection{Bounds on $\M^{\mathrm{EM}}_p({\mathcal{S}})$}\label{subsection43boundsonMpEM}

Throughout the section, $p\ge 2$ denotes an \textit{even} integer and we let $r=\frac{p}{2}$. 

\subsubsection{The case ${\mathcal{S}}=\Ss_p$}\label{subsecion431}

In order to estimate the 
\[
\left\vert Q_{p,m,\ell}^{\mathrm{EM}}(z)\right\vert = \frac{2 m^{2r} }{(r-m)! (r+m)!}
\,\vert q_{m,2m-\ell+1}(z)  \vert
\]
quantities for $1\le m \le r$ and $1\le \ell \le 2m$  from Lemma \ref{QEMpmlexplicitform},  let
us define for any 
$z\in \Complex$ an auxiliary sequence $\widetilde{q}_{m,j}(z)$ by $\widetilde{q}_{m,0}(z):=0$,
$\widetilde{q}_{m,1}(z):=1$, and
\[
\widetilde{q}_{m,j}(z):=\frac{|z|}{m}\widetilde{q}_{m,j-1}(z)+\widetilde{q}_{m,j-2}(z) 
\quad\quad (2\le j\le 2m).
\]
Then an inductive application of the triangle inequality shows that for each admissible 
$(m,j)$ pair we have
$|q_{m,j}(z)|\le\widetilde{q}_{m,j}(z)$. Notice that $\widetilde{q}_{m,j}(z)$ is real and non-negative, so 
by solving its defining recursion we get
\[
|q_{m,j}(z)|\le\widetilde{q}_{m,j}(z)=\frac{m}{\sqrt{\left| z\right| ^2+4 m^2}}
\left(
\left(\frac{\left| z\right|+\sqrt{\left| z\right| ^2+4 m^2}}{2m}\right)^j
-\left(\frac{\left| z\right| -\sqrt{\left| z\right| ^2+4m^2}}{2m}\right)^j
\right)\le
\]
\[
\frac{m}{2m}\left(
\left(\frac{\left| z\right|+\sqrt{\left| z\right| ^2+4 m^2}}{2m}\right)^j
+\left\vert\frac{\left| z\right| -\sqrt{\left| z\right| ^2+4m^2}}{2m}\right\vert^j
\right)\le \left(\frac{\left| z\right|+\sqrt{\left| z\right| ^2+4 m^2}}{2m}\right)^j\le
\]
\[
\left(\frac{\left| z\right|+\sqrt{\left| z\right| ^2+4 m^2}}{2m}\right)^{2m}.
\]
For $z\in \Ss_p$, Lemma \ref{extrapollemma4.3shape}
tells us that $|z|\le 3.2r$, therefore
\[
\M^{\mathrm{EM}}_p(\Ss_p)\equiv \max_{1\le m \le r}\max_{1\le \ell \le 2m}\max_{z\in \Ss_p}\left\vert
Q_{p,m,\ell}^{\mathrm{EM}}(z)\right\vert\le \]
\[
\max_{1\le m \le r} \frac{2 m^{2r} }{(r-m)! (r+m)!}
\left(\frac{3.2r+\sqrt{(3.2r) ^2+4 m^2}}{2m}\right)^{2m}.
\]
Now we proceed similarly as in Section \ref{subsection421}. The $m=r$ case is treated later.
For $1\le m\le r-1$, the lower estimate from Lemma \ref{factorialestimatelemma} is used to
eliminate the factorials, then a new variable $x:=\frac{m}{r}\in \left[0,1-\frac{1}{r}\right]$ is introduced to get
\[
\frac{2 m^{2r} }{(r-m)! (r+m)!}
\left(\frac{3.2r+\sqrt{(3.2r) ^2+4 m^2}}{2m}\right)^{2m}\le 
\]
\[
\frac{e^{2r}}{\pi}\cdot m^{2 r} \cdot (r-m)^{m-r-\frac{1}{2}}\cdot 
   (r+m)^{-m-r-\frac{1}{2}} \cdot 
\left(
\frac{3.2 r+\sqrt{10.24 r^2+4m^2}}{2 m}
\right)^{2 m}=
\]
\[
\frac{e^{2r}}{\pi r}\cdot \left( x^{2 r-2 r x}\cdot (1-x)^{r x-r-\frac{1}{2}}\cdot 
   (1+x)^{-r x-r-\frac{1}{2}}\cdot
   \left(1.6+\sqrt{2.56+x^2}\right)^{2 r x}\right)=
\]
\[
\frac{e^{2r}}{\pi r}\cdot\frac{1}{\sqrt{1-x^2}}
\left[\sqrt{x^{2-2 x} \cdot (1-x)^{x-1} \cdot  (1+x)^{-x-1}\cdot 
   \left(1.6+\sqrt{2.56+x^2}\right)^{2 x}}\right]^{2r}.
\]
It can be proved that 
the function in $[\ldots]$ above is again a strictly unimodal
function (in the sense of Remark \ref{remark41}), and
 $[\ldots]<1.834$ for $x\in [0,1)$. As for the other factor, we have that
$
\frac{1}{\sqrt{1-x^2}}\le \frac{r}{\sqrt{2 r-1}}
$
since $x\in \left[0,1-\frac{1}{r}\right]$. Combining the above, we have proved that
\[
\max_{1\le m \le r-1}\max_{1\le \ell \le 2m}\max_{z\in \Ss_p}\left\vert
Q_{p,m,\ell}^{\mathrm{EM}}(z)\right\vert\le 
\frac{e^{2r}}{\pi r}\cdot \frac{r}{\sqrt{2 r-1}} \cdot 1.834^{2r}<
\frac{4.986^{2r}}{\pi\sqrt{2 r-1}}.
\]
As a last step, let us consider the $m=r$ case. Then
\[
\frac{2 m^{2r} }{(r-m)! (r+m)!}
\left(\frac{4r+\sqrt{(4r) ^2+4 m^2}}{2m}\right)^{2m}\Bigg|_{m=r}=
\frac{2 \left(\frac{8+\sqrt{89}}{5}\right)^{2 r} r^{2 r}}{(2r)!}<
\frac{\left(\frac{e}{10} \left(8+\sqrt{89}\right)
   \right)^{2 r}}{\sqrt{\pi } \sqrt{r}}.
\]
By comparing these upper bounds we obtain the following.

\begin{thm} For any $1\le r=\frac{p}{2}\le 8$,
\[
\M^{\mathrm{EM}}_p(\Ss_p)<\sqrt{\frac{2}{\pi}}\cdot\frac{4.74^{p}}{\sqrt{p}},
\]
while for any $r\ge 9$,
\[
\M^{\mathrm{EM}}_p(\Ss_p)<\frac{4.986^{p}}{\pi\sqrt{p-1}}.
\]
\end{thm}

If instead of Lemma \ref{extrapollemma4.3shape} and $|z|\le 3.2r$, we use
Lemma \ref{exp_stab_region1+eps} and $|z|\le 2(1+\varepsilon)r$, we get our
next result.

\begin{thm}\label{3.539theorem}
For any $\varepsilon>0$
there is a $p(\varepsilon)\in\mathbb{N}^+$ such that for all even $p\ge p(\varepsilon)$
\[
\M^{\mathrm{EM}}_p(\Ss_p)< \frac{(3.539+\varepsilon)^{p}}{\pi\sqrt{p-1}}.
\]
\end{thm}

The exact values of $\M^{\mathrm{EM}}_p(\Ss_p)$ for even $2\le p\le 8$ are summarized also in
Table \ref{maxQEMpmlexactvalues}---by taking into account the equality observed in 
 (\ref{equalityoftwotables}). The computing time for $p=6$ was $0.5$ minutes, whereas the
$p=8$ case took $12$ minutes. Supposing an underlying growth law of the form
$c_1 \cdot\frac{c_2^p}{\sqrt{p}}$, we see that for the $p$ values given in this table 
\[
\M^{\mathrm{EM}}_p(\Ss_p)\approx 0.906 \cdot\frac{2.022^p}{\sqrt{p}}.
\]

Table \ref{MEMupperbounds} contains upper bounds on $\M^{\mathrm{EM}}_p(\Ss_p)$ 
(extended up to $p=20$) based on the 
upper estimate
\begin{equation}\label{MEMestimateintermsofabssquareroot}
\M^{\mathrm{EM}}_p(\Ss_p)\le 
\max_{1\le m \le r}\max_{z\in\Ss_p}\frac{2 m^{2r} }{(r-m)! (r+m)!}
\left(\frac{|z|+\sqrt{|z|^2+4 m^2}}{2m}\right)^{2m}.
\end{equation}
Then, instead of using $|z|\le 3.2r$ as earlier, we estimated $|z|$ from 
above by using Table \ref{farthestpointTaylorpoly} directly.

\begin{table}[h!]
\begin{center}
    \begin{tabular}{c|c||c|c}
    $p$ & $\M^{\mathrm{EM}}_p(\Ss_p)\le\ldots$ & $p$ & $\M^{\mathrm{EM}}_p(\Ss_p)\le\ldots$ \\ \hline
    2 &  6.69 & 12 & 19113 \\
    4 &  20.7 & 14 &  157442 \\
    6 & 85.5 & 16 &  $1.308\cdot 10^6$ \\
    8 &  439 & 18  &  $1.092\cdot 10^7$ \\
    10 &  2609 & 20 & $9.198\cdot 10^7$ \\ \hline
         \end{tabular}
\caption{Upper bounds on $\M^{\mathrm{EM}}_p(\Ss_p)$\label{MEMupperbounds}}
\end{center}
\end{table}

\subsubsection{The case ${\mathcal{S}}=\Ss_p\cap \Cleft$}

By using the same approach as in the previous subsection, but applying Lemmas
\ref{exp_stab_region95/100} and \ref{exp_stab_region1/e} instead of 
Lemmas \ref{extrapollemma4.3shape} and \ref{exp_stab_region1+eps}, respectively, 
the following two results can be proved (the $2\le p \le 10$ values in Theorem 
\ref{3.423theorem} are covered by Tables \ref{maxQEMpmlexactvalues} and \ref{MEMupperboundslefthalfplane}).

\begin{thm}\label{3.423theorem} For any $p\ge 12$, 
\[
\M^{\mathrm{EM}}_p(\Ss_p\cap \Complex_{-,0})<\frac{3.423^{p}}{\pi\sqrt{p-1}}.
\]
\end{thm}

\begin{thm}\label{2.157theorem}
For any $\varepsilon>0$
there is a $p(\varepsilon)\in\mathbb{N}^+$ such that for all even $p\ge p(\varepsilon)$
\[
\M^{\mathrm{EM}}_p(\Ss_p\cap \Complex_{-,0})< \frac{(2.157+\varepsilon)^{p}}{\pi\sqrt{p-1}}.
\]
\end{thm}

The quantities $\M^{\mathrm{EM}}_p(\Ss_p\cap \Complex_{-,0})$ have also been determined
exactly for even $2\le p\le 8$, and at least for these $p$ values we have found that
\begin{equation}\label{equalityoftwotables}
\M^{\mathrm{EM}}_p(\Ss_p\cap \Complex_{-,0})=\M^{\mathrm{EM}}_p(\Ss_p),
\end{equation}
see Table \ref{maxQEMpmlexactvalues}.

Finally, Table \ref{MEMupperboundslefthalfplane} (extending Table \ref{maxQEMpmlexactvalues} to 
larger $p$ values) is based on 
estimate (\ref{MEMestimateintermsofabssquareroot}) with 
$z\in\Ss_p\cap \Complex_{-,0}$ instead of $z\in\Ss_p$, then the maximal $|z|$ values
from Table \ref{farthestpointTaylorpolylefthalfplane} have been used.

\begin{table}[h!]
\begin{center}
    \begin{tabular}{c|c||c|c}
    $p$ & $\M^{\mathrm{EM}}_p(\Ss_p\cap \Complex_{-,0})$ & $p$ & $\M^{\mathrm{EM}}_p(\Ss_p\cap \Complex_{-,0})$ \\ \hline
    2 & $\sqrt{2 \left(1+\sqrt{2}\right)}\approx 2.198$  & 6 & 25.378 \\
    4 & $7.332$ & 8 &  88.755\\  \hline   
       \end{tabular}
\caption{Maximum internal amplification factors for the first few explicit midpoint extrapolation methods 
of order $p$ with respect to the absolute stability region of the method in the left half of the
complex plane. For $p\ge 4$, the exact maximum values are rounded up. 
The algebraic degrees of the entries corresponding to $p=2, 4, 6, 8$ 
and returned by \textit{Mathematica} are $4, 38, 5, 7$,
respectively. \label{maxQEMpmlexactvalues}}
\end{center}
\end{table}

\begin{table}[h!]
\begin{center}
    \begin{tabular}{c|c||c|c}
    $p$ & $\M^{\mathrm{EM}}_p(\Ss_p\cap \Complex_{-,0})\le\ldots$ & $p$ & $\M^{\mathrm{EM}}_p(\Ss_p\cap \Complex_{-,0})\le\ldots$ \\ \hline
     10 & 836 & 16 &  57225 \\
    12 & 3108  & 18  &  242706 \\
    14 & 14491  & 20 & $1.110\cdot 10^6$ \\ \hline
         \end{tabular}
\caption{Upper bounds on $\M^{\mathrm{EM}}_p(\Ss_p\cap \Complex_{-,0})$
\label{MEMupperboundslefthalfplane}}
\end{center}
\end{table}

\subsubsection{The case ${\mathcal{S}}=\{0\}$}

Table \ref{maxQEMpmlattheorigin} lists 
the corresponding 
\[
\M^{\mathrm{EM}}_p(\{0\})\equiv \max_{1\le m \le r}\max_{1\le \ell \le 2m}
\frac{2m^{2r} }{(r-m)! (r+m)!}
\vert q_{m,2m-\ell+1}(0)\vert
\]
values for the EM extrapolation (with $r=\frac{p}{2}$). Notice that $q_{m,j}(0)=0$ for $j$ even, and 
$q_{m,j}(0)=1$ for $j$ odd, so
\[
\M^{\mathrm{EM}}_p(\{0\})= \max_{1\le m \le r}\frac{2m^{2r} }{(r-m)! (r+m)!}.
\]

\begin{table}[h!]
\begin{center}
    \begin{tabular}{c|c||c|c}
    $p$ & $\M^{\mathrm{EM}}_p(\{0\})$ & $p$ & $\M^{\mathrm{EM}}_p(\{0\})$ \\ \hline
    2 & $1$  & 12 & $12.3$  \\
    4 & $4/3\approx 1.34$ & 14 &  $25.2$ \\
    6 & $81/40=2.025$ & 16 &  $50.9$ \\
    8 & $1024/315\approx 3.26$ & 18 &   $101.3$ \\
    10 & $16384/2835\approx 5.78$  & 20 &  $199.9$ \\ \hline
           \end{tabular}
\caption{For $p\ge 12$, the exact (rational) values are rounded up. 
\label{maxQEMpmlattheorigin}}
\end{center}
\end{table}

We remark that if $x^{***}\approx 0.8336$ denotes the unique root in $(0,1)$ of the equation 
\[
\ln \left(\frac{1+x}{1-x}\right)=\frac{2}{x},
\]
then $\M^{\mathrm{EM}}_p(\{0\})$ is asymptotically equal to 
\[
f_{6,p}(x^{***})\approx
1.1524\cdot \frac{1.509^p}{p} ,
\]
where
\[
f_{6,p}(x):=
\frac{2  \left(ex\, (1-x)^{\frac{1}{2}(x-1)} \, (1+x)^{-\frac{1}{2} (1+x)}\right)^p}{\pi  p \sqrt{1-x^2}}.
\]

\section{Conclusion}
Roundoff is not usually a significant source of error in traditional
Runge--Kutta methods with typical error tolerances and double-precision arithmetic.
However, when very high accuracy is required, it is natural to use very high order
methods, which necessitates the use of large numbers of stages and can lead to
substantial amplification of roundoff errors.  This amplification is problematic
precisely when high precision is desired. In fact, traditional error bounds and estimates
that neglect roundoff error become useless in this situation.

In the past, internal error
amplification has been a practical issue only in rare cases.  However, current trends
toward the use of higher-order methods and higher-accuracy computations suggest
that it may be a more common concern in the future.
The analysis and bounds given in this work can be used to accurately estimate
at what tolerance roundoff will become important in a given computation.
More generally, the {\em maximum internal amplification factor} that we
have defined provides a single useful metric for deciding whether internal
stability may be a concern in a given computation.

We have emphasized that internal amplification depends on implementation,
and that the choice of implementation can be used to modify the internal
stability polynomials.  However, it is not yet clear whether dramatic improvements
in internal stability can be achieved in this manner for methods of interest.

For implicit Runge--Kutta methods, internal stability may become important 
even when the number of stages is moderate.  The numerical solution of the 
stage equations is usually performed iteratively and stopped when the 
stage errors are estimated to be ``sufficiently small''.  If the amplification factor is large,
then the one-step error may be also large even if the stage errors (and truncation
error) are driven iteratively to small values.  A study of the amplification of solver errors for 
practical implicit methods constitutes interesting future work.

\section*{Acknowledgment}
We are indebted to Yiannis Hadjimichael for input into the bounds
on the internal stability polynomials for third order SSP methods.
We thank Umair bin Waheed for
running computations that led to the example in Section \ref{sec:example}.


\bibliographystyle{plain}
\bibliography{internal-errors-paper}

\end{document}